\documentclass[acmsmall]{acmart}
\def\BibTeX{{\rm B\kern-.05em{\sc i\kern-.025em b}\kern-.08emT\kern-.1667em\lower.7ex\hbox{E}\kern-.125emX}}

\copyrightyear{}

\setcopyright{acmcopyright}
\acmJournal{TOMACS}
\numberwithin{equation}{section}

\newtheorem{theorem}{Theorem}
\numberwithin{theorem}{subsection} % important bit
\newtheorem{lemma}[theorem]{Lemma}

\newtheorem{prop}[theorem]{Proposition}
\newtheorem{coro}[theorem]{Corollary}
\newtheorem{example}[theorem]{Example}

\theoremstyle{definition}
\newtheorem{defn}[theorem]{Definition}

\newtheorem{assump}{}

\theoremstyle{remark}
\newtheorem*{remark}{Remark}

\newcommand{\lemref}[1]{Lemma~\ref{#1}}

\newcommand{\abs}[1]{\lvert#1\rvert}

\usepackage{eucal}

\usepackage{amsmath}
\usepackage{algorithm}
\usepackage[noend]{algpseudocode}
\usepackage{subfigure}
\usepackage{graphicx}
\usepackage{color}

\title[Parametric SO via DRO]{Parametric Scenario Optimization under Limited Data: A Distributionally Robust Optimization View\\}
\author{Henry Lam}
\authornotemark[1]
\email{khl2114@columbia.edu}
\affiliation{%
  \institution{Department of Industrial Engineering and Operations Research, Columbia University}
  \streetaddress{500 W. 120th Street}
  \city{New York}
  \state{NY}
  \postcode{10027}
}

\author{Fengpei Li}
\authornotemark[1]
\email{fl2412@columbia.edu}
\affiliation{%
\institution{Department of Industrial Engineering and Operations Research, Columbia University}
\streetaddress{500 W. 120th Street}
\city{New York}
\state{NY}
\postcode{10027}
}

\renewcommand{\shortauthors}{Lam and Li}

\begin{abstract}
We consider optimization problems with uncertain constraints that need to be satisfied  probabilistically. When data are available, a common method to obtain feasible solutions for such problems is to impose sampled constraints, following the so-called scenario optimization approach. However, when the data size is small, the sampled constraints may not statistically support a feasibility guarantee on the obtained solution. This paper studies how to leverage parametric information and the power of Monte Carlo simulation to obtain feasible solutions for small-data situations. Our approach makes use of a distributionally robust optimization (DRO) formulation that translates the data size requirement into a Monte Carlo sample size requirement drawn from what we call a generating distribution. We show that, while the optimal choice of this generating distribution is the one eliciting the data or the baseline distribution in a nonparametric divergence-based DRO, it is not necessarily so in the parametric case. Correspondingly, we develop procedures to obtain generating distributions that improve upon these basic choices. We support our findings with several numerical examples.

\end{abstract}

\ccsdesc[500]{Mathematics of computing~Probability and statistics}
\ccsdesc[500]{Computing methodologies~Modeling and simulation}
\ccsdesc[300]{Simulation types and techniques~Uncertainty quantification}
\ccsdesc[300]{Mathematical analysis~Mathematical optimization}

\keywords{chance constraint, distributionally robust optimization, scenario optimization, parametric uncertainty}

\begin{document}

 \maketitle

\section{Introduction}

We consider optimization problems in the form
\begin{equation}\label{main_problem}
\begin{aligned}
 &\underset{x \in \mathcal{X} \subseteq \mathbb{R}^d}{\text{min}}
& & c^Tx ,\\
&\text{\quad s.t.}
& & \mathbb P(x\in \mathcal{X}_{\xi})\geq1-\epsilon,
\end{aligned}
\end{equation}
where $\mathbb P$ is a probability measure governing the random variable $\xi$ on some space $\mathcal Y$ and $\mathcal X_\xi\subseteq\mathcal X\subseteq\mathbb{R}^d$ is a set depending on $\xi$. Problem \eqref{main_problem} enforces a solution $x$ to satisfy $x\in\mathcal X_{\xi}$ with high probability, namely at least $1-\epsilon$. This problem is often known as a probabilistically constrained or chance-constrained program (CCP) \cite{prekopa2003probabilistic}. It provides a natural framework for decision-making under stochastic resource capacity or risk tolerance, and has been applied in various domains such as production planning \cite{murr2000solution}, inventory management \cite{lejeune2007efficient}, reservoir design \cite{prekopa1978flood,prekopa1978serially}, communications \cite{shi2015optimal}, and ranking and selection \cite{hong2015chance}.

We focus on the situations where $\mathbb P$ is unknown, but some i.i.d. data, say $\xi_1,\ldots,\xi_n$, are available. One common approach to handle \eqref{main_problem} in these situations is to use the so-called scenario optimization (SO) or constraint sampling \cite{campi2011sampling,nemirovski2006scenario}. This replaces the unknown constraint in \eqref{main_problem} with $x\in\mathcal X_{\xi_i},i=1,\ldots,n$, namely, by considering
\begin{equation}\label{SG_problem}
\begin{aligned}
& \underset{x \in \mathcal{X} \subseteq \mathbb{R}^d}{\text{min}}
& & c^Tx, \\
& \text{\quad s.t.}
& & x\in \mathcal{X}_{\xi_i},\ i=1,\ldots,n .
\end{aligned}
\end{equation}

Note that CCP \eqref{main_problem} is generally difficult to solve even when the set $\mathcal X_{\xi}$ is convex for any given $\xi$ and the distribution $\mathbb P$ is known \cite{prekopa2003probabilistic}.  Thus, the sampled problem \eqref{SG_problem} offers a tractable approximation for the difficult CCP even in non-data-driven situations, assuming the capability to generate these samples.
% Nonetheless, we focus on the data-driven case in this paper.
% Nonetheless, our premise in this paper is under the availability of i.i.d. data.%  (which we call safety condition for convenience and also adopting existing terminology)

Our goal is to find a good feasible solution for \eqref{main_problem} by solving \eqref{SG_problem} under the availability of i.i.d. data described above. Intuitively, as the sample size $n$ increases, the number of constraints in \eqref{SG_problem} increases and one expects them to sufficiently populate the safety set $\{\xi:x\in\mathcal X_\xi\}$, thus ultimately give rise to a feasible solution for \eqref{main_problem}. To make this more precise, we first mention that because of the statistical noise from the data, one must settle for finding a solution that is feasible with a high confidence. More specifically, define, for any given solution $x$, $$V(x,\mathbb{P})=\mathbb{P}(x\notin \mathcal{X}_{\xi})$$to be the violation probability of $x$ under probability measure $\mathbb P$ that generates $\xi$. Obviously, $x$ is feasible for \eqref{main_problem} if and only if
\begin{equation}
V(x,\mathbb P) \leq \epsilon.
\end{equation}
We would like to obtain a solution, say $\hat x$, from the data such that
\begin{equation}
\mathbb P_{data}(V(\hat x,\mathbb P) \leq \epsilon) \geq 1-\alpha,\label{confidence guarantee}
\end{equation}
where $\mathbb P_{data}$ is the distribution that generates the i.i.d. data $\xi_i,i=1,\ldots,n$ (each sampled from $\mathbb P$), and $1-\alpha$ is a given confidence level (e.g., $\alpha=5\%$). In other words, we want $\hat x$ to satisfy the chance constraint in \eqref{main_problem} with the prescribed confidence.

Under the convexity of $\mathcal X_\xi$ and mild additional assumptions (namely, that every instance of \eqref{SG_problem} has a feasible region with nonempty interior and a unique optimal solution), the seminal work \cite{campi2008exact} provides a tight estimate on the required data size $n$ to guarantee \eqref{confidence guarantee}. They show that a solution $\hat x$ obtained by solving \eqref{SG_problem} satisfies
\begin{equation}
\mathbb P_{data}(V(\hat x,\mathbb P)>\epsilon)\leq\sum_{i=0}^{d-1}\binom{n}{i}\epsilon^i(1-\epsilon)^{n-i},\label{CCP guarantee}
\end{equation}
with equality held for the class of ``fully-supported" optimization problems \cite{campi2008exact}. Thus, suppose we have a sample size $n$ large enough such that
\begin{equation}\label{gamma_function}
% \gamma(n,\epsilon,d)=
B(\epsilon,d,n)=\sum_{i=0}^{d-1}\binom{n}{i}\epsilon^i(1-\epsilon)^{n-i}\leq\alpha,
\end{equation}
then from \eqref{CCP guarantee} we have $\mathbb P_{data}(V(\hat x,\mathbb P)>\epsilon)\leq\alpha$ or \eqref{confidence guarantee}.
% would guarantee $\epsilon$-feasibility with confidence at least $1-\alpha$:
% \begin{equation}
% \mathbb P_{data}(V(\hat x,\mathbb P) \leq \epsilon) \geq 1-\alpha,
% \end{equation}
% as long as $\gamma(n,\epsilon,d) \leq \alpha$.

% As $n\to\infty$, it can be shown $\gamma(n,\epsilon,d)$ converges to 0 at an exponential rate which matches our expectations since as $n$ grows, samples $\xi_1,...,\xi_n$ would sufficiently populate the support of $\xi$ leading to $V(\hat x(\xi_1,...,\xi_n),\mathbb P) \leq \epsilon$ with confidence converging to 1. Thus, for a given confidence level $\alpha$, if the data size $n$ is large enough where $\gamma(n,\epsilon,d) \leq \alpha$, we can solve \eqref{SG_problem} and obtain a feasible solution for \eqref{main_problem} with at least $1-\alpha$ confidence.

% . In particular, $\mathbb P_{data}$ denotes the measure generating the data $\xi_i,i=1,\ldots,n$ ( hence also $\hat x$, obtained from random data) $\gamma(n,\epsilon,d)\leq\alpha$ for the theorem in \cite{campi2008exact} to hold
However, in small-sample situations in which the data size $n$ is not large enough to support \eqref{gamma_function}, the feasibility guarantee described above may not hold. It can be shown \cite{campi2008exact} that the minimum $n$ that achieves \eqref{confidence guarantee} is linear in $d$ and reciprocal in $\epsilon$, thus may impose challenges especially in high-dimensional and low-tolerance problems. Similar dependence on the key problem parameters also appears in other related methods such as \cite{de2004constraint}, which uses the Vapnik-Chervonenkis dimension to infer required sample sizes, the sampling-and-discarding approach in \cite{campi2011sampling}, and the closely related approach using sample average approximation in \cite{luedtke2008sample}. Several recent lines of techniques have been suggested to overcome these challenges and reduce sample size requirements, including the use of support rank and solution-dependent support constraints  \cite{schildbach2013randomized,campi2018wait}, regularization \cite{campi2013random}, and sequential approaches \cite{care2014fast,calafiore2011research,chamanbaz2016sequential, calafiore2017repetitive}.

In this paper, we offer a different path to alleviate the data size requirement than the above methods, when $\mathbb P$ possesses known parametric structures. Namely, we assume $\mathbb P \in \{\mathbb P_{\theta}\}_{\theta\in\Theta}$ for some parametric family of distribution, where $\mathbb P_\theta$ satisfies two basic requirements: It is estimatable, i.e., the unknown quantity or parameter $\theta$ can be estimated from data, and simulatable, i.e., given $\theta$, samples from $\mathbb P_{\theta}$ can be drawn using Monte Carlo methods. Under these presumptions, our approach turns the CCP \eqref{main_problem}, with an unknown parameter, into a CCP that has a definite parameter and a suitably re-adjusted tolerance level, which then allows us to generate enough Monte Carlo samples and consequently utilize the guarantee provided from \eqref{CCP guarantee}. On a high level, this approach replaces the data size requirement in using \eqref{SG_problem} (or, in fact, any of its variant methods) with a Monte Carlo size requirement, the latter potentially more available given cheap modern computational power. Our methodological contributions consist of the development of procedures, related statistical results on their sample size requirement translations, and also showing some key differences between parametric and nonparametric regimes.

Our approach starts with a distributionally robust optimization (DRO) to incorporate the data-driven parametric uncertainty. The latter is a framework for decision-making under modeling uncertainty on the underlying probability distributions in stochastic problems. It advocates the search for decisions over the worst case, among all distributions contained in a so-called uncertainty set or ambiguity set (e.g., \cite{wiesemann2014distributionally,delage2010distributionally,goh2010distributionally}). In CCP, this entails a worst-case chance constraint over this set (e.g., \cite{hanasusanto2015distributionally, zymler2013distributionally, hanasusanto2017ambiguous,li2016ambiguous,jiang2016data,zhang2016ambiguous,hu2013kullback,cheng2014distributionally,xie2018deterministic,chen2010cvar,chen2018data,ji2018data}). When the uncertainty set covers the true distribution with a high confidence (i.e., the set is a confidence region), then feasibility for the distributionally robust CCP converts into a confidence guarantee on the feasibility for the original CCP. We follow this viewpoint and utilize uncertainty sets in the form of a neighborhood ball surrounding a baseline distribution, where the ball size is measured by a statistical distance (e.g., \cite{ben2013robust,petersen2000minimax,hansen2008robustness,love2015phi,dupuis2016path,lam2017empirical,lam2016recovering,hu2013kullback,gotoh2018robust,duchi2016statistics,esfahani2015data,blanchet2016quantifying,blanchet2016robust,gao2016distributionally}). In the parametric case, a suitable choice of this distance (such as the $\phi$-divergence that we focus on) allows easy and meaningful calibration of the ball size from the data, so that the resulting DRO provides a provable feasibility conversion to the CCP.

% an observation that we utilize to provide meaningful guarantees for our developed procedure.

% Here, we remark the connection of our ``Monte Carlo-based SG" with the Data-driven CCP \cite{jiang2016data} and the robust Monte Carlo considered in \cite{hu2012robust,glasserman2014robust} using a similarly well-motivated change-of-measure type argument.

Our next step is to combine this DRO with Monte Carlo sampling and scenario approximation. The definition of DRO means that there are many possible candidate distributions that can govern the truth, whereas the statistical guarantee for SO assumes a specific distribution that generates the data or Monte Carlo samples. To resolve this discrepancy, we select a \emph{generating distribution} that draws the Monte Carlo samples, and develop a translation of the guarantee from a fixed distribution into one on the DRO.
% We highlight that in doing so, our approach translates the generality in applying SO to approximate unambiguous CCP to our DRO.
We highlight the benefits in using SO to handle this DRO, as opposed to other potential methods. While there exist many good results on tractable reformulations of DRO for chance constraints (e.g., \cite{hanasusanto2015distributionally, hanasusanto2017ambiguous,ji2018data,li2016ambiguous,xie2018deterministic}), the reformulation tightness typically relies on using moment-based uncertainty sets and particular forms of the safety condition. Compared to moments, divergence-based uncertainty sets can be calibrated with data to consistently shrink to the true distribution. Importantly, in the parametric case, the calibration of divergence-based sets is especially convenient, and achieves a tight convergence rate by using maximum likelihood theory that efficiently captures parametric information. Our condition for applying SO to this DRO is at the same level of generality as applying SO to an unambiguous CCP, which, as mentioned before, only requires the convexity of $\mathcal X_\xi$ and mild conditions.
% the fully-supported assumption when using the standard SO procedure.

To exploit the full capability of our approach, we investigate the optimal choice of the generating distribution in relation to the target DRO, in the sense of requiring the least Monte Carlo size. We show that, if there is no ambiguity on the distribution (i.e., a standard CCP), or when the uncertainty set of a DRO is constructed via a divergence ball in the nonparametric space, the best generating distribution is, in a certain sense, the true or the baseline distribution at the center of the ball. However, if there is parametric information, the optimal choice of the generating distribution can deviate from the baseline distribution in a divergence-based DRO. We derive these results by casting the problem of selecting a generating distribution into a hypothesis testing problem, which connects the sampling efficiency of the generating distribution with the power of the test and the Neyman-Pearson lemma \cite{lehmann2006testing}. The results on DRO in particular combine this Neyman-Pearson machinery with the established DRO reformulation of chance constraints in \cite{jiang2016data,hu2013kullback}, with the discrepancy between the best generating distribution and the baseline distribution in the parametric case stemming from the removal of the extremal distributions in the corresponding nonparametric uncertainty set. These connections among hypothesis testing, SO and DRO are, to our best knowledge, the first of its kind in the literature.

Finally, given the non-optimality of the baseline distribution of a divergence-based DRO in generating Monte Carlo samples, we further develop procedures to search over generating distributions that improve upon this baseline. On a high level, this can be achieved by increasing the sampling variability to incorporate the uncertainty of the distributional parameters (one may intuit this from the perspective of a posterior distribution in a Bayesian framework), which is implemented by utilizing suitable mixture distributions. We provide several classes of mixture distributions to attain such a variability enlargement, and study descent-type algorithms to search for good distributions in these classes. In the experiments, we show our methods can be combined with SO or other SO-based methods including FAST \cite{care2014fast} to solve
a variety of optimization problems and data distributions, some of which are not amenable
to RO, especially when the objective function is non-linear or the feasible sets are jointly chance-constrained. Furthermore, we also demonstrate how to search for more judicious choices of generating distributions that
can significantly reduce the required number of Monte Carlo samples.

We conclude this introduction by briefly discussing a few other lines of related literature. The first is the so-called robust Monte Carlo or robust simulation that, like us, also considers using Monte Carlo sampling together with DRO \cite{hu2012robust,glasserman2014robust,lam2016robust,huhong2015,lam2011sensitivity,ghosh2015computing}. However, this literature focuses on approximating DRO with stochasticity in the objective function, and does not study the chance constraint feasibility and SO that constitute our main focus. We also contrast our work with \cite{lam2016recovering} that also considers likelihood theory and utilizes simulation in tackling uncertain constraints. The study \cite{lam2016recovering} focuses on the nonparametric regime and uses the empirical likelihood to construct uncertainty sets. Unlike our work, there is no parametric information there that can be leveraged to overcome sample size requirements in SO. Moreover, the simulation used in \cite{lam2016recovering} is for calibrating the uncertainty set, instead of drawing sample constraints. Next, \cite{erdougan2006ambiguous} considers a scenario approach to distributionally robust CCP with an uncertainty set based on the Prohorov distance. Like \cite{de2004constraint}, \cite{erdougan2006ambiguous}  utilizes the  Vapnik-Chervonenkis dimension in studying feasibility, in contrast to the convexity-based argument in \cite{campi2008exact} that we utilize. More importantly, we aim to optimize the efficiency of Monte Carlo sampling in handling limited-data CCP, thus motivating us to study the choice of distance, calibration schemes, and selection of generating distributions that are different from \cite{erdougan2006ambiguous}. Finally, a preliminary conference version of this work has appeared in \cite{lam2018sampling}, which contains a basic introduction of our framework, without detailed investigation of the optimality of generating distributions, improvement strategies, and extensive numerical demonstrations.

To summarize, our main contributions of this paper are:
\begin{enumerate}
\item We propose a framework to obtain good feasible solutions in data-driven CCPs in small-sample situations, where the data size is insufficient to support the use of SO with valid statistical guarantees. Focusing on the parametric regime, our framework operates by setting up a DRO, with an uncertainty set constructed from parameter estimates using the data, that can in turn be tackled by using SO with Monte Carlo samples. In doing so, our framework effectively leverages the parametric information to convert the SO requirement on the data size into a requirement on the Monte Carlo size, the latter can be much more abundant given cheap modern computational power. The overview of this framework and the DRO construction are in Sections \ref{sec:DRO} and \ref{sec:set}.
\item We investigate and present the Monte Carlo size requirements needed to give statistically feasible solutions to the divergence-based DRO used in our framework. This relies on developing an implementable mechanism to connect the sample size requirement for SO, which attempts to solve a CCP with a fixed underlying distribution, to the sample size requirement needed to solve a DRO, by selecting a suitable generating distribution to draw the Monte Carlo samples. This contribution is presented in Section \ref{sec:sampling}.
% \item  that are used to approximate the DRO.
% \item function using a notion of ``bounding function"
\item We study the optimality of generating distributions, in a sense of minimizing the Monte Carlo effort that we will describe precisely. In particular, we show that the optimal generating distributions for an unambiguous CCP, and for a distributionally robust CCP with nonparametric divergence-based uncertainty sets, are simply their respective natural choices, namely the original underlying distribution and the baseline distribution (i.e., center of the divergence ball). In contrast, the optimal generating distribution for a distributionally robust CCP in the parametric case is more delicate, and the baseline distribution there can be readily dominated by other generating distributions. These results are derived by bridging the Neyman-Pearson lemma in statistical hypothesis testing with SO and DRO, which appears to be the first of its kind in the literature as far as we know. This contribution is presented in Section \ref{sec:stat}.
\item Motivated by the non-optimality of the baseline distribution, we propose several approaches to construct generating distributions that dominate the baseline distributions for parametric DRO, by using mixture schemes that, on a high level, enlarge the variability of the generating distributions. We show how to use descent-type search procedures to construct these distributions. This contribution is presented in Section \ref{gen}.
\end{enumerate}

Lastly, we also present in full detail our implementation algorithms in Section \ref{fastsec}, numerically demonstrate our approach and compare with other methods in Section \ref{sec:numerics}, and conclude in Section \ref{sec:conclusion}.

\section{From Data-Driven DRO to Scenario Optimization}\label{sec:outline}
This section introduces our overall framework. Recall our goal as to find a good feasible solution $\hat x$ for \eqref{main_problem}, and suppose that we have an i.i.d. data size $n$  possibly less than the requirement shown in \eqref{gamma_function}. As discussed in the introduction, we first formulate a DRO that incorporates the parametric estimation noise and subsequently allows us to resort to Monte Carlo sampling to obtain a feasible solution for \eqref{main_problem}. In the following, Section \ref{sec:DRO} first describes the basic guarantees from DRO. Section \ref{sec:sampling} investigates Monte Carlo sampling that provides guarantees on DRO. Section \ref{sec:set} discusses the choice of the uncertainty set.

\subsection{Overview of Data-Driven DRO}\label{sec:DRO}
% In this section, we give a brief introduction of the key points in our method. First recall that given problem \eqref{main_problem}, we are interested in finding a solution $\hat x$ such that \eqref{CCP guarantee} holds, in the case that $\mathbb P$ is observable only through data.  with $n < N_{exact}$ for a given confidence level $1-\alpha$ (i.e. insufficient sample). Then,

For concreteness, suppose the unknown true distribution $\mathbb P\in\mathcal P$, the class of possible probability distributions for $\xi$ (to be specified later). Given the observed data $\xi_1,...,\xi_n$, the basic steps in our data-driven DRO are:

\begin{itemize}
	\item Step 1: Find a data-driven uncertainty set $\mathcal U_{data}=\mathcal U_{data}(\xi_1,\ldots,\xi_n) \subseteq \mathcal P$ such that
	\begin{equation}\label{alpha}
	\mathbb{P}_{data}(\mathbb P \in \mathcal{U}_{data}) \geq 1-\alpha,
	\end{equation}
	where $\mathbb P_{data}$ denotes the measure generating the data $\xi_i,i=1,\ldots,n$.
	\item Step 2: Given $\mathcal U_{data}$, set up the distributionally robust CCP:
	\begin{equation}\label{DRO_problem}
	\begin{aligned}
	& \underset{x \in \mathcal{X} \subseteq \mathbb{R}^d}{\text{min}}
	& & c^Tx, \\
	& \text{\quad s.t.}
	& & \min_{\mathbb Q\in\mathcal U_{data}}\mathbb Q(x\in \mathcal{X}_{\xi})\geq1-\epsilon,
	\end{aligned}
	\end{equation}
	where the probability measure $\mathbb Q$ is the decision variable in the minimization in the constraint.
	
	\item Step 3: Find a solution $\hat x$ feasible for \eqref{DRO_problem}.
\end{itemize}

It is straightforward to see that $\hat x$ obtained from the above procedure is feasible for \eqref{main_problem} with confidence at least $1-\alpha$: If $\mathbb P\in\mathcal U_{data}$, then any $\hat x$ feasible for \eqref{DRO_problem} satisfies
\begin{equation*}
\mathbb P(\hat x\in\mathcal X_{\xi})\geq\min_{\mathbb Q\in\mathcal U_{data}}\mathbb Q(\hat x\in \mathcal{X}_{\xi})\geq 1-\epsilon
\end{equation*}
Thus
% \implies V(\hat x,\mathbb P) \leq \epsilon,
% \end{equation}
%  which, combined with \eqref{alpha}, further satisfies
 \begin{equation}
 \mathbb P_{data}(\mathbb P(\hat x\in\mathcal X_{\xi})\geq 1-\epsilon)\geq\mathbb P_{data}(\mathbb P\in\mathcal U_{data})\geq 1-\alpha,
 \end{equation}
which gives our conclusion.%  the desired result.

\subsection{Monte Carlo Sampling for DRO}\label{sec:sampling}
To use the above procedure, we need to provide a way to construct the depicted $\mathcal U_{data}$ and to find a (confidently) feasible solution for \eqref{DRO_problem}. We postpone the set construction to the next subsection and focus on finding a feasible solution here. We resort to SO, via Monte Carlo sampling, to handle \eqref{DRO_problem}. Note that, unlike in the standard SO discussed in the introduction, the distribution $\mathbb Q$ here can be any candidate within the set $\mathcal U_{data}$. Thus, let us select a generating distribution, called $\mathbb P_0$ (which can depend on the data), to generate Monte Carlo samples $\xi_i^{MC},i=1,\ldots,N$, and solve
% can be any of the p Monte Carlo After the introduction of basic framework in three steps, we are ready to introduce our Monte Carlo based Scenario Generation. The Monte Carlo based SG aims to address the challenge in Step 1, 2 and 3. In particular, for \eqref{DRO_problem}, we first consider solving a SG with Monte Carlo samples fetched from some \underset{\xi^{MC} \sim \mathbb P_0}
 \begin{equation}\label{Monte_SG}
 \begin{aligned}
 & \underset{x \in \mathcal{X} \subseteq \mathbb{R}^d}{\text{min}}
 & & c^Tx, \\
 & {\text{s.t.}}
 & & x\in \mathcal{X}_{\xi_i^{MC}},\ i=1,\ldots,N .
 \end{aligned}
 \end{equation}
For convenience, denote, for any $\epsilon,\beta>0$,
\begin{equation}\label{gamma_function1}
% \gamma(n,\epsilon,d)=
N_{exact}(\epsilon,\beta,d)=\min\left\{n:\sum_{i=0}^{d-1}\binom{n}{i}\epsilon^i(1-\epsilon)^{n-i}\leq\beta\right\}.
\end{equation}
From the result of \cite{campi2008exact} discussed in the introduction, using $N_{exact}(\epsilon,\beta,d)$ or more Monte Carlo samples from $\mathbb P_0$ in \eqref{Monte_SG} would give a solution $\hat x^{MC}$ that satisfies $V(\hat x^{MC},\mathbb P_0) \leq \epsilon$ with confidence level $1-\beta$. This is not exactly the distributionally robust feasibility statement for problem \eqref{DRO_problem}. To address this discrepancy, we consider, conditional on the data $\xi_1,\ldots,\xi_n$,
\begin{equation}\label{bv}
\begin{aligned}
& \underset{\mathbb Q\in\mathcal U_{data}}{\text{max}}
& & V(\hat x^{MC},\mathbb Q) \\
& {\text{\quad s.t.}}
& & V(\hat x^{MC},\mathbb P_0) \leq \delta.
\end{aligned}
\end{equation}
This optimization problem serves to translate a guarantee on the violation probability under $\mathbb P_0$ to any $\mathbb Q$ in $\mathcal U_{data}$. If we can bound the optimal value in \eqref{bv}, then we can trace back the level of $\delta$ that is required to ensure a chance constraint validity of tolerance level $\epsilon$. However, the event involved in defining $V(\hat x^{MC},\mathbb P_0)$ and $V(\hat x^{MC},\mathbb Q)$, namely $\{ \xi: \hat x^{MC}\notin \mathcal X_{\xi}\}$, can be challenging to handle in general. Thus, we relax \eqref{bv} to
% Notice the constraint does not involve $\mathbb Q$ as we have relaxed the problem by investigating, given a bound on the violation probability under $\mathbb P_0$, how to bound the violation probability under all of $\mathcal U_{data}$.
%   as it would help us translate the feasibility statement for $\mathbb P_0$ to a feasibility statement for all $\mathbb Q \in \mathcal U_{data}$. On the other hand, a deeper look into the definiton of $V(x,\mathbb P)=\mathbb P(x\notin \mathcal X_{\xi})$ suggests that the optimization problem \eqref{bv} involves measures $\mathbb P_0,\mathbb Q \in \mathcal U_{data}$ and sets $\{ \xi: x\notin \mathcal X_{\xi}\}$.
\begin{equation}\label{hypo_test}
\begin{aligned}
& \underset{\mathbb Q \in \mathcal U_{data}, A\subset \mathcal Y }{\text{max}}
& & \mathbb Q (A) \\
& {\text{\quad s.t.}}
& & \mathbb P_0(A) \leq \delta.
\end{aligned}
\end{equation}
 where the decision variables now include the set $A$ in addition to the probability measure $\mathbb Q$. Conditional on the data $\xi_1,\ldots,\xi_n$, the optimal value of optimization problem \eqref{hypo_test}, which we denote $M(\mathbb P_0,\mathcal U_{data},\delta)$, is clearly an upper bound for that of \eqref{bv}. In fact, it is also clear from \eqref{hypo_test} that $M(\mathbb P_0,\mathcal U_{data},\delta)$ is non-decreasing in $\delta>0$ and
\begin{equation}
\max\limits_{\mathbb Q\in\mathcal U_{data}}V(\hat x^{MC},\mathbb Q)\leq M(\mathbb P_0,\mathcal U_{data},V(\hat x^{MC},\mathbb P_0)),\label{basic_bound}
\end{equation}
by simply taking $A=\{\xi: \hat x^{MC} \notin \mathcal X_{\xi}  \}$ and $\delta=V(\hat x^{MC},\mathbb P_0)$ in \eqref{hypo_test}. We have the following guarantee:
\begin{theorem}\label{main_theorem}
Given $\mathbb P_0$, $\mathcal U_{data}$ and $\epsilon>0$, suppose there exists $\delta_{\epsilon}>0$ small enough such that
\begin{equation}
M(\mathbb P_0,\mathcal U_{data},\delta_\epsilon)\leq\epsilon.\label{interim beta}
\end{equation}
If we solve \eqref{Monte_SG} with $N_{exact} (\delta_{\epsilon},\beta,d)$ number of samples drawn from $\mathbb P_0$, then the obtained solution $\hat x^{MC}$ would be feasible for \eqref{DRO_problem} with confidence at least $1-\beta$. Furthermore, if
\begin{equation}\label{alpha_u}
\mathbb{P}_{data}(\mathbb P \in \mathcal{U}_{data}) \geq 1-\alpha,
\end{equation}
where $\mathbb P_{data}$ is the measure governing the real-data generation under the true distribution $\mathbb P $, then the obtained solution $\hat x^{MC}$ would be feasible for \eqref{main_problem} with confidence at least $1-\alpha-\beta$.
\end{theorem}
%  defined in \eqref{nexactdef}
\begin{proof}
	By results in \cite{campi2008exact},  we know that by solving \eqref{Monte_SG} with $N_{exact} (\delta_{\epsilon},\beta,d)$ number of samples from $\mathbb P_0$, the obtained solution $\hat x^{MC}$ would satisfy
	\begin{equation}\label{cofi}
	\mathbb P_{MC,0}(V(\hat x^{MC},\mathbb P_0)>\delta_{\epsilon})\leq \beta
	\end{equation}
	where $\mathbb P_{MC,0}$ is the measure with respect to the Monte Carlo samples drawn from $\mathbb P_0$. Moreover, based on the monotonicity property of $M(\cdot)$ and \eqref{basic_bound}, we have
	\begin{equation}
	V(\hat x^{MC},\mathbb P_0)\leq\delta_{\epsilon} \implies \max\limits_{\mathbb Q\in\mathcal U_{data}}V(\hat x^{MC},\mathbb Q)\leq M(\mathbb P_0,\mathcal U_{data},\delta_\epsilon).
	\end{equation}
Thus \eqref{interim beta} implies that
$$\mathbb P_{data}\left(\max\limits_{\mathbb Q\in\mathcal U_{data}}V(\hat x^{MC},\mathbb Q)>\epsilon\right)\leq\mathbb P_{data}(V(\hat x^{MC},\mathbb P_0)>\delta_{\epsilon})\leq\beta$$
and hence $\hat x^{MC}$ is feasible for \eqref{DRO_problem} with confidence at least $1-\beta$. Furthermore, if $\mathbb P\in\mathcal U_{data}$, then a $\hat x^{MC}$ feasible for \eqref{DRO_problem} is also feasible for \eqref{main_problem} since $\max\limits_{\mathbb Q\in\mathcal U_{data}}V(\hat x^{MC},\mathbb Q) \geq V(\hat x^{MC},\mathbb P)$ and hence
	\begin{equation}
	\max\limits_{\mathbb Q\in\mathcal U_{data}}V(\hat x^{MC},\mathbb Q) \leq \epsilon \implies V(\hat x^{MC},\mathbb P) \leq \epsilon.
	\end{equation}
	Thus, if we denote  ${\Xi}=\{\xi_1,...,\xi_n,\xi^{MC}_{1},...,\xi^{MC}_{N}\}$ to be entire sequence consisting of real data and the generated Monte Carlo samples, it then follows that
	\begin{equation}
	\{\Xi:V(\hat x^{MC},\mathbb P )> \epsilon\} \subseteq \{\Xi: \mathbb P \notin \mathcal U_{data}\} \cup \{ \Xi: V(\hat x^{MC},\mathbb P_0)> \delta_\epsilon \}.
	\end{equation}
	It now follows by \eqref{alpha_u} and \eqref{cofi} that $\hat x^{MC}$ is feasible for \eqref{main_problem} with probability at least $1-\alpha-\beta$.
	\end{proof}

Theorem \ref{main_theorem} can be cast in terms of asymptotic instead of finite-sample guarantees by following the same line of arguments. We summarize it as the following corollary.
% we can obtain the following corollary.
\begin{coro}\label{main cor}
In Theorem \ref{main_theorem}, if the condition $\mathbb{P}_{data}(\mathbb P \in \mathcal{U}_{data}) \geq 1-\alpha$ is substituted by the asymptotic condition
\begin{equation}
  \liminf_{n\to\infty}\mathbb P_{data}(\mathbb P\in\mathcal U_{data})\geq1-\alpha,
\end{equation}
then the feasibility of $\hat x^{MC}$ in the last conclusion of Theorem \ref{main_theorem} holds with confidence asymptotically tending to at least $1-\alpha-\beta$.
\end{coro}

 To summarize, in the presence of data insufficiency, if we choose $\mathcal U_{data}$ to satisfy the confidence property \eqref{alpha}, and are able to evaluate the bounding function $M(\mathbb P_0, \mathcal U_{data},\delta)$ that translates the violation probability under $\mathbb P_0$ to a worst-case violation probability over $\mathcal U_{data}$, then we can run SO with $N_{exact}(\delta_{\epsilon},\beta,d)$ Monte Carlo samples from $\mathbb P_0$ to obtain a solution for \eqref{main_problem} with confidence $1-\alpha-\beta$.

 We also note that the above scheme still holds if the $N_{exact}(\epsilon,\beta,d)$ in \eqref{gamma_function1} is replaced by the sample size requirements of other variants of SO (e.g., FAST \cite{care2014fast}) that are potentially smaller. This works as long as we stay with the same SO-based procedure in using the Monte Carlo samples. For clarity, throughout most of our exposition we will focus on the sample size requirement depicted in \eqref{gamma_function1}, but we will discuss other variants in our implementation and numerical sections.

 Finally, let us take a step back and justify why we use SO to tackle \eqref{DRO_problem}, as opposed to other potential means. Indeed, as pointed out in the introduction, there exist many good results on tractable reformulations of DRO. As will be discussed in detail in the next subsection, in the present context we will choose an uncertainty set that can leverage parametric information efficiently. Sets based on the neighborhoods of distributions measured by $\phi$-divergences are particularly attractive choices, as they can be calibrated easily (both the ball center and the size) in a way that efficiently uses parametric information. The dependence on the parameter dimension in particular is reflected in the degree of freedom in the $\chi^2$-distribution used in the calibrating the ball size, which shrinks to zero at a canonical rate as the data size increases. Other sets, such as moment-based ones, though possibly amenable to tight tractable reformulations, do not enjoy these statistical properties in the parametric context. Thus, in view of tackling $\phi$-divergence-based DRO, SO appears to be a natural choice, and we have set up a framework to utilize it under conditions at the same level of generality as required for the unambiguous counterpart. Sections \ref{sec:stat} and \ref{gen} will study this framework in further depth and enhance its efficiency. We caution, however, that the conservativeness in our proposed uncertainty set (which affects the optimality of the obtained solution) relies on the dimensionality of the distributional parameters. Our approach is expected to work well when this dimension is moderate, but not in high-dimensional problems where other approaches could be better choices.

\subsection{Constructing Uncertainty Sets}\label{sec:set}

% Before we find a baseline measure $\mathbb{P}_0$, in this section we first address the problem of calibrating a uncertainty set $\mathcal U_{data}$.  For illustration, we first
In this section we discuss the construction of the uncertainty set $\mathcal U_{data}$, using the $\phi$-divergence approach \cite{ben2013robust}.
% , with further improvement building upon this discussion in Section \ref{sec:stat}.
We assume the true distribution $\mathbb P$ of $\xi$ lies in a parametric family. We denote the true parameter as $\theta_{true}$. To highlight the parametric dependence, we call the true distribution $\mathbb P_{\theta_{true}}\in\mathcal P_{para}=\{\mathbb{P}_{\theta}\}_{\theta \in \Theta\subset\mathbb R^D}$ indexed by $\theta$, where $D$ is the dimension of parameter space. Given data $\xi_1,\xi_2,...,\xi_n$, we want to construct an uncertainty set $\mathcal U_{data}$ satisfying
\begin{equation}\label{truegu}
\lim_{n\to\infty}\mathbb P_{data}(\mathbb P_{\theta_{true}}\in\mathcal U_{data})=1-\alpha
\end{equation}
so that Corollary \ref{main cor} applies. To do so, we first estimate $\theta_{true}$ from the data. There are various approaches to do so; here we apply the common maximum likelihood estimator (MLE) $\hat\theta_n$, and set $\mathcal U_{data}$ to be
\begin{equation}
\mathcal U_{data}=\left\{\mathbb Q\in\mathcal P_{para}:d_\phi(\mathbb{P}_{\hat\theta_n},\mathbb{Q})\leq\frac{\phi''(1)\chi^2_{1-\alpha,D}}{2n}\right\},\label{uncertainty choice}
\end{equation}
where $\chi^2_{1-\alpha,D}$ is the $1-\alpha$ quantile of $\chi^2_D$, the $\chi^2$-distribution with degree of freedom $D$, and $d_\phi(\cdot,\cdot)$ is the $\phi$-divergence between two probability measures, i.e., given a convex function $\phi:\mathbb R_+\to\mathbb R_+$, with $\phi(1)=0$, a distance between two probability measures $\mathbb P_1$ and $\mathbb P_2$  defined as
\begin{align}
d_\phi(\mathbb{P}_1, \mathbb{P}_2)=&\int_{\mathcal Y} \phi\left(\frac{d\mathbb{P}_2}{d\mathbb{P}_1}\right) \mathbb{P}_1(dy),\label{chidef1}
\end{align}
assuming $\mathbb P_2$ is absolutely continuous with respect to $\mathbb P_1$ with Radon-Nikodym derivative $\frac{d\mathbb P_2}{d\mathbb P_1}$ on $\mathcal Y$. Moreover, we assume that $\phi$ is twice continuously differentiable with $\phi''(1)\neq0$, and if necessary set the continuation of $\phi$ to $\mathbb R_-$ as $\phi(x)=+\infty$ for $x<0$. In \eqref{uncertainty choice}, we call the center of the divergence ball, $\mathbb P_{\hat\theta_n}$, the baseline distribution.

To guarantee desirable asymptotic properties of our uncertainty set, we make the following assumption:

\begin{assump}\label{con_and_nor}
	Let $\theta_{true}\in \Theta$ be the true parameter and let  $\hat{\theta}_{n}$ be the MLE of $\theta_{true}$ estimated from $n$ i.i.d. data points. Then, as $n \rightarrow \infty$,  $\hat\theta_n$ satisfies consistency and asymptotic normality condition:
	\begin{equation}
	\hat{\theta}_{n} \xrightarrow{\mathbb P} \theta_{true} \quad\text{and} \quad \sqrt{n}(\hat{\theta}_{n}-\theta_{true}) \xrightarrow{\mathcal D} \mathcal{N}(0,\mathcal{I}^{-1}(\theta_{true})),
	\end{equation}
	where $\mathcal I (\theta) $ is the Fisher information for the parametric family $\mathcal P_{para}$ with well-defined inverse that is continuous in the domain $\theta\in\Theta$.
\end{assump}
Assumption \ref{con_and_nor} of MLE estimator is known to hold under various regularity conditions \cite{van2000asymptotic,lehmann2004elements}. We list a set of such conditions in Appendix \ref{sec:conditions}.

Under Assumption \ref{con_and_nor}, it can be shown \cite{pardo2005statistical,van2000asymptotic} that $\mathcal{U}_{data}$ in \eqref{uncertainty choice} satisfies the confidence guarantee \eqref{truegu}. Furthermore, since we can identify each $\mathbb P_\theta$ in $\mathcal P_{data}$ with $\theta$, we can equivalently view $\mathcal U_{data}$ as a subset of $\theta\in\Theta$, and write it as
\begin{equation}
{\mathcal{U}}_{data}\triangleq\left\{\theta \in \Theta: d_\phi(\mathbb P_{\hat\theta}, \mathbb P_\theta)\leq \frac{\phi''(1)\chi^2_{1-\alpha,D}}{2n}\right\}.\label{parametric uncertainty}
\end{equation}
For convenience, we shall use the two definitions of $\mathcal U_{data}$ interchangeably depending on the context. It is also known that the asymptotic confidence properties of \eqref{uncertainty choice} or \eqref{parametric uncertainty} are the same among different choices within the $\phi$-divergence class. These can be seen via a second order expansion of the $\phi$-divergences. Moreover, they are asymptotically equivalent to
% To express these observations more precisely, we redefine
% \begin{equation}
% {\mathcal{U}}_{data}\triangleq\left\{\theta \in \Theta: \chi^2(\mathbb P_{\hat\theta}, \mathbb P_\theta)\leq \frac{\chi^2_{1-\alpha,D}}{n}\right\},\label{parametric uncertainty}
% \end{equation}
%  However, from the second order expansion of $\phi$-divergence, we can show that $\mathcal U_{data}$ in \eqref{parametric uncertainty}, as well as all uncertainty sets $\mathcal U_{data}$  calibrated by $\phi$-divergence has an equivalent asymptotic form as:  with $\hat\theta_n$ being the MLE
\begin{equation}
\left\{\theta \in \Theta: (\theta-\hat{\theta}_{n})^T\mathcal I(\hat\theta_n)(\theta-\hat{\theta}_{n})\leq \frac{\chi^2_{1-\alpha,D}}{n}\right\},\label{parametric uncertainty1}
\end{equation}
where $\mathcal I(\hat\theta_n)$ is the estimated Fisher information, under the regularity conditions above \cite{pardo2005statistical,nielsen2014chi,van2000asymptotic}. In other words, under Assumption \ref{con_and_nor}, both \eqref{parametric uncertainty} and \eqref{parametric uncertainty1} satisfy
% Under the asymptotic form \eqref{parametric uncertainty1} as well as assumptions \ref{rexp},\ref{fnice} and \ref{con_and_nor}, it can be shown that  $\mathcal{U}_{data}$ satisfies the variant of \eqref{truegu}:
\begin{equation}
\lim_{n\to\infty}\mathbb P_{data}(\theta_{true}\in\mathcal U_{data})=1-\alpha.\label{confidence parameter}
\end{equation}
% as desired.

% \begin{remark}
	Note that the convergence rate of \eqref{truegu} or \eqref{confidence parameter} depends on the higher-order properties of the parametric model, which in turn can depend on the parameter dimension. Different from the sample size requirements in SO, this convergence rate is a consequence of MLE properties. Some details on finite-sample behaviors of MLE can be found in \cite{korostelev2011mathematical}.

    The $\mathcal U_{data}$ discussed above is a set over the parametric class of distributions (or parameter values). Considering tractability, DRO over nonparametric space could be easier to handle than parametric, which suggests a relaxation of the parametric constraint to estimate the bounding function $M$. This also raises the question of whether one can possibly contain $\mathcal U_{data}$ in a nonparametric ball with a shrunk radius and subsequently obtain a better $M$. These would be the main topics of Sections \ref{sec:stat} and \ref{gen}.
%     It is worth mentioning that even though we assumed the sample size $n$ is insufficient, the number of $n$ can still be large enough for the asymptotic approximation to take effect. Thus the consideration of asymptotic guarantee of uncertainty set is justified, especially considering the convenience  offered by the asymptotic construction. Moreover, the procedures following the asymptotic construction can be easily translated to finite sample bound on uncertainty sets \cite{korostelev2011mathematical}.
% \end{remark}

\section{Bounding Functions and Generating Distributions}\label{sec:stat}
Given the uncertainty set $\mathcal U_{data}$ in \eqref{parametric uncertainty}, we turn to the choice of the generating measure $\mathbb P_0$ and the bounding function $M(\mathbb P_0,\mathcal U_{data},\delta)$ which, as we recall, is the optimal value of optimization problem \eqref{hypo_test}. In the discussed parametric setup, the latter becomes
\begin{equation}\label{hypo_test1}
\begin{aligned}
& \underset{ \theta \in \mathcal U_{data},  A\subset \mathcal Y }{\text{max}}
& & \mathbb P_{\theta} (A) \\
& {\text{\quad s.t.}}
& & \mathbb P_0(A) \leq \delta.
\end{aligned}
\end{equation}
From Theorem \ref{main_theorem} and the fact that $M(\mathbb P_0,\mathcal U_{data},\delta)$ is non-decreasing in $\delta$, we want to choose $\mathbb P_0$ that minimizes $M(\mathbb P_0,\mathcal U_{data}, \delta)$ so that we can take the maximum $\delta_\epsilon$ and subsequently achieve overall confident feasibility with the least Monte Carlo sample size.  Note that $M(\mathbb P_0,\mathcal U_{data}, \delta)$ is a multi-input function depending on both $\mathbb P_0$ and $\delta$, and so a priori it is not clear that a uniform minimizer $\mathbb P_0$ can exist across all values of $\delta$ so that the described task is well-defined. It turns out that this is possible in some cases, which we shall investigate in detail. In the following, we discuss results along this line at three levels: The unambiguous case, namely when $\mathcal U_{data}$ in \eqref{hypo_test1} is a singleton (Section \ref{step1}), the case where $\mathcal U_{data}$ is nonparametric (Section \ref{sec:nonparametric}), and the case where $\mathcal U_{data}$ is parametric (Section \ref{sec:parametric}). The first two cases pave the way to the last one, which is  most important to our development and also motivates Section \ref{gen}. With these results in hand, we also discuss the possibility of using other statistical distances in our framework in Section \ref{diss}.
% by first setting up a connection between minimizing \eqref{hypo_test1} and hypothesis testing.
% This is because The tighter the bound $M(\cdot)$ is, the larger is the value of $\delta_{\epsilon}$ in \eqref{interim beta} and less number of Monte Carlo sample we need. The question now is how do we find such $\mathbb P_0$.

\subsection{Neyman-Pearson Connections and A Least Powerful Null Hypothesis}\label{step1}
% Thus, to understand and find an appropriate $\mathbb P_0$ for the optimization problem \eqref{hypo_test1},
We first consider, for a given $\theta_1 \in \mathcal U_{data}$, the optimization problem
\begin{equation}\label{hypo_test2}
\begin{aligned}
& \underset{ A\subset \mathcal Y }{\text{max}}
& & \mathbb P_{\theta_1} (A) \\
& {\text{s.t.}}
& & \mathbb P_0(A) \leq \delta.
\end{aligned}
\end{equation}
This problem can be viewed as choosing a most powerful decision rule in a statistical hypothesis test. More precisely, one can think of $A$ as a rejection region for a simple test with null hypothesis $\mathbb P_0$ and alternate hypothesis $\mathbb P_{\theta_1}$. Subject to a tolerance of $\delta$ Type-I error, optimization problem \eqref{hypo_test2} looks for a decision rule that maximizes the power of the test. By the Neyman-Pearson lemma \cite{lehmann2006testing}, under mild regularity conditions on the parametric family, the optimal set $ A^\star_{0,\theta_1,\delta}$ of \eqref{hypo_test2} takes the form
\begin{equation}\label{np_form}
 A^\star_{0,\theta_1,\delta}=\{ \xi\in \mathcal Y: \frac{d \mathbb P_{\theta_1}}{d\mathbb P_{0}} (\xi) > K_{0,\theta_1,\delta}^\star\},
\end{equation}
with $K_{0,\theta_1,\delta}^\star$ chosen so that $\mathbb P_0 ( A^\star_{0,\theta_1,\delta})=\delta$. Also, then, the optimal value of \eqref{hypo_test2} is $\mathbb P_{\theta_1}( A^\star_{0,\theta_1,\delta})$. Generalizing the above analysis to all $\theta\in\mathcal U_{data}$, we conclude that
\begin{equation}\label{hypo_bound}
M(\mathbb P_0,\mathcal U_{data},\delta)=\sup_{\theta\in \mathcal U_{data}} \mathbb P_{\theta} ( A^\star_{0,\theta,\delta}),
\end{equation}
is the optimal value of \eqref{hypo_test1}. These observations will be useful for deriving our subsequent results.
% Given these observations, how should we choose the baseline measure $\mathbb P_0$?

Our goal is to choose $\mathbb P_0$ to minimize \eqref{hypo_bound}. To start our analysis, let us first consider the extreme case where the uncertainty set $\mathcal U_{data}$ consists of only one point $\mathbb Q$. In this case, we look for $\mathbb P_0$ that minimizes $M(\mathbb P_0, \{\mathbb Q \}, \delta )$, the optimal value of
\begin{equation}\label{hypo_test_one}
\begin{aligned}
& \underset{ A\subset\mathcal Y }{\text{max}}
& & \mathbb Q (A) \\
& {\text{s.t.}}
& & \mathbb P_0(A) \leq \delta.
\end{aligned}
\end{equation}

That is, for a given measure $\mathbb Q$, we seek for the maximum discrepancy between $\mathbb Q$ and $\mathbb P_0$ over all  $\mathbb P_0$-measure sets that have $\delta$ or less content. This is similar to minimizing the total variation distance between $\mathbb Q$ and $\mathbb P_0$, and hints that the optimal choice of $\mathbb P_0$ is $\mathbb Q$. The following theorem, utilizing the Neyman-Pearson lemma depicted above, confirms this intuition. We remark that the assumptions of the theorem can be relaxed by using more general versions of the lemma, but the presented version suffices for most purposes and also the subsequent examples we will give.

\begin{theorem}\label{inp}
    Given a measure $\mathbb Q $ with continuous density on $\mathcal X$, among all $\mathbb P_0$ such that $\frac{d\mathbb Q}{d\mathbb P_0}$ exists and is continuous and positive almost surely, the minimum $M (\mathbb P_0,\{\mathbb Q\}, \delta)$ is obtained by choosing $\mathbb P_0=\mathbb Q$, giving $M(\mathbb P_0,\{\mathbb Q\},\delta)=\delta$.
\end{theorem}
\begin{proof}
 Under the assumptions, by the Neyman-Pearson lemma, for a fixed measure $\mathbb P_0$, the set achieving the optimal value of \eqref{hypo_test_one} takes the form $ A^\star=\{ \xi\in\mathcal Y: \frac{d \mathbb Q}{d\mathbb P_0} (\xi) > K^\star\}$ for some $K^\star\geq 0$ with $\mathbb P_0 (    A^\star)=\delta$. It then follows that \begin{equation*}
M(\mathbb P_0, \{\mathbb Q \}, \delta )-\delta=\mathbb Q( A^\star)-\mathbb P_0 ( A^\star)=\int_{\frac{d \mathbb Q}{d\mathbb P_0} (\xi) > K^\star} (\frac{d \mathbb Q}{d\mathbb P_0}-1) d\mathbb P_0(\xi).
\end{equation*}
Under the absolute continuity assumption, we define
\begin{equation*}
g(K)=\int_{\frac{d \mathbb Q}{d\mathbb P_0} (\xi) > K} (\frac{d \mathbb Q}{d\mathbb P_0}-1) d\mathbb P_0(\xi),
\end{equation*}
which can be seen to be a non-increasing function for $K\geq 1$ and a non-decreasing function for $K\leq 1$. To see this, take $K_1\geq K_2$, and we have
\begin{equation*}
g(K_2)=g(K_1)+\int_{K_1 \geq \frac{d \mathbb Q}{d\mathbb P_0} (\xi) > K_2} (\frac{d \mathbb Q}{d\mathbb P_0}-1) d\mathbb P_0(\xi).
\end{equation*}
Thus, when $K_1\geq K_2 \geq 1$, we have $g(K_2) \geq g(K_1)$ because
\begin{equation*}
    \int_{K_1 \geq \frac{d \mathbb Q}{d\mathbb P_0} (\xi) > K_2} (\frac{d \mathbb Q}{d\mathbb P_0}-1) d\mathbb P_0(\xi) \geq (K_2-1) \mathbb P_0({K_1 \geq \frac{d \mathbb Q}{d\mathbb P_0} (\xi) > K_2}) \geq 0,
\end{equation*}
while when $1 \geq K_1\geq K_2$, we have $g(K_2) \leq g(K_1)$ because
\begin{equation*}
      \int_{K_1 \geq \frac{d \mathbb Q}{d\mathbb P_0} (\xi) > K_2} (\frac{d \mathbb Q}{d\mathbb P_0}-1) d\mathbb P_0(\xi) \leq (K_1-1) \mathbb P_0({K_1 \geq \frac{d \mathbb Q}{d\mathbb P_0} (\xi) > K_2}) \leq 0.
\end{equation*}

Then, to identify the minimum of $g(K)$, we either decrease $K$ from 1 to 0 which gives
\begin{equation}
    \liminf_{K\to0} g(K) = \int (\frac{d \mathbb Q}{d\mathbb P_0}-1) d\mathbb P_0(\xi)=0,\label{interim1single}
\end{equation}
by using the dominated convergence theorem (e.g., by considering the set $\{1>d\mathbb Q/d\mathbb P_0(\xi)>K\}$) or we increase $K$ from 1 to $\infty$ which gives
\begin{equation}
    \liminf_{K\to \infty} g(K) \geq0.\label{interim2single}
\end{equation}
by Fatou's lemma. Observations \eqref{interim1single} and \eqref{interim2single} suggest that $g(K) \geq 0$ for all $K \geq 0$ and imply that $g(K^\star) \geq 0 $. Thus, we must have $ M(\mathbb P_0, \{\mathbb Q \}, \delta ) \geq \delta$. Note that this holds for any $\mathbb P_0$. Now, since choosing $\mathbb P_0=\mathbb Q$ gives $M(\mathbb Q, \{\mathbb Q \}, \delta )=\delta$, an optimal choice of $\mathbb P_0$ is $\mathbb Q$.
% \end{equation}
% where the last equality follows from the definition of $M(\cdot)$ and is obtained by taking $\mathbb P_0=\mathbb Q$.
 \end{proof}

 Theorem \ref{inp} shows that under mild regularity conditions, in terms of choosing the generating distribution $\mathbb P_0$ and minimizing $M(\mathbb P_0,\{\mathbb Q \},\delta)$, we cannot do better than simply choosing $\mathbb Q$ itself. This means that if we had known the true distribution was $\mathbb Q$, and without additional knowledge of the event of interest, the safest choice (in the minimax sense) for sampling would be $\mathbb Q$, a quite intuitive result. In the language of hypothesis testing, given the simple alternate hypothesis $\mathbb Q$, the null hypothesis $\mathbb P_0$ that provides the least power for the test, i.e., makes it most difficult to distinguish between the two hypotheses, is $\mathbb Q$.

\subsection{Nonparametric DRO}\label{sec:nonparametric}
Building on the discussion in Section \ref{step1}, we now consider the choice of generating distribution $\mathbb P_0$ to minimize the bounding function obtained from \eqref{hypo_test1}. Before so, we first discuss the nonparametric case, where the analog of \eqref{hypo_test1} is in the form:
\begin{equation}\label{non_par_dro}
\begin{aligned}
& \underset{ d_\phi(\mathbb P_{\hat\theta},
	\mathbb Q) \leq \lambda, A\subset \mathcal Y}{\text{max}}
& & \mathbb Q (A) \\
& {\qquad\qquad\text{s.t.}}
& & \mathbb P_0(A) \leq \delta.
\end{aligned}
\end{equation}
for some ball radius $\lambda>0$, where the decision variables are $\mathbb Q$ in the space of all distributions absolutely continuous with respect to $\mathbb P_{\hat\theta}$, and $A$.

We show that the above setting can be effectively reduced to the unambiguous case, i.e., when $\mathbb Q$ lies in a singleton discussed in Section \ref{step1}. This comes from an established equivalence between a distributionally robust chance constraint and an unambiguous chance constraint evaluated by the center of the divergence ball, when the event $A$ is fixed \cite {hu2013kullback,jiang2016data}. In particular, suppose the stochasticity space is $\mathcal Y=\mathbb R^k$, and $\mathbb P_{\hat\theta}$ admits a density $p_{\hat\theta}$. Theorem 1 in \cite{jiang2016data} shows that for any $A$,
\begin{eqnarray}
\underset{d_\phi(\mathbb P_{\hat\theta},
	\mathbb Q) \leq \lambda}{\text{max}} \mathbb Q( A) \leq \epsilon \iff \mathbb P_{\hat\theta} ( A) \leq \epsilon',\label{equivalence DRO}
\end{eqnarray}
where $\epsilon'=\epsilon'(\epsilon,\lambda,\phi)>0$ can be explicitly determined by $\epsilon$, $\lambda$ and $\phi$ as
\begin{equation}
\epsilon'(\epsilon,\lambda,\phi)=\max\left\{1-\inf_{\substack{z>0,z+\pi z\leq\ell_\phi\\\underline m(\phi^*)\leq z_0+z\leq\overline m(\phi^*)}}\left\{\frac{\phi^*(z_0+z)-z_0-\epsilon z+\lambda}{\phi^*(z_0+z)-\phi^*(z_0)}\right\},0\right\}
\label{adjusted tolerance}
\end{equation}
with $\phi^*(t)=\sup_{x}\{tx-g(x)\}$ being the conjugate function of $\phi$ and $\underline m(\phi^*)=\sup\{m\in\mathbb R:\phi^*\text{\ is a finite constant on\ }(-\infty,m]\}$, $\overline m(\phi^*)=\inf\{m\in\mathbb R:\phi^*(m)=+\infty\}$, $\ell_\phi=\lim_{x\to+\infty}\phi(x)/x$, and $\pi=-\infty$ if $Leb\{[p_{\hat\theta}=0]\}=0$, $0$ if $Leb\{[p_{\hat\theta}=0]\}>0$ and $Leb\{[p_{\hat\theta}=0]\setminus A\}=0$, and $1$ otherwise, where $Leb\{\cdot\}$ is the Lebesgue measure on $\mathbb R^k$.

The above equivalence can be used to obtain the following result.

\begin{theorem}\label{inp2}
Suppose $\mathcal Y=\mathbb R^k$ and $\mathbb P_{\hat\theta}$ admits a density. Among all $\mathbb P_0$ such that $\frac{d\mathbb P_{\hat\theta}}{d\mathbb P_0}$ exists and is continuous, positive almost surely, an optimal choice of $\mathbb P_0$ that minimizes $M(\mathbb P_0, \{\mathbb Q: d_\phi(\mathbb P_{\hat\theta},\mathbb Q)\leq \lambda\},\delta)$, namely the optimal value of \eqref{non_par_dro}, is the center of the $\phi$-divergence ball $\mathbb P_{\hat\theta}$. Moreover, this gives $M(\mathbb P_0, \{\mathbb Q: d_\phi(\mathbb P_{\hat\theta},\mathbb Q)\leq \lambda\},\delta)={\epsilon'}^{-1}(\delta,\lambda,\phi)$, where ${\epsilon'}^{-1}(\cdot,\lambda,\phi)$ is the inverse of the function $\epsilon'=\epsilon'(\epsilon,\lambda,\phi)$ defined in \eqref{adjusted tolerance} with respect to $\epsilon$, given by
\begin{equation}
    {\epsilon'}^{-1}(x,\lambda,\phi)\triangleq \min\{\epsilon\geq 0: \epsilon'(\epsilon,\lambda,\phi)\geq x\}\label{inverse}
\end{equation}
\end{theorem}

\begin{proof}
    From Theorem 1 in \cite{jiang2016data}, we know that, for any $ A \subset \mathcal Y$ and $0\leq\epsilon\leq1$, \eqref{equivalence DRO} holds.
%     \begin{eqnarray}
% \underset{d_\phi(\mathbb P_{\hat\theta},
% 	\mathbb Q) \leq \lambda}{\text{min}} \mathbb Q( A) \geq 1-\epsilon \iff \mathbb P_{\hat\theta} ( A) \geq 1-\epsilon^{+},\label{equivalence DRO1}
% \end{eqnarray}
% \begin{eqnarray}
% \underset{d_\phi(\mathbb P_{\hat\theta},
% 	\mathbb Q) \leq \lambda}{\text{max}} \mathbb Q( A) \leq \epsilon \iff \mathbb P_{\hat\theta} ( A) \leq \epsilon',\label{equivalence DRO}
% \end{eqnarray}
% where $\epsilon'=\epsilon'(\epsilon,\lambda,\phi)>0$ can be explicitly determined by $\epsilon$, $\lambda$ and $\phi$ as defined in \eqref{adjusted tolerance}.
% Moreover, note that for a fixed $\phi$-divergence, $\epsilon'$ is a non-decreasing function of $\epsilon$ and a non-increasing function of $\lambda$.
% By taking the complement of the arbitrary set $A$, we can rewrite \eqref{equivalence DRO1} as
% \begin{eqnarray}
% \underset{d_\phi(\mathbb P_{\hat\theta},
% 	\mathbb Q) \leq \lambda}{\text{max}} \mathbb Q( A) \leq \epsilon \iff \mathbb P_{\hat\theta} ( A) \leq \epsilon^{+},\label{equivalence DRO}
% \end{eqnarray}
% for any $ A \subset \mathcal Y$.
We can rewrite the optimal value of problem \eqref{non_par_dro} in the form:
\begin{equation}
    \begin{aligned}
 &\underset{ \epsilon \geq 0}{\text{min}}
& & \epsilon \\
&{\text{ s.t.}}
& & \underset{d_\phi(\mathbb P_{\hat\theta},
	\mathbb Q) \leq \lambda}{\text{max}}\mathbb Q(A) \leq \epsilon \text{ for all $A \subset \mathcal Y$ such that } \mathbb P_0(A) \leq \delta,
\end{aligned}
\end{equation}
which, according to \eqref{equivalence DRO}, has the same optimal value as
\begin{equation}\label{non-par-dro11}
\begin{aligned}
&\underset{ \epsilon \geq 0}{\text{min}}
& & \epsilon \\
&{\text{ s.t.}}
& & \mathbb P_{\hat\theta}(A) \leq \epsilon'  \text{ for all $A \subset \mathcal Y$ such that } \mathbb P_0(A) \leq \delta.
\end{aligned}
\end{equation}
Since, fixing $\phi$ and $\lambda$, $\epsilon'$ is a non-decreasing function of $\epsilon$, we see that minimizing $\epsilon$ is equivalent to minimizing $\epsilon'$.
% Let us define the inverse of $\epsilon'$ with respect to $\epsilon$ as
% \begin{equation}
%     {\epsilon'}^{-1}(x)\triangleq \min\{\epsilon\geq 0: \epsilon'(\epsilon,\lambda,\phi)\geq x\},
% \end{equation}
Denoting $\nu^\star$ as the optimal value of the optimization problem
\begin{equation}\label{non_par_dro1}
    \begin{aligned}
&\underset{ A\subset \mathcal Y}{\text{max}}
& & \mathbb P_{\hat\theta}(A)\\
&{\text{ s.t.}}
& & \mathbb P_0(A) \leq \delta,
\end{aligned}
\end{equation}
then the optimal value of \eqref{non-par-dro11} is ${\epsilon'}^{-1}(\nu^*,\lambda,\phi)$. Moreover, this is achievable by setting $\mathbb P_0=\mathbb P_{\hat\theta}$ that gives the optimal value $\nu^*=\delta$ to \eqref{non_par_dro1} by Theorem \ref{inp}.

% we notice that, subject to a difference of $\epsilon^{-}(\nu^\star)-\nu^\star$ in the objective function, \eqref{non_par_dro1} is in the form of \eqref{hypo_test_one} with $\mathbb Q$ replaced by $\mathbb P_{\hat\theta}$. Furthermore, since $\epsilon^{-}(\nu)$ is a increasing function of $\nu$, choosing a $\mathbb P_0$ which minimizes the optimal value $\epsilon^{-}(\nu^\star)$ of \eqref{non_par_dro} is equivalent as choosing a $\mathbb P_0$ which minimizes the optimal value $\nu^\star$ of $\eqref{non_par_dro1}$. Thus, according to Theorem \ref{inp}, the optimal choice of generating measure $\mathbb P_0$ to minimize the optimal value of \eqref{non_par_dro} is $\mathbb P_{\hat\theta}$.
\end{proof}

An implication of Theorem \ref{inp2} is that, by noting that a parametric divergence ball  lies inside a corresponding nonparametric ball, we can compute a bound for $M$ to obtain a required Monte Carlo size, drawn from the baseline $\mathbb P_{\hat\theta}$, to get a feasible solution for the distributionally robust CCP \eqref{DRO_problem} and subsequently the CCP \eqref{main_problem}. More precisely, recall the bounding function $M(\mathbb P_0,\mathcal U_{data},\delta)=M(\mathbb P_0,\{\mathbb Q: d_\phi(\mathbb P_{\hat\theta},
	\mathbb Q) \leq \lambda,\mathbb Q\in\mathcal P_{para}\},\delta)$ with $\lambda=\phi''(1)\chi^2_{1-\alpha,D}/(2n)$, given by \eqref{hypo_test1}, as the optimal value of
    \begin{equation}\label{par_dro}
\begin{aligned}
& \underset{ d_\phi(\mathbb P_{\hat\theta},
	\mathbb Q) \leq \lambda,\mathbb Q\in\mathcal P_{para}, A\subset \mathcal Y}{\text{max}}
& & \mathbb Q (A) \\
& {\text{\qquad \qquad
s.t.}}
& & \mathbb P_0(A) \leq \delta.
\end{aligned}
\end{equation}

We have:
\begin{coro}
Given a data size $n$, suppose $\mathcal Y=\mathbb R^k$ and $\mathbb P_{\hat\theta}$ admits a density, where $\hat\theta$ is the MLE under Assumption \ref{con_and_nor}. If we choose $\delta_\epsilon=\epsilon'(\epsilon,\phi''(1)\chi^2_{1-\alpha,D}/(2n),\phi)$ and draw $N_{exact}(\delta_\epsilon,\beta,d)$ Monte Carlo samples from the generating distribution $\mathbb P_{\hat\theta}$ to construct the sampled problem \eqref{Monte_SG}, then the obtained solution will be feasible for \eqref{main_problem} with asymptotic confidence level at least $1-\alpha-\beta$.\label{para nonpara}
\end{coro}

\begin{proof}
Note that a parametric divergence ball  lies inside a corresponding nonparametric ball in the sense that
$$\{\mathbb Q:d_\phi(\mathbb P_{\hat\theta},\mathbb Q)\leq\lambda,\mathbb Q\in\mathcal P_{para}\}\subseteq\{\mathbb Q:d_\phi(\mathbb P_{\hat\theta},\mathbb Q)\leq\lambda\}$$
Thus, by the definition of $M$, we have
$$M(\mathbb P_0, \{\mathbb Q: d_\phi(\mathbb P_{\hat\theta},\mathbb Q)\leq \lambda,\mathbb Q\in\mathcal P_{para}\},\delta)\leq M(\mathbb P_0, \{\mathbb Q: d_\phi(\mathbb P_{\hat\theta},\mathbb Q)\leq \lambda\},\delta)$$
In particular,
$$M(\mathbb P_{\hat\theta}, \{\mathbb Q: d_\phi(\mathbb P_{\hat\theta},\mathbb Q)\leq \lambda,\mathbb Q\in\mathcal P_{para}\},\delta)\leq M(\mathbb P_{\hat\theta}, \{\mathbb Q: d_\phi(\mathbb P_{\hat\theta},\mathbb Q)\leq \lambda\},\delta)={\epsilon'}^{-1}(\delta,\lambda,\phi)$$
where the equality follows from Theorem \ref{inp2}. Thus, if we choose $\delta_\epsilon$ such that ${\epsilon'}^{-1}(\delta_\epsilon,\lambda,\phi)\leq\epsilon$, or $\delta_\epsilon=\epsilon'(\epsilon,\lambda,\phi)$, where $\lambda=\phi''(1)\chi^2_{1-\alpha,D}/(2n)$ as presented in \eqref{uncertainty choice}, and the generating distribution as $\mathbb P_{\hat\theta}$, then Corollary \ref{main cor} guarantees that running SO on $N_{exact}(\delta_\epsilon,\beta,d)$ Monte Carlo samples gives a feasible solution for \eqref{main_problem} with confidence asymptotically at least $1-\alpha-\beta$.
\end{proof}

Corollary \ref{para nonpara} thus provides an implementable procedure to handle \eqref{main_problem} through \eqref{DRO_problem}.

\subsection{Parametric DRO}\label{sec:parametric}
Next we discuss further the choice of generating distributions in parametric DRO beyond $\mathbb P_{\hat\theta}$. While the ball center $\mathbb P_{\hat\theta}$ is a valid choice, the equivalence relation \eqref{equivalence DRO} does not apply when the divergence ball is in a parametric class, and the optimal choice of the generating distribution may no longer be $\mathbb P_{\hat\theta}$, as shown in the next result.
\begin{theorem}
In terms of selecting a generating distribution $\mathbb P_0$ to minimize $M(\mathbb P_0, \{\mathbb Q: d_\phi(\mathbb P_{\hat\theta},\mathbb Q)\leq \lambda, \mathbb Q \in \mathcal P_{para}\},\delta)$, the optimal value of  \eqref{par_dro}, the choice $\mathbb P_{\hat\theta}$ can be strictly dominated by other distributions.\label{main parametric DRO}
\end{theorem}

Intuitively, Theorem \ref{main parametric DRO} arises because the extreme distribution that achieves the equivalence relation \eqref{equivalence DRO} may not be in the considered parametric family. It implies more flexibility in choosing the generating measure $\mathbb P_0$, in the sense of requiring less Monte Carlo samples than using $\mathbb P_{\hat\theta}$.

% Although not the exact formulation, lehmann1948most,lehmann1952existence,
From the standpoint of hypothesis testing in Section \ref{step1}, the imposed minimax problem \eqref{par_dro} in searching for the best $\mathbb P_0$ can be viewed as finding a simple null hypothesis that is uniformly least powerful across the uncertainty set. This question is related and appears more general than finding the least favorable or powerful prior in testing against composite null hypothesis \cite{lehmann2006testing}. In the latter context, given a set $\Theta_1$, one aims to find
% The latter To be specific, if we consider the following hypothesis testing problem:
% \begin{equation}\label{least_favorable_prior}
% \begin{aligned}
% & \underset{\theta_1\in\Theta_1 }{\text{max}}
% & & \mathbb P_{\theta_1} (A) \\
% & \underset{\theta_0\in\Theta_0 }{\text{s.t.}}
% & & \mathbb P_0(A) \leq \delta,
% \end{aligned}
% \end{equation}
% with composite null (i.e., $|\Theta_0|>1$), then
% In other words, if we can choose from a family of measures
a distribution $\mu^\star(d\theta_0)$ such that $\Gamma(\mu^\star) \leq \Gamma(\mu)$ for all distributions $\mu(d\theta_0)$ on $\Theta_0$, where $\Gamma(\mu)$ is the optimal value of
\begin{equation}\label{least_favorable_prior_de}
\begin{aligned}
& \underset{\theta_1\in\Theta_1 }{\text{max}}
& & \mathbb P_{\theta_1} (A) \\
& \underset{ }{\text{s.t.}}
& &\int_{\Theta_0} \mathbb P_{\theta_0}(A) \mu(d\theta_0) \leq \delta.
\end{aligned}
\end{equation}
The distribution $\mu(d\theta_0)$ is interpreted as a prior on a composite null hypothesis parametrized by $\theta_0$, and $\mu^*(d\theta_0)$ is the least favorable prior. The difference between \eqref{least_favorable_prior_de} and our formulation \eqref{par_dro} lies in the restriction to measures of the form $\mathbb P_0=\int_{\Theta_0}\mathbb P_{\theta_0}\mu(d\theta_0)$ for the former, leading to a smaller search space than ours. This mixture-type $\mathbb P_0$ and the Bayesian connection will partly motivate our investigation in Section \ref{gen}.

% , then we want to pick a $\mathbb P_0$ with some $\mu^\star(d\theta_0)$ such that it has the least power, or with the minimum optimal value for \eqref{least_favorable_prior_de}. As we have seen, this shares a similar taste as \eqref{par_dro} and will motivate our discussion in section \ref{sec:mixture}.

To prove Theorem \ref{main parametric DRO}, we present a counter example and also some related discussion.

\begin{example}\label{exam1}
Consider the uncertainty set $\mathcal U_{data}=\{\mathbb P_\theta, : -1 \leq \theta \leq 1 \}$ within Gaussian location family on $\mathbb R$ with  $\mathbb P_\theta(dy)=\frac{1}{\sqrt{2\pi}}e^{-\frac{(y-\theta)^2}{2}}$. This can be thought of, e.g., as an uncertainty set based on the $\chi^2$-distance, the latter defined between two probability measures $\mathbb P_1$ and $\mathbb P_2$  as
\begin{align}
\chi^2(\mathbb{P}_1, \mathbb{P}_2)=&\int_{\mathcal Y} (\frac{d\mathbb{P}_2}{d\mathbb{P}_1}-1)^2 \mathbb{P}_1(dy).\label{chidef1}
\end{align}
Note that the $\chi^2$-distance is in the family of $\phi$-divergences, by choosing $\phi=(x-1)^2$.  We aim to find a generating distribution $\mathbb P_{0}$ to minimize $M(\mathbb P, \{\mathbb P_{\theta}:\theta\in\mathcal U_{data} \},\delta)$, the optimal value of
\begin{equation}\label{opt example}
\begin{aligned}
& \underset{\theta\in\mathcal U_{data}, A\subset \mathbb R}{\text{max}}
& & \mathbb P_{\theta} (A) \\
& {\text{\quad s.t.}}
& & \mathbb P_0(A) \leq \delta.
\end{aligned}
\end{equation}
We consider several symmetric distributions as $\mathbb P_0$ (symmetry is reasonably conjectured as a good property since an imbalanced shift might increase the power for the alternative hypothesis on one side and the worst case overall). We list these symmetric distributions in increasing variability:

\begin{equation}\label{hypo_test3}
\begin{aligned}
& \mathbb P_{0}^1(dy)=\frac{1}{\sqrt{2\pi}}e^{-\frac{y^2}{2}} \\
& \mathbb P_{0}^2(dy)=\frac{1}{\sqrt{2\pi\cdot 2}}e^{-\frac{y^2}{2\cdot 2}} \quad \\
&
\mathbb P_{0}^3(dy)=\frac{1}{2\sqrt{2\pi}}\bigg(e^{-\frac{(y-1)^2}{2}}+e^{-\frac{(y+1)^2}{2}} \bigg).
\end{aligned}
\end{equation}
 Given $0 \leq \theta\leq 1$, it can be shown by the Neyman-Pearson lemma that the rejection region $ A^\star$ (i.e. the set giving the optimal value of \eqref{opt example} for a given $\theta$) for $\mathbb P_0^1$ has the form $\{ y:y>c_1 \}$, for $\mathbb P_0^2$ the form $\{ y: \abs{y-2\theta} \leq c_2\}$ and for $\mathbb P_0^3$ the form $\{y: \frac{e^{\theta y}}{e^y+e^{-y}} > c_3 \}$, for some $c_1$, $c_2$ and $c_3$. Let $\delta=0.05$ be the tolerance level, it can be shown through numerical verification that
\begin{equation}\label{hypo_test4}
\begin{aligned}
& M(\mathbb P_{0}^1,\mathcal \{\mathbb P_{\theta}:\theta\in\mathcal U_{data} \},0.05)=0.2595 \\
& M(\mathbb P_{0}^2,\mathcal \{\mathbb P_{\theta}:\theta\in\mathcal U_{data} \},0.05)=0.1160\\
&M(\mathbb P_{0}^3,\mathcal \{\mathbb P_{\theta}:\theta\in\mathcal U_{data} \},0.05)=0.0995.
\end{aligned}
\end{equation}
Thus, the natural choice $\mathbb P_{\hat\theta}=\mathbb P_{0}^1$ based on relaxing to nonparametric DRO yields a bounding function $M(\cdot)$ that is outperformed by $\mathbb P_{0}^2$ or $\mathbb P_{0}^3$. Later in Section \ref{gen} we will see numerically how $\mathbb P_0^2$ and $\mathbb P_0^3$ can lead to a smaller sample size requirements.
\end{example}

Although Theorem \ref{main parametric DRO} reveals room to search for the best generating distribution, the involved optimization, or even just finding an improved distribution over $\mathbb P_{\hat\theta}$, appears to be nontrivial. In particular, the maximization problem in \eqref{par_dro} depends on the computation of $ A^\star$ for each alternative of $\theta\in\mathcal U_{data}$. Section \ref{gen} discusses some approaches to search for improvements. We conclude the current section with some discussion on the choice of statistical distances used in the uncertainty set.

% On the other hand, note that one may attempt to consider only set in the form $
% which may change in form for different distributions in the uncertainty set, drastically increasing the amount of computation.
% Thus, we try to make a relxation of the optimization \eqref{par_dro} by making the inner maximization problem more easily available with the help of $\chi^2$ distance.

% \leq&\mathbb{P}_{0}(\xi\in \mathcal{A})+\sup\limits_{\theta\in\mathcal{U}_{data}} \int_{\mathcal{A}} p(y; \theta)-{p_0(y)} dy \nonumber\\
% =&\mathbb{P}_{0}(\xi\in \mathcal{A})+ \sup\limits_{\theta\in\mathcal{U}_{data}} \int_{\mathcal{Y}} \mathbf{1}\{y \in \mathcal{A}\}  \Big(\frac{p(y; \theta)}{p_0(y)}-1\Big) {p_0(y)} dy \nonumber\\
% \leq & \mathbb{P}_{0}(\xi\in \mathcal{A})+\sup\limits_{\theta\in\mathcal{U}_{data}} \bigg(\int_{\mathcal{Y}}\mathbf{1}\{y \in \mathcal{A}\} {p_0(y)}dy \bigg)^{1/2} \bigg(\int_{\mathcal{Y}}\Big(\frac{p(y; \theta)}{p_0(y)}-1\Big)^2 {p_0(y)} dy\bigg)^{1/2}\nonumber\\
% \leq & \mathbb{P}_{0}(\xi\in \mathcal{A})+\mathbb{P}_{0}(\xi\in \mathcal{A})^{1/2}\cdot (\sup\limits_{\theta\in\mathcal{U}_{data}} \chi^2(\mathbb{P}_{0},\mathbb{P}_{\theta}))^{1/2}.
% \end{align}

\subsection{Choice of Statistical Distance}\label{diss}
% We previously mentioned that the $\chi^2$-distance is one of many statistical distances based on $\phi$-divergences \cite{vajda1972f}. However, the uncertainty set construction and baseline measure search using the $\chi^2$ distance possess advantages in our framework. We compare this with other alternate choices and provide some reasons for our selection.

% First, we note that
We have chosen to use $\phi$-divergence to construct our uncertainty set $\mathcal U_{data}$, and we have seen how this allows us to effectively translate sample size requirements from the data to Monte Carlo. Note that another common type of distance is the Wasserstein distance (e.g., \cite{esfahani2015data,blanchet2016quantifying,gao2016distributionally}). If one can translate the violation probability under a generating distribution into the worst-case violation probability over a Wasserstein ball, then the same line of arguments in Section \ref{sec:outline} applies to using SO on this DRO.
% However, there is some hint that Wasserstein can be too conservative for this application.
Presuming that the size of a parametric Wasserstein-based confidence region can be properly calibrated from data, it is conceivable that the above can give rise to an alternate solution route. It is known (Theorem 3 in \cite{blanchet2016quantifying}), under suitable regularity conditions, that one can equate a Wasserstein-ambiguous probability$\sup\limits_{d_W(\mathbb Q,\mathbb P_{\hat\theta})\leq\lambda} \mathbb{Q}(\xi \in {A})$, where $d_W$ denotes a Wasserstein distance of order 1 and cost function $c$, and $A$ is an event, to $\mathbb P_{\hat\theta}(c(\xi,A)\leq1/\nu^*)$ where $\nu^*\geq0$ is a dual multiplier for the associated optimization problem, and $c(\xi,A)$ denotes the cost-induced distance between a point $\xi$ and a set $A$. Thus, $M(\mathbb P_0,\{\mathbb Q:d_W(\mathbb P_{\hat\theta},\mathbb Q)\leq\lambda\},\delta)$ can be written as
\begin{equation}\label{non_par_dro_Wasserstein}
\begin{aligned}
& \underset{A\subset \mathcal Y}{\text{max}}
& & \mathbb P_{\hat\theta} (c(\xi,A)\leq1/\nu^*) \\
&{\text{ s.t.}}
& & \mathbb P_0(A) \leq \delta.
\end{aligned}
\end{equation}
Compared to the evaluation of $M(\mathbb P_0,\{\mathbb Q:d_\phi(\mathbb P_{\hat\theta},\mathbb Q)\leq\lambda\},\delta)$ in Theorem \ref{inp2}, the tightening of the tolerance level from $\epsilon$ to $\epsilon'$ is now replaced by the set inflation from $A$ to the $(1/\nu^*)$-neighborhood of $A$ given by $\{\xi:c(\xi,A)\leq1/\nu^*\}$. Note that, regardless of the distance used, one could reduce the conservativeness of our analysis by focusing on $A$ in the form $\{x\notin\mathcal X_\xi\}$, but this would require looking at the specific form of the safety set $\mathcal X_\xi$.

\section{Improving Generating Distributions}\label{gen}
This section discusses some approaches to search for better generating distributions beyond the baseline distribution in a divergence ball of DRO. Section \ref{sec:reduction} first states a general result to create better generating distributions. Section \ref{sec:mixture} then specializes to using a mixture distribution on $\theta$ to exploit this result. Sections \ref{sec:line search} and \ref{sec:mix var} then provide two specific ways to construct these mixtures. Finally, Section \ref{sec:numerics mixing} demonstrates some numerical comparisons in using these new mixing generating distributions and also simply using the baseline.
% We will focus on the use of $\chi^2$-distance as it gives us some computational benefits.

% In this section, we introduce several generalizations to improve its efficiency.
% We have mentioned that the required sample size $N_{exact}$ of our scheme depends on the choice of $\mathbb P_0$. In the previous section, we have seen that the choice dictated by the non-parametric DRO is not necessarily optimal one while the search of an optimal choice can be challenging. In this section, we propose a relaxation of the optimization problem \eqref{par_dro} using $\chi^2$ distance for a more tractable search of a better $\mathbb P_0$.

% Second, as we have mentioned in the introduction, several developed methods are dedicated to reduce the sample size $N_{exact}$ for the standard SG which can be readily transformed into our method. For demonstration, we introduce one particular method: the FAST method \cite{care2014fast} while the incorporation of other methods follows similarly.

\subsection{A Framework to Reduce Divergence Ball Size by Incorporating Parametric Information}\label{sec:reduction}
% We first point out that the $\epsilon^+$ in the reformulation \eqref{equivalence DRO} in \cite{jiang2016data} dictates the conservativeness level in using Monte Carlo sampling. The smaller the $\epsilon^+$, the more Monte Carlo samples one needs as indicated by \eqref{CCP guarantee} (or its variants), or in other words the smaller the bounding function $M$. This $\epsilon^+$ in generally depends on the choice of $\phi$, the original tolerance level $\epsilon$, and the divergence ball size $\lambda$. In particular, it decreases when $\lambda$ decreases. Our idea of searching for better generating distributions is the following: Given $\phi$ and $\epsilon$, we look for smaller $\lambda$ that satisfies the parametric uncertainty than is required in nonparametric DRO. To illustrate, let us focus on $\chi^2$-distance, in which case if $\epsilin < 1/2$, then \eqref{equivalence DRO} can be written as
% \begin{eqnarray}\label{jiangresult}
% \underset{\chi^2(\mathbb P_{\hat\theta},
% 	\mathbb Q) \leq \lambda}{\text{max}} \mathbb Q( A) \leq \epsilon \iff \mathbb P_{\hat\theta} ( A) \leq \max(0,\epsilon-\frac{\sqrt{\lambda^2+4\lambda(\epsilon-\epsilon^2)}-(1-2\epsilon)\lambda}{2\lambda+2}).
% \end{eqnarray}
The reason why the best choice of generating distribution $\mathbb P_0$ is not the baseline of the divergence ball, $\mathbb P_{\hat\theta}$, in minimizing $M(\mathbb P_0, \{\mathbb Q: d_\phi(\mathbb P_{\hat\theta},\mathbb Q)\leq \lambda, \mathbb Q \in \mathcal P_{para}\},\delta)$ is that the equivalence relation \eqref{equivalence DRO} does not hold when $\mathbb Q$ is restricted to a parametric class. In some sense the reduction to the unambiguous chance constraint in the right hand side of \eqref{equivalence DRO} is over-conservative as it does not account for parametric information. Suppose we would still like to use the analytically tractable relation \eqref{equivalence DRO}, but at the same time be less conservative. Then, one approach is to find a new baseline distribution, say $\tilde{\mathbb P}$, such that the parametrically restricted divergence ball $\{\mathbb Q: d_\phi(\mathbb P_{\hat\theta},\mathbb Q)\leq \lambda, \mathbb Q \in \mathcal P_{para}\}$ lies inside a new nonparametric divergence ball at the center $\tilde{\mathbb P}$, namely $\{\mathbb Q: d_\phi(\tilde{\mathbb P},\mathbb Q)\leq \tilde\lambda\}$. If we can obtain a nonparametric ball size $\tilde\lambda$ such that $\tilde\lambda<\lambda$ and the set inclusion holds, then this new ball is also a valid uncertainty set, and, when simply setting the generating distribution as $\mathbb P_0=\tilde{\mathbb P}$ and applying Theorem \ref{inp2}, we have a smaller upper bound for $M(\mathbb P_0, \{\mathbb Q: d_\phi(\mathbb P_{\hat\theta},\mathbb Q)\leq \lambda, \mathbb Q \in \mathcal P_{para}\},\delta)$ than ${\epsilon'}^{-1}(\delta,\lambda,\phi)$ obtained from using Theorem \ref{inp2} directly with the parametric constraint relaxed.
% (note that ${\epsilon'}^{-1}(\delta,\lambda,\phi)$ is non-decreasing in $\lambda$ for fixed $\delta$) .

To above mechanism can be executed as follows. Let $\mathcal U_{data}=\{\mathbb Q: d_\phi(\mathbb P_{\hat\theta},\mathbb Q)\leq \lambda, \mathbb Q \in \mathcal P_{para}\}$. For any $\mathbb P_0$, let
\begin{equation}
\mathcal{D}_{data}(\mathbb{P}_0,\phi)\triangleq   \sup\limits_{\mathbb Q\in\mathcal U_{data}} d_\phi(\mathbb{P}_0,\mathbb{Q}).\label{transform para to nonpara}
\end{equation}Then we clearly have
\begin{equation}
\mathcal U_{data}\subseteq\{\mathbb Q:d_\phi(\mathbb P_0,\mathbb Q)\leq\mathcal D_{data}(\mathbb P_0,\phi)\}\label{inclusion},
\end{equation}
since the right-hand-side set includes distributions outside of the parametric family as well.

Our goal is to find $\mathbb P_0$ to minimize $\mathcal{D}_{data}(\mathbb{P}_0,\phi)$ or any upper bound of $\mathcal{D}_{data}(\mathbb{P}_0,\phi)$ so that it is smaller than the ball size $\lambda$ appearing in the original parametric divergence ball $\mathcal U_{data}$. We state the implication of this as follows:

\begin{theorem}\label{inp3}
Suppose $\mathcal Y=\mathbb R^k$ and $\mathbb P_{\hat\theta}$ admits a density. Consider the parametric divergence ball $\mathcal U_{data}=\{\mathbb Q: d_\phi(\mathbb P_{\hat\theta},\mathbb Q)\leq \lambda, \mathbb Q \in \mathcal P_{para}\}$. Suppose we can find $\mathbb P_0$ such that $\mathcal{D}_{data}(\mathbb{P}_0,\phi)$ defined in \eqref{transform para to nonpara} satisfies $\mathcal{D}_{data}(\mathbb{P}_0,\phi)<\lambda$. Then we have
\begin{align}
\min_{\mathbb P_1}M(\mathbb P_1, \{\mathbb Q: d_\phi(\mathbb P_{\hat\theta},\mathbb Q)\leq \lambda, \mathbb Q \in \mathcal P_{para}\},\delta)&\leq\min_{\mathbb P_1}M(\mathbb P_1, \{\mathbb Q: d_\phi(\mathbb P_{\hat\theta},\mathbb Q)\leq\mathcal{D}_{data}(\mathbb{P}_0,\phi)\},\delta)\notag\\
&\leq\min_{\mathbb P_1}M(\mathbb P_1, \{\mathbb Q: d_\phi(\mathbb P_{\hat\theta},\mathbb Q)\leq\lambda\},\delta)\label{bounding relation}
\end{align}
and
\begin{equation}
\min_{\mathbb P_1}M(\mathbb P_1, \{\mathbb Q: d_\phi(\mathbb P_{\hat\theta},\mathbb Q)\leq \lambda, \mathbb Q \in \mathcal P_{para}\},\delta)\leq{\epsilon'}^{-1}(\delta,\mathcal{D}_{data}(\mathbb{P}_0,\phi),\phi)\leq{\epsilon'}^{-1}(\delta,\lambda,\phi)\label{bounding explicit}
\end{equation}
where ${\epsilon'}^{-1}(\epsilon,\lambda,\phi)$ is defined in \eqref{inverse}.
\end{theorem}

\begin{proof}
By the definition of $\mathcal{D}_{data}(\mathbb{P}_0,\phi)$, \eqref{inclusion} holds. Together with the condition $\mathcal{D}_{data}(\mathbb{P}_0,\phi)<\lambda$, we have the set inclusions
\begin{equation}
\{\mathbb Q: d_\phi(\mathbb P_{\hat\theta},\mathbb Q)\leq \lambda, \mathbb Q \in \mathcal P_{para}\}\subseteq\{\mathbb Q:d_\phi(\mathbb P_0,\mathbb Q)\leq\mathcal D_{data}(\mathbb P_0,\phi)\}\subseteq\{\mathbb Q: d_\phi(\mathbb P_{\hat\theta},\mathbb Q)\leq \lambda\}
\end{equation}
The inequalities \eqref{bounding relation} then follow from the definition of $M$. The inequalities \eqref{bounding explicit} in turn follow immediately from Theorem \ref{inp2}.
\end{proof}

Theorem \ref{inp3} stipulates that choosing $\mathbb P_0$ depicted in the theorem as the generating distribution, and setting ${\epsilon'}^{-1}(\delta,\mathcal{D}_{data}(\mathbb{P}_0,\phi),\phi)$ as an upper bound for $M(\mathbb P_0, \{\mathbb Q: d_\phi(\mathbb P_{\hat\theta},\mathbb Q)\leq \lambda, \mathbb Q \in \mathcal P_{para}\},\delta)$ to obtain the required Monte Carlo size $N_{exact} (\delta_{\epsilon},\beta,d)$ implied by Corollary \ref{main cor}, will give a lighter Monte Carlo requirement than using the bound ${\epsilon'}^{-1}(\delta,\lambda,\phi)$ directly obtained by relaxing the parametric constraint and using $\mathbb P_{\hat\theta}$ as the generating distribution as in Corollary \ref{para nonpara}.

\subsection{Mixture as Generating Distribution}\label{sec:mixture}
Since optimization \eqref{transform para to nonpara} can be difficult to solve generally, we focus on finding improved generating distribution $\mathbb P_0$ so that the implication of Theorem \ref{inp3} holds, instead of fully optimizing \eqref{transform para to nonpara}. In this and the next subsections, we design a search space $\mathcal P_{0}$ for $\mathbb P_0$ that allows the construction of tractable procedures to achieve such improvements, while at the same time ensures the obtained $\mathbb P_0$ are amenable to Monte Carlo simulation.

% To do so, we propose a set of probability distributions $\mathcal P_{prop}$ as the search space for $\mathbb P_0$ that satisfies two properties. First, it is a convex set and also contains the original baseline measure $\mathbb P_{\hat\theta}$ to facilitate efficient search (in a sense that we will discuss in detail). Second, any distribution in $\mathcal P_{prop}$ can  be simulated via Monte Carlo.

% Since we have assumed that we can efficiently obtain samples from the parametric family $\{\mathbb P_{\theta} \}_{\theta \in \Theta}$, the above sampling procedure is efficient as long as it is easy to sample $\theta$ under the measure $\mu$.
% A class of distributions that satisfy these two properties constitutes the mixtures of parametric distributions. To be specific, We start with our first search procedure.
From now on we will focus on $\chi^2$-distance as our choice of $\phi$ for convenience (as will be seen). Suppose that $\mathbb P_\theta$ has density $p(y;\theta)$. We then set $\mathcal P_{0}$ to be the collection of distributions with densities in the form
\begin{equation}\label{mixturedensity}
p_0(y)=\int_{\Theta} p(y;\theta)  \mu(d\theta),
\end{equation}
for some probability measure $\mu$ on $\Theta$.
% , where we expand the definition of $\Theta$ to be either the space of $\theta$ or other natural parameters when no confusion arises.
This class of distributions is easy to sample assuming $p(y;\theta)$ and $\mu$ are, as one can first sample $\theta\sim\mu(d\theta)$  and then $\xi \sim \mathbb{P}_{\theta}$ given $\theta$.
% Thus $p_0(y)$ satisfies the second desirable property discussed above.

Searching for the best $p_0(y)$ requires minimizing $\mathcal{D}_{data}(\mathbb{P}_0)$ over $\mathbb P_0\in\mathcal P_{0}$ (where for convenience we denote $\mathcal D_{data}(\mathbb P_0)$ as $\mathcal D_{data}(\mathbb P_0,\phi)$ with $\phi$ representing the $\chi^2$-distance). We first use \eqref{chidef1} to write
\begin{align}\label{equivalent}
\mathcal{D}_{data}(\mathbb{P}_0)=&\sup_{\theta\in\mathcal{U}_{data}}\int_{\mathcal{Y}}\Big(\frac{p(y; \theta)}{p_0(y)}-1\Big)^2 {p_0(y)} dy \nonumber\\
=& \sup_{\theta\in\mathcal{U}_{data}}\int_{\mathcal{Y}}\frac{(p(y; \theta))^2}{p_0(y)} dy -1 \nonumber\\
=&\sup_{\theta\in\mathcal{U}_{data}}\int_{\mathcal{Y}}\frac{(p(y; \theta))^2}{\int_{\Theta} p(y;\theta') \mu(d\theta')} dy -1 .
\end{align}
Denoting $\mathcal{P}(\Theta)$ as the space of probability measures on $\Theta$, we define the function $L: \mathcal{P}(\Theta) \times \Theta \rightarrow \mathbb R$ to be
\begin{equation}\label{def}
L(\mu,\theta)\triangleq\int_{\mathcal{Y}}\frac{(p(y; \theta))^2}{\int_{\Theta} p(y;\theta') \mu(d\theta')} dy,
\end{equation}
assuming the integral is well-defined for $\mathcal P(\Theta) \times \Theta$ and further define
\begin{equation}\label{def1}
l(\mu)\triangleq\sup\limits_{\theta\in\mathcal{U}_{data}}L(\mu,\theta).
\end{equation}
Thus \eqref{equivalent} can be written as $\mathcal{D}_{data}(\mathbb{P}_0)=l (\mu)-1$, and minimizing $\mathcal{D}_{data}(\mathbb{P}_0)$ is equivalent to solving
\begin{equation}\label{minimax}
\min_{\mu\in\mathcal{P}(\Theta)} l(\mu)=\min_{\mu\in\mathcal{P}(\Theta)}\max_{\theta\in\mathcal{U}_{data}}L(\mu,\theta).
\end{equation}
% the minimax optimization problem:
Optimization \eqref{minimax} has the following convexity property:
\begin{lemma}
	The outer minimization in problem \eqref{minimax} is convex. \label{convexity}
\end{lemma}

Lemma \ref{convexity} can be proved by direct verification, which is shown in Appendix \ref{sec:proofs}. Note also that, if $\mu$ is the point mass $\delta_{\theta}$ for $\theta\in\Theta$, then the mixture distribution would recover the parametric distribution $\mathbb{P}_{\theta}$. Hence the proposed family $\mathcal P_{0}$ includes $\{\mathbb P_{\theta} \}_{\theta \in \Theta}$, and in particular the original baseline distribution $\mathbb P_{\hat\theta}$.
% Thus $\mathcal P_{para}$ also satisfies the first desirable property listed above.
Although the outer minimization of \eqref{minimax} is a convex problem, computing $l(\mu)$ involves a non-convex optimization and is difficult in general.
% \subsection{Danskin's Theorem and Descent Directions}
% Recognizing the convexity of problem \eqref{minimax} is a promising start. However, as it turned out, solving \eqref{minimax} exactly can still be difficult due to the fact that $l(\mu)$ is not straightforward to compute for a generic $\mu\in\mathcal P(\Theta)$. To overcome this difficulty and improve on $l(\delta_{\hat\theta})=\mathcal D_{data}(\mathbb P_0)+1$, In order to do this, we introduce a version of Danskin's theorem that we will utilize that can be proved by applying a version of Danskin's Theorem .
Our approach is to search for a descent direction for the convex function $l(\cdot)$ from $\delta_{\hat\theta}$. In the following, we will study two types of search directions, each using its own version of Danskin's Theorem \cite{bertsekas1971control,bertsekas2003convex}. To proceed, we introduce the following definition:

\begin{defn}
	Define $\Theta^*(\mu)$ to be the set of optimal points for the maximization problem in ${l}(\mu)=\sup\limits_{\theta\in\mathcal{U}_{data}}L(\mu,\theta)$ given $\mu\in\mathcal{P}(\Theta)$ :
	\begin{equation}
	\Theta^*(\mu)=\text{argmax}_{\theta\in\mathcal U_{data}}L(\mu,\theta)
%     \{\theta \in \mathcal{U}_{data}:{L}(\mu,\theta)={l}(\mu)\}.
	\end{equation}
\end{defn}

It can be shown that $\Theta^*(\mu)$ is non-empty and  $\Theta^*(\mu)\subseteq\mathcal{U}_{data}$ because  $\mathcal{U}_{data}$ is compact and ${L}(\mu,\theta)$ is continuous in $\theta$.
% We have the following:

\subsection{Mixing with a Proposed Distribution}
\label{sec:line search}
We consider mixing distributions in the form $(1-t)\delta_{\hat\theta}+t\mu_{prop}$ for some proposed distribution $\mu_{prop}$, and look for a descent direction by varying $t$ from 0 to 1. We have the following result that is a consequence of Danskin's Theorem  that involves a one-sided derivative. We provide proofs both for this theorem and our following result in Appendix \ref{sec:proofs}.
% To proceed, we have the following definition:

% we get the following more explicit expression for the directional gradient of $l(\mu)$:

\begin{theorem}\label{ana}
	Fix any $\mu_1,\mu_2\in\mathcal P(\Theta)$ and $\theta \in \Theta$. Under the assumptions that $\psi(t)=L((1-t)\mu_1+t\mu_2,\theta)$ is well defined for $0\leq t \leq 1$, we know that the function $g(y,t)$
	\begin{equation*}
	g(y,t):\mathcal Y \times [0,1]\triangleq\frac{(p(y; \theta))^2}{(1-t)\int_{\Theta} p(y;\theta') \mu_1(d\theta')+t\int_{\Theta} p(y;\theta') \mu_2(d\theta')}
	\end{equation*}
	is integrable for $t\in[0,1]$. If we further assume that there exists a integrable function $g_0(y)$ such that
	\begin{equation*}
	\bigg|\frac{(p(y; \theta))^2\cdot \int_{\Theta}p(y;\theta')(\mu_1-\mu_2)(d\theta')}{(\int_{\Theta} p(y;\theta') ((1-t)\mu_1+t\mu_2)(d\theta'))^2}\bigg| \leq g_0(y),
	\end{equation*}
	then we have the right derivative of $\psi(t)$ at $t=0$ given by
%     could extend the result of \thmref{dktheorem} to
	\begin{align}
	{\psi}^{+}(0)&=\sup_{\theta\in \Theta^*(\mu_1)}\lim_{t\downarrow 0} \frac{{L}((1-t)\mu_1+t\mu_2,\theta)-{L}(\mu_1,\theta)}{t}\notag\\
    &=\sup_{\theta\in \Theta^*(\mu_1)}\int_{\mathcal{Y}}\frac{(p(y; \theta))^2\cdot \int_{\Theta}p(y;\theta')(\mu_1-\mu_2)(d\theta')}{(\int_{\Theta} p(y;\theta') \mu_1(d\theta'))^2} dy.\label{mixture direction}
	\end{align}
\end{theorem}

% p_{prop}(y)=\int_{\Theta} p(y;\theta')
The quantity $\psi^+(0)$ is the directional derivative of $L(\mu_1)$ in the direction $\mu_2-\mu_1$. Thus, to improve on $\mathcal D_{data}(\mathbb P_{\hat\theta})$, we can propose a mixing distribution $\mu_{prop}(d\theta')$, and substitute $\mu_1=\delta_{\hat\theta}$ and $\mu_2=\mu_{prop}$ in \eqref{mixture direction} to check if
\begin{equation}
\sup_{\theta\in \Theta^*(\delta_{\hat\theta})}\int_{\mathcal{Y}}\frac{(p(y; \theta))^2\cdot \int_{\Theta}p(y;\theta')(\delta_{\hat\theta}-\mu_{prop})(d\theta')}{ p(y;\hat\theta)^2} dy < 0,\label{verify descent}
\end{equation}
which indicates a strict descent for $l(\cdot)$ from $\delta_{\hat\theta}$ to $\mu_{prop}$. In this case, it follows from the convexity of $l(\cdot)$  that we can find some $0<t\leq 1$ such that $l((1-t)\delta_{\hat\theta}+t\mu_{prop}) < l(\delta_{\hat\theta})$, so that
\begin{equation}
p_{t}(y)=\int_{\Theta} p(y;\theta')  ((1-t)\delta_{\hat\theta}+t\mu_{prop})(d\theta'),\label{bisection}
\end{equation}
gives rise to $\mathcal D_{data}(\mathbb P_0) < \mathcal D_{data}(\mathbb P_{\hat\theta})$. Finding such a $t$ can be done by a bisection search or enumerating $\mathcal D_{data}(\mathbb P_0)$ on $p_t$ over a grid of $t$. Note that the above can be implemented only if \eqref{verify descent} can be verified and also if $\mathcal D_{data}(\mathbb P_0)$ is computable. We will show that both properties are satisfied for the case of multivariate Gaussian when $\mu_{prop}$ is properly chosen. In particular, we will identify general sufficient conditions for $\mu_{prop}$ to guarantee \eqref{verify descent}, and also find $\mu_{prop}$ such that the maximization involved in computing $\mathcal D_{data}(\mathbb P_0)$ in \eqref{equivalent} can be reduced to a one-dimensional problem.

% the involved integratl  If the integration involved in evaluating $\mathcal D_{data}(\mathbb P_0)$ is difficult, we can first use Monte Carlo to approximate $\mu_{prop}(d\theta)$ and the integral. Then optimize the empirical value over different points in $\Theta$ provided we can overcome the curse of dimensionality for large $D$. On the other hand, we also note that the above idea can be implemented only if  Fortunately, both problems can be overcome. We summarize the results in proposition \ref{collec}, \ref{cod} and \ref{cod1}.

% The next two subsections will provide examples for $\mu_{prop}$ such that this is doable.
% examples in utilizing the above procedures to search for better generating distributions.
% To illustrate our method in this section, we use it on a concrete example, the multivariate Gaussian distribution with unknown mean $\theta\in \mathbb R^D$ and known covariance matrix $\Sigma\in\mathbb R^{D\times D}$.

% $\{\theta_i\}_{1\leq i \leq M} \subseteq \Theta$ which leads to ${p_0}(y)=\sum_{i=1}^{M}p(y,\theta_i)/M$.

% \subsection{Example: Multivariate Gaussian Boundary Mixing}
% To facilitate our discussion, we use a concrete example,

 Consider a multivariate Gaussian distribution with unknown mean $\Theta \subset \mathbb{R}^D$ in an open convex set with density
\begin{equation}\label{gaupdf}
p(y;\theta)=\frac{1}{\sqrt{(2\pi)^D |\Sigma|}} \cdot  e^{-\frac{1}{2}(y-\theta)^\intercal \Sigma^{-1}(y-\theta)},
\end{equation}
where $\Sigma$ is a fixed positive semi-definite covariance matrix.
% and the Fisher information as
% \begin{equation}\label{fish}
% \mathcal{I}(\theta)=-\mathbb{E}_{\theta}[\nabla^2\log p(Y; \theta)]=\nabla^2 F(\theta).
% \end{equation}
Direct verification (in Appendix \ref{sec:proofs}) shows that
\begin{align}\label{final1}
{\mathcal{U}}_{data}\triangleq\left\{\theta \in \Theta: \chi^2(\mathbb P_{\hat\theta}, \mathbb P_\theta)\leq \frac{\chi^2_{1-\alpha,D}}{n}\right\}
% =&\left\{\theta \in \Theta: e^{F(2\theta -\hat\theta)-(2F(\theta)-F(\hat\theta))}-1\leq \frac{\chi^2_{1-\alpha,D}}{n}\right\} \nonumber\\
=&\left\{\theta \in \Theta: e^{(\theta-\hat\theta)^\intercal\Sigma^{-1}(\theta-\hat\theta)}-1\leq \frac{\chi^2_{1-\alpha,D}}{n}\right\}\nonumber\\
=&\left\{\theta: \hat\theta+\Sigma^{\frac{1}{2}}v, \quad \text {for }  \|v\|_2^2 \leq \log(1+\frac{\chi^2_{1-\alpha,D}}{n})\right\},
\end{align}
and thus
% =\{\theta \in \mathcal{U}_{data}:{L}(\delta_{\hat\theta},\theta)={l}(\delta_{\hat\theta})\}
\begin{align}\label{final2}
\Theta^* (\delta_{\hat\theta})
=&  \text{argmax}_{\theta\in\mathcal U_{data}} e^{(\theta-\hat\theta)^\intercal\Sigma^{-1}(\theta-\hat\theta)}=\big\{\theta: \hat\theta+\Sigma^{\frac{1}{2}}v, \quad \text {for }  \|v\|_2^2 = \log(1+\frac{\chi^2_{1-\alpha,D}}{n})\big\}.
\end{align}

% Now, we can calculate a descent direction as in the following proposition. To be specific, we call a distribution $\mu_{prop}(\theta')\in \mathcal P (\Theta)$ ``a of $\mu_{prop}$
We propose the following $\mu_{prop}$. First, we call a distribution on $\Theta$ symmetrical around $\theta\in\Theta$ if its probability density or mass function has the same value for any $\theta_1,\theta_2 \in \Theta$ such that $\theta=\frac{\theta_1+\theta_2}{2}$.
\begin{prop}\label{collec}
	Let $\mu_{prop}(d\theta')$ be any symmetrical distribution around $\hat\theta$. Given $\theta \in \Theta^* (\delta_{\hat\theta})$, we define $Y_{\theta}=(\theta-\hat\theta)^\intercal\Sigma^{-1}(\theta'-\hat\theta)$ with $\theta'\sim\mu_{prop}(d\theta')$. Suppose there exists an integrable random variable  $Y$ under the measure $\mu_{prop}$ such that $e^{2Y_{\theta}}  \leq Y $ for all $\theta \in \Theta^* (\delta_{\hat\theta})$. If, for each $\theta \in \Theta^* (\delta_{\hat\theta})$, $Y_{\theta}$ does not equal to 0 with probability 1, then \eqref{verify descent} holds and the mixture distribution produced by $\mu_{prop}(d\theta)$ would result in a descent direction on $\mathcal D_{data}(\mathbb P_{\hat\theta})$.
\end{prop}

One can check that any Gaussian distribution with mean $\hat\theta$ satisfies the conditions of Proposition \ref{collec}, and so does any $\mu_{prop}(d\theta')$ that is discrete, symmetrical around $\hat\theta$, whose outcome directions $\theta'-\theta$ constitute a basis of $\mathbb R^D$.
% has a continuous density function, we propose a following one that is easy to simulate. Specifically, In other words, in order to simulate $\theta\sim\mu_{prop}(d\theta)$,  uniformly on the surface of the $D$ dimension ball with unit radius representing each coordinate of $\eta$ and Then, we set $\theta=\hat{\theta}+\sqrt{\frac{\chi^2_{1-\alpha,D}}{n}}\cdot\Sigma^{-1/2}\cdot\eta$ and obtain $\theta\sim\mu_{prop}(d\theta)$. that satisfies the condition of Proposition \ref{collec} , it satisfies the conditions (see, e.g., \cite{Muller:1959:NMG:377939.377946})
Alternately, we also consider the following continuous $\mu_{prop}$. We set $\theta'\sim\hat{\theta}+\sqrt{\frac{\chi^2_{1-\alpha,D}}{n}}\cdot\Sigma^{1/2}\eta$ where $\eta$ is a random vector uniformly distributed on the surface of the $D$-dimension unit ball. Note that this $\theta'$ can be efficiently simulated by sampling $D$ independent standard Gaussian random variables and scaling their norm to unit length to obtain $\eta$. While this $\mu_{prop}$ can be readily checked to satisfy the conditions in Proposition \ref{collec}, we also provide an alternate proof on the validity of this $\mu_{prop}$ in achieving a descent direction in \lemref{move} in the Appendix, as results proven therein provide important reference to calculations in numerical experiments regarding $\mu_{prop}$.

Next, we discuss the computation of $\mathcal D_{data}(\mathbb P_0)$ for a given $\mathbb P_0$. First, we call a random variable $Y$ on $\mathcal Y\subset \mathbb R^k$ rotationally invariant if $ Y {\buildrel \mathcal D \over =} Q^\intercal Y$ for any rotational matrix $Q\in\mathbb R^{k \times k}$. Using this notion, the following shows how one can reduce the $D$-dimensional maximization problem in the definition of $\mathcal D_{data}(\mathbb P_0)$ into a one-dimensional problem.
\begin{prop}\label{cod}
Given a nominal distribution $Y\sim\mathbb P_0$ and a multivariate Gaussian family with known covariance $\Sigma$ denoted $\mathbb P_\theta=\mathcal N(\theta,\Sigma)$. If the nominal distribution $Y\sim\mathbb P_0$ satisfies the condition that the random variable $Z=\Sigma^{-1/2}(Y-\hat\theta)$ is rotationally invariant, then for any $\theta_1,\theta_2$ satisfying $(\theta_1-\hat\theta)^\intercal\Sigma^{-1}(\theta_1-\hat\theta)=(\theta_2-\hat\theta)^\intercal\Sigma^{-1}(\theta_2-\hat\theta)$, we have
\begin{equation}\label{codp1}
    \chi^2(\mathbb P_0,\mathbb P_{\theta_1})=\chi^2(\mathbb P_0,\mathbb P_{\theta_2}).
\end{equation}
 Thus, for $\mathcal D_{data}(\mathbb P_0)=\max_{\theta\in\mathcal U_{data}}\chi^2(\mathbb P_0,\mathbb P_\theta)$ with $\mathcal U_{data}=\{\theta \in \Theta: (\theta-\hat\theta)^\intercal\Sigma^{-1}(\theta-\hat\theta)\leq \lambda\}$ as in \eqref{final1}, we have
 \begin{equation}\label{codp2}
\mathcal D_{data}(\mathbb P_0)=\max_{0\leq t \leq 1}\chi^2(\mathbb P_0,\mathbb P_{(1-t)\hat\theta+t\theta^\star}),
\end{equation}
given any $\theta^\star$ satisfying $(\theta^\star-\hat\theta)^\intercal\Sigma^{-1}(\theta^\star-\hat\theta)=\lambda$.
\end{prop}

\begin{prop}\label{cod1}
Given $0 \leq t \leq 1$ and $\mu_{prop}(d\theta)=\hat\theta+\sqrt{\frac{\chi^2_{1-\alpha,D}}{n}}\cdot \Sigma^{1/2}\eta$, where $\eta$ is a random vector uniformly distributed on the surface of the $D$-dimension unit ball, the nominal measure $\mathbb P_t$ with density
\begin{equation*}
p_{t}(y)=\int_{\Theta} p(y;\theta')  ((1-t)\delta_{\hat\theta}+t\mu_{prop})(d\theta')=(1-t)\mathbb P_{\hat\theta}+t\int_{\Theta}p(y;\theta')\mu_{prop}(d\theta'),
\end{equation*}
satisfies the conditions in  Proposition \ref{cod}.
\end{prop}
Therefore, in computing $\mathcal D_{data}(\mathbb P_0)$ derived from the proposed distribution $\mu_{prop}(d\theta)=\hat\theta+\sqrt{\frac{\chi^2_{1-\alpha,D}}{n}}\cdot \Sigma^{1/2}\eta$,
using Propositions \ref{cod} and \ref{cod1} we can change the domain of the involved maximization from $\Theta\subset\mathbb R^D$ into $\mathbb R$, leading to a substantial reduction in the search space and a tractable problem.

\subsection{Enlarging Mixture Variability}
\label{sec:mix var}
Our next proposal is to consider a continuous mixing distribution $\mu_r(d\theta')$ on $\Theta$ where $r\geq0$ controls the variability of the distribution, so that $r=0$ corresponds to $\delta_{\hat\theta}$. Here, we can parametrize the density of the generating distribution as
\begin{equation}
p_r(y)=\int_{\Theta} p(y;\theta')\mu_{r}(d\theta'),\label{mix var}
\end{equation}
and our search direction is along $r$ starting from $r=0$.
% Still under the framework of Multivariate Gaussian, to connect our example in the previous section more closely with the methods in the Bayesian setting, we introduce an alternate construction for our mixture density. To be specific,
% for measures $\{\mu_r \}_{r \geq 0}$ on $\Theta$.
We propose two possible ways to define $\mu_r(d\theta')$. First is to let $\mu^1_r(d\theta')$ follow the distribution of $\theta' \sim \hat\theta+\Sigma^{\frac{1}{2}}\cdot \eta_{\sqrt{r}}$ where $\eta_{\sqrt{r}}$ is the uniform distribution inside the $D$-dimensional unit ball with radius $\sqrt{r}$. Second is to let $\mu^2_{r}(d\theta')$ follow $\mathcal N(\hat\theta,r\Sigma)$. The second approach in particular can  be intuited as the posterior distribution of the parameter from a Bayesian perspective. In both cases, we notice that letting $r=0$ would recover the original baseline distribution $p(y;\hat\theta)$.

To analyze these schemes, we abuse notation slightly and now define $L: \mathbb R^+ \times \Theta \rightarrow \mathbb R$ to be
\begin{equation}\label{def12}
L(r,\theta)\triangleq\int_{\mathcal{Y}}\frac{(p(y; \theta))^2}{p_r(y)} dy,
\end{equation}
and
\begin{equation}\label{def123}
l(r)\triangleq\sup\limits_{\theta\in\mathcal{U}_{data}}L(r,\theta).
\end{equation}
We show that increasing $r$ to positive values would produce a descent direction for $l(r)$ at $r=0$, when the underlying distribution is Gaussian. Recall that in this case $\Theta^*(\delta_{\hat\theta})$ can be expressed by \eqref{final2}. As $l(r)$ is not necessarily convex in this situation, we use a generalized version of Danskin's Theorem \cite{clarke1975generalized} for non-convex problems to get the following result:
% . However, we show in that $l(r)$ that we can use a generalized version of the Danskin theorem to claim:

\begin{theorem}\label{more}
    With $l(r)$ and $L(r,\theta)$ defined in \eqref{def12} and \eqref{def123}, and $p(y;\theta)$ multivariate Gaussian with mean $\theta$ and known positive definite covariance $\Sigma$, we have
    \begin{align}\label{dproof}
l^+(0)=\lim_{r\downarrow 0} \frac{l(r)-l(0)}{r}= (1+\frac{\chi^2_{1-\alpha,D}}{n})\cdot\lim_{r\downarrow 0}\frac{1}{r}\Big(1-\inf_{\theta\in \Theta^*(\delta_{\hat\theta})}\mathbb E_{\theta'\sim\mu_{r}}[e^{2(\theta-\hat\theta)^\intercal\Sigma^{-1}(\theta'-\hat\theta)}]\Big)
\end{align}
% similarly as in \eqref{cproof}.
\end{theorem}
The proof is in Appendix \ref{sec:proofs}. With Theorem \ref{more}, we can show that both $\mu_r^1$ and $\mu_r^2$ proposed above are valid choices to produce descent directions. Moreover, we can also show that they allow tractable computation of $\mathcal D_{data}(\mathbb P_0)$. These are depicted as follows.
%  Now, we introduce two following lemma. for any $\theta \in \Theta^*(\delta_{\hat\theta})$, both $\lim_{r\downarrow 0}\frac{1}{r}\Big(1-\mathbb E_{\theta'\sim\mu^1_{r}}[e^{2(\theta-\hat\theta)^\intercal\Sigma^{-1}(\theta'-\hat\theta)}]\Big)<0$ and $\lim_{r\downarrow 0}\frac{1}{r}\Big(1-\mathbb E_{\theta'\sim\mu^2_{r}}[e^{2(\theta-\hat\theta)^\intercal\Sigma^{-1}(\theta'-\hat\theta)}]\Big)<0$. Hence

\begin{coro}\label{variability}
	Under the assumptions in Theorem \ref{more},  $l^+(0)<0$ for both $\mu_r^1$ and $\mu_r^2$.
\end{coro}
\begin{coro}\label{cod2}
Given $r\geq0$ and $\mu_{prop}$ being $\mu_{r}^1(d\theta)$ or $\mu_{r}^2(d\theta)$, the nominal measure $\mathbb P_r$ with density given by \eqref{mix var} satisfies the conditions in Proposition \ref{cod}.
% \begin{equation*}
% p_{t}(y)=\int_{\Theta} p(y;\theta')  ((1-t)\delta_{\hat\theta}+t\mu_{prop})(d\theta')=(1-t)\mathbb P_{\hat\theta}+t\int_{\Theta}p(y;\theta')\mu_{prop}(d\theta'),
% \end{equation*}
\end{coro}

The proofs of Corollaries \ref{variability} and \ref{cod1} are in Appendix \ref{sec:proofs}.

% These two lemmas combined with \eqref{dproof} indicate that increasing $r$ to positive value would produce a descent direction for $l(r)$ at $r=0$.
% we provide a self-contained procedure description which summarizes all our previous discussions on solving chance-constrained problem

\subsection{Numerical Demonstrations}\label{sec:numerics mixing}
To confirm our findings in Section \ref{sec:line search} and \ref{sec:mix var}, we perform several numerical experiments. Consider $\mathbb P_\theta$ to be multivariate Gaussian $\mathcal N(\theta,I_D)$ with $k=D=10$. We set $\epsilon=\alpha=0.05$ while $\beta=0.01$ and data size $n=10$ or 5. Notice in this case, the dimension $D$ is high but the available sample $n$ is low and we would actually need $N_{exact}=371$ data points to perform standard SO. Based on our discussion, we compare three choices of $\mu_{prop}$:
\begin{itemize}
    \item $\mu_1=\delta_{\hat\theta}$, the point mass at $\hat\theta$.
    \item $\mu_{2}\sim\hat\theta+\sqrt{\frac{\chi^2_{1-\alpha,D}}{n}}\cdot \eta$, where $\eta$ is the uniform random vector on the surface of a $D$-dimension unit ball, discussed in Section \ref{sec:line search}.
\item $\mu_3\sim \mathcal N(\hat\theta,I_D/n)$, the Gaussian distribution with mean $\hat\theta$ and covariance matrix $I_D/n$, discussed in Section \ref{sec:mix var}.
    \end{itemize}
For $\mu_1$, $\mu_2$ and $\mu_3$, the calculation of $\mathcal D_{data}(\mathbb P_0)$ is tractable. We leave the details in the Appendix as remarks following Lemma \ref{move} and summarize the results in Table \ref{10dim1} and \ref{10dim2}. We use $N$ to denote the number of Monte Carlo samples needed. Moreover, we use both algorithms Extended SO and Extended FAST discussed in Section 5 for demonstration. As we can see, the decrease in $N$ under a better sampling distribution can be considerable, down to less than a third compared to using the baseline in some cases. Mixing with a proposed uniform distribution ($\mu_2$) appears to reduce $N$ more than applying a Gaussian mixture ($\mu_3$). As a side note, we also observe Extended FAST requires significantly less sample size than Extended SO in this example.
% , exhibiting better suitability in these situations.

\begin{table}[h]
	\centering
	\caption{Comparisons among choices of $\mathbb{P}_0$ for 10 dimensional multivariate Gaussian when $n=5$.}
	\begin{tabular*}{0.85\textwidth}{@{}l @{\extracolsep{\fill}} *{4}{c} @{}}
		
		\cmidrule{1-5}
		& $\mathcal{D}_{data}(\mathbb{P}_0)$  & $\delta_\epsilon$ & $N$ for Extended SO & $N$ for Extended FAST \\
		\midrule
		$\mu_1 (\delta_{\hat{\theta}})$ & 37.9161   & $6.5766\times 10^{-5}$ & 285601 & 70221 \\
		$\mu_{2}$ & 11.0368  & $2.2454 \times 10^{-4}$ & 83649 & 20707\\
		$\mu_{3}$ & 14.7391  & $1.6850\times 10^{-4}$ &  111465 & 27528 \\
		\bottomrule\label{10dim1}
	\end{tabular*}
\end{table}
\begin{table}[h]
	\centering
	\caption{Comparisons among choices of $\mathbb{P}_0$ for 10 dimensional multivariate Gaussian when $n=10$.}
	\begin{tabular*}{0.85\textwidth}{@{}l @{\extracolsep{\fill}} *{4}{c} @{}}
		
		\cmidrule{1-5}
		& $\mathcal{D}_{data}(\mathbb{P}_0)$  & $\delta_\epsilon$ & $N$ for Extended SO & $N$ for Extended FAST \\
		\midrule
		$\mu_1 (\delta_{\hat{\theta}})$ & 5.2383   & $4.6857\times 10^{-4}$ & 40081 & 10026 \\
		$\mu_{2}$ & 3.3139  & $7.3298 \times 10^{-4}$ & 25621 & 6481 \\
		$\mu_{3}$ & 3.7926  & $6.4275\times 10^{-4}$ & 29219 & 7363 \\
		        \bottomrule\label{10dim2}
	\end{tabular*}
\end{table}

\section{Procedural Description}\label{fastsec}
This section presents our procedures to find solutions for CCP \eqref{main_problem} using SO-based methods, when the direct use of data $\xi_1,...,\xi_n$ from $\mathbb P$ is possibly insufficient to achieve feasibility with a given confidence. Algorithm \ref{esgnew}, which we call ``Extended SO", first presents the basic and most easily applicable procedure arising from Corollary \ref{para nonpara}. Notice that, given an overall target confidence level, say $c$, we have flexibility in choosing $\alpha$ and $\beta$ such that $\alpha+\beta=c$. In our experiments, we simply choose $\alpha=\beta=\frac{c}{2}$. However, if the required  confidence level is high, it is more beneficial to choose a relatively small $\beta$, since the required Monte Carlo sample size depends only logarithmically on $\beta$ (i.e., the required sample size for SO is of order $\log\frac{1}{\beta}$) \cite{campi2008exact}. On the other hand, as the confidence level $1-\alpha$ grows higher, the size of uncertainty set $\mathcal U_{data}$ would grow and cause the tolerance level $\epsilon$ for the SO (under the baseline $\mathbb P_0$) to decrease. Here, the dependence of Monte Carlo sample size on $\epsilon$ is less favorable, typically of order $\frac{1}{\epsilon}$ \cite{campi2008exact}.

\begin{algorithm}[H]
	\caption{\textit{Extended SO} to obtain a feasible solution $\hat x$ for \eqref{main_problem} with asymptotic confidence $1-\alpha-\beta$}\label{esgnew}
	\begin{algorithmic}[1]
% 		\Procedure{Extended\_``\textit{SO-method}''\_Center} {$n,\beta,\alpha,\epsilon,\phi,\mathcal P_{para} $} with insufficient samples $\xi_1,\xi_2,...,\xi_n$ from $\mathbb P$ such that $n<N_{exact}(\epsilon,\beta,d)$, the minimal data size needed for the SO-based method to achieve $\epsilon$-feasibility with $1-\beta$ confidence. Furthermore, $\alpha>0$ specifies a confidence level for uncertainty set, $\phi$ function denotes the divergence type for the uncertainty set and $\mathcal P_{para}=\{\mathbb P_{\theta}\}_{\theta\in\Theta\subset \mathbb R^D}$ is the parametric information.
        \State\textbf{Inputs:} data points $\xi_1,\ldots,\xi_n$, a $\phi$-divergence, parametric information $\mathcal P_{para}=\{\mathbb P_{\theta}\}_{\theta\in\Theta\subset \mathbb R^D}$.
		\State Find the MLE $\hat\theta$ from the data $\xi_1,\ldots,\xi_n$ for parameter $\theta$.
		\State Set $\lambda \gets \frac{\phi''(1)\chi^2_{1-\alpha,D}}{2n}$ where $\chi^2_{1-\alpha,D}$ is the $1-\alpha$ quantile of a $\chi^2_D$ distribution.
		\State Set $\delta_\epsilon\gets\epsilon'(\epsilon,\lambda,\phi)$ where $\epsilon'$ is defined in \eqref{adjusted tolerance}.
        \State Set $N\gets N_{exact}(\delta_\epsilon,\beta,d)$ where $N_{exact}$ is defined in \eqref{gamma_function1}.
		\State Generate $\xi^{MC}_1,...,\xi^{MC}_{N}$ from $\mathbb P_{\hat\theta}$ to construct \eqref{Monte_SG} and obtain a solution $\hat x$.
% 		\State Output $\hat x$.
% 	\EndProcedure
	\end{algorithmic}
\end{algorithm}

There are several variants of Algorithm \ref{esgnew}. First, we have discussed the use of plain SO and that the required sample size is \eqref{gamma_function1}, while on the other hand, as mentioned at the end of Section \ref{sec:sampling}, we can use other variants of SO such as FAST that requires a smaller sample size for either the data or the Monte Carlo samples we generate. In the case of FAST, we would have $N_{exact}(\epsilon,\beta,d)=20d+\frac{1}{\epsilon}\log \frac{1}{\beta}$, as suggested by \cite{care2014fast}. Thus, a variant of Algorithm \ref{esgnew} is to replace $N_{exact}$ with this latter quantity, and replace \eqref{Monte_SG} with the FAST procedure in \cite{care2014fast} for the last step of Algorithm \ref{esgnew} (we call this algorithm ``Extended FAST" which will also be used in the next section).

% is taken from \cite{jiang2016data} and we have introduced it at \eqref{adjusted tolerance}.
The explicit expression for $\epsilon'(\epsilon,\lambda,\phi)$ for different $\phi,\epsilon$ and $\lambda$ can be found in \cite{jiang2016data}. For example, if we choose $\phi=(x-1)^2$ which corresponds to the $\chi^2$-distance, then for $\epsilon<1/2$, we have $\epsilon'=\max\{0,\epsilon-\frac{\sqrt{\lambda^2+4\lambda(\epsilon-\epsilon^2)}-(1-2\epsilon)\lambda}{2\lambda+2}\}$. We can also replace $\epsilon'(\epsilon,\lambda,\phi)$ by any $\delta_\epsilon$ that achieves $M(\mathbb P_{\hat\theta},\{\mathbb Q:d_\phi(\mathbb P_{\hat\theta},\mathbb Q)\leq\lambda\},\delta_\epsilon)\leq\epsilon$. In Appendix \ref{relax}, we derive a self-contained easy upper bound for $M(\mathbb P_{\hat\theta},\{\mathbb Q:d_\phi(\mathbb P_{\hat\theta},\mathbb Q)\leq\lambda\},\delta)$ in the case of $\chi^2$-distance and use it to find such a $\delta_\epsilon$. This easy computation of $\delta_\epsilon$ will also be used in our numerics in the next section.

Section \ref{sec:reduction} has investigated some proposals to improve the generating distributions. Algorithm \ref{esgnew1} depicts these proposals in a general form. The main difference of Algorithm \ref{esgnew1} compared to Algorithm \ref{esgnew} is the introduction of $\mathcal D_{data}(\mathbb P_0,\phi)$ that one can attempt to minimize over a class of generating distribution $\mathbb P_0$ or evaluate for trial-and-error choices of $\mathbb P_0$, so that at the end we have $\mathcal D_{data}(\mathbb P_0,\phi)<\phi''(1)\chi^2_{1-\alpha,D}/(2n)$. As discussed in Section \ref{sec:reduction}, using this $\mathbb P_0$ allows us to obtain a smaller Monte Carlo size requirement than simple relaxation of the parametric constraint. Sections \ref{sec:line search} and \ref{sec:mix var} describe the possibilities of achieving such a reduction, in the case of Gaussian underlying distributions and using $\chi^2$-distance. Note that, just like in Algorithm \ref{esgnew}, we can consider other variants such as incorporating FAST and using alternate bounds for $M$ instead of $\epsilon'$, by undertaking the same modifications as in Algorithm \ref{esgnew}.

\begin{algorithm}[H]
	\caption{\textit{Extended SO with improved generating distribution} to obtain a feasible solution $\hat x$ for \eqref{main_problem} with asymptotic confidence $1-\alpha-\beta$}\label{esgnew1}
	\begin{algorithmic}[1]
% 		\Procedure{Extended\_``\textit{SO-method}''\_Center} {$n,\beta,\alpha,\epsilon,\phi,\mathcal P_{para} $} with insufficient samples $\xi_1,\xi_2,...,\xi_n$ from $\mathbb P$ such that $n<N_{exact}(\epsilon,\beta,d)$, the minimal data size needed for the SO-based method to achieve $\epsilon$-feasibility with $1-\beta$ confidence. Furthermore, $\alpha>0$ specifies a confidence level for uncertainty set, $\phi$ function denotes the divergence type for the uncertainty set and $\mathcal P_{para}=\{\mathbb P_{\theta}\}_{\theta\in\Theta\subset \mathbb R^D}$ is the parametric information.
        \State\textbf{Inputs:} data points $\xi_1,\ldots,\xi_n$, a $\phi$-divergence, parametric information $\mathcal P_{para}=\{\mathbb P_{\theta}\}_{\theta\in\Theta\subset \mathbb R^D}$.
		\State Find the MLE $\hat\theta$ from the data $\xi_1,\ldots,\xi_n$ for parameter $\theta$.
		\State Set $\lambda \gets \frac{\phi''(1)\chi^2_{1-\alpha,D}}{2n}$ where $\chi^2_{1-\alpha,D}$ is the $1-\alpha$ quantile of a $\chi^2_D$ distribution.
%         \State Set $\mathcal U_{data}=\{\theta:\chi^2(\mathbb P_{\hat\theta},\mathbb P_{\theta})\leq\lambda\}$ which can be written as \eqref{final1}.
%         \State\lambda \gets \frac{\phi''(1)\chi^2_{1-\alpha,D}}{2n}$ where $\chi^2_{1-\alpha,D}$ is the $1-\alpha$ quantile of a $\chi^2_D$ distribution.
        \State Obtain $\mathbb P_0$ by minimizing $\mathcal D_{data}(\mathbb P_0,\phi)$ defined in \eqref{transform para to nonpara} over a class of distributions or simple trial-and-error search so that $\mathcal D_{data}(\mathbb P_0,\phi)<\lambda$.
        \State Set $\delta_\epsilon\gets\epsilon'(\epsilon,\mathcal D_{data}(\mathbb P_0,\phi),\phi)$ where $\epsilon'$ is defined in \eqref{adjusted tolerance}.
        \State Set $N\gets N_{exact}(\delta_\epsilon,\beta,d)$ where $N_{exact}$ is defined in \eqref{gamma_function1}.
		\State Generate $\xi^{MC}_1,...,\xi^{MC}_{N}$ from $\mathbb P_0$ to construct \eqref{Monte_SG} and obtain a solution $\hat x$.
% 		\State Output $\hat x$.
% 	\EndProcedure
	\end{algorithmic}
\end{algorithm}

\section{Numerical Experiments}\label{sec:numerics}
This section presents some numerical examples to support our theoretical findings and illustrate the performance of our proposed procedures for data-driven CCPs. We focus on Algorithm \ref{esgnew} (Extended SO) and its FAST variant discussed in Section \ref{fastsec} (Extended FAST). We consider both single and joint CCPs (i.e., one and multiple inequalities respectively in the safety condition of the probability) as well as quadratic optimization problems. Moreover, we compare numerically with methods of robust optimization (RO) in \cite{convex_app,BEN:09}. The experimental outputs that we report include:
\begin{itemize}
% 	\item For every experiment in each problem, We obtain an optimal solution $\hat{x}$ by solving Extended\_SO in algorithm \ref{esgnew} combined with algorithm \ref{esgnew2} through results in Appendix \ref{relax}. In particular, we used $\chi^2$ distance.
	\item Under each setting, we repeat the experiment 1000 times with new data generated each time. For the solution $\hat{x}$ obtained in each trial from a given algorithm, we evaluate the violation probability $V(\hat{x},\mathbb{P})$ under the true probability measure $\mathbb{P}$ (under $\theta_{true}$) either through exact calculation or Monte Carlo simulation with sample size 10000. Moreover, using the empirical distribution for the violation probabilities, we report $\hat{\epsilon}$ as the average violation probability $V(\hat{x},\mathbb{P})$ as well as $Q_{95}$, the 95-percentile. Finally, we report and compare ``$f_{val}$'', the average objective value for the optimization problem across all 1000 runs.
	
	\item We fix $\alpha=0.05$ and $\beta=0.01$ across different values of $\epsilon$ and $d$. However, when we compare our methods with robust optimization approaches, we set $\alpha=0.05$ and $\beta=0.001$, since RO approaches essentially guarantee $\beta=0$. On the other hand, the sample size chosen for FAST is taken with default values $N_1=20d$ in stage 1 and $N_2=\frac{\log \beta-\log(B_\epsilon^{N_1,d})}{\log(1-\epsilon)}$ in stage 2 as discussed in \cite{care2014fast}.
	
	\item For given $\epsilon$ and $d$, we denote $N_{exact}$ as the required sample size if we can directly sample from $\mathbb{P}$ and use standard SO. We denote $n$ as the available data size ($n<N_{exact}$) and $N$ as the Monte Carlo size needed for the our DRO-based methods. In DRO-based methods, we fix our generating distribution $\mathbb{P}_0$ as $\mathbb{P}_{\hat{\theta}}$ and use the $\chi^2$-distance across the experiments.
% \item  (Extended SO or Extended SO with FAST with algorithm) given the real data size $n<N_{true}$.
\end{itemize}

\subsection{Single Linear Chance Constraint Problem}

We first consider a single linear CCP
\begin{equation}\label{single_ccp}
\begin{aligned}
& \underset{x \in \mathcal{X} \subseteq \mathbb{R}^d}{\text{min}}
& & c^Tx \\
& \quad\text{s.t.}
& & \mathbb{P}((a+\xi)^{T}x \leq b) \geq 1-\epsilon,  x\geq 0
\end{aligned}
\end{equation}
where $x\in\mathbb{R}^d$ is the decision variable, $a,c\in \mathbb{R}^d$ and $b\in \mathbb{R}$ are fixed and $\xi \in \mathbb{R}^d$ is a random vector following some parametric distribution. We fix  $a=[5,5,... ,5]\in \mathbb R^{d}$, $b=5$ and $c=[-1,-1,... ,-1]\in\mathbb R^d$ and the problem would have a non-empty feasible region with high probability for $\xi$ considered here. Moreover, a robustly feasible point for FAST \cite{care2014fast} is chosen to be $\bar{x}=\bold{0}\in \mathbb R^{d}$ and an explicit $\mathcal U_{data}$ is constructed as \eqref{parametric uncertainty1} for our DRO.
\subsubsection{ Multivariate Gaussian } We conduct experiments when $\xi\sim\mathcal N(\theta, \Sigma)$ with fixed but a priori randomly generated positive definite covariance matrix $\Sigma\in\mathbb R^{d\times d}$ and unknown $\theta\in\mathbb R^d$. Due to the normality of $\xi$, for any given $\theta$, we can reformulate the chance constraint exactly as a second-order cone constraint, which can be robustified straightforwardly in the ambiguous chance constraint case. The underlying true parameter is taken to be $\theta_{true}=\bold{0}\in \mathbb R^{d}$ and the results are summarized in Table \ref{single_results_gau} and \ref{single compare}. \captionsetup{skip=0.33\baselineskip}

% The following macro gets used three times in the body of the document
\newcommand\tabbody{%

	& \multicolumn{6}{c}{} \\
	\cmidrule{1-7}
	& $\epsilon=0.1$ & $\epsilon=0.1$ & $\epsilon=0.1$ & $\epsilon=0.05$ & $\epsilon=0.05$&$\epsilon=0.05$ \\
	& $d=5$ & $d=10$ & $d=20$ & $d=5$ & $d=10$&$d=20$ \\
	\midrule
	$n$ & 50 & 80& 200 & 50 & 80 &200\\
	$N_{exact}$ & 113 & 183 & 312 & 229 & 371 &631\\
	$N$ & 449 & 743 & 1016 &1443& 2349 &3118\\
	$\hat\epsilon$ & 0.0050 & 0.0041 & 0.0041 & 0.0015 & 0.0015 &0.0014\\
	$Q_{95}$& 0.0136 & 0.0103 & 0.0088 & 0.0045 & 0.0037 &0.0031\\
	$f_{val}$& -0.7577 & -0.7447 & -0.7360 & -0.7353 & -0.7243 &-0.7128\\

	\bottomrule}

\begin{table}[h]
	\centering
	\caption{Single linear CCP under Gaussian with unknown mean for different $\epsilon$ and $d$.}
	\begin{tabular*}{0.85\textwidth}{@{}l @{\extracolsep{\fill}} *{6}{c} @{}}
		\tabbody\label{single_results_gau}
	\end{tabular*}
	
\end{table}

\captionsetup{skip=0.66\baselineskip}

% The following macro gets used three times in the body of the document
\newcommand\tabbbody{%
	\toprule
% 	\cmidrule{1-3}\\
	& RO & Extended SO & Extended FAST  \\
	\midrule
	$n$ & 80 & 80& 80 &\\
	$N_{exact}$ & NA & 447 & 447 \\
	$N$ & NA & 2887 & 1079 \\
	$\hat{\epsilon}$ & 0.0180 & 0.0011  & 0.00069 \\
	$Q_{95}$ & 0.0272 & 0.0029  & 0.0019 \\
	$f_{val}$ & -0.8008
	& -0.7212 & -0.7093 \\

	\bottomrule}

\begin{table}[h]
	\centering
	\caption{Comparisons for single linear CCP under Gaussian: $\epsilon=0.05$, $d=10$ and $\beta=0.001$.}
	\begin{tabular*}{0.85\textwidth}{@{}l @{\extracolsep{\fill}} *{5}{c} @{}}
		\tabbbody\label{single compare}
	\end{tabular*}
	
\end{table}

\subsubsection{Exponential Distribution} We conduct experiment when each coordinate $\xi_i$ of  $\xi\in\mathbb R^d$ independently follows exponential distribution with rate $\lambda_i$. Since $\xi$ is no longer Gaussian and the domain of the moment generating moment function for exponential distribution depends on $\lambda=(\lambda_1,\ldots,\lambda_d)$, for convenience we use RO constructed from a convex approximation using Chebyshev's inequality:
\begin{equation*}
    \mathbb{P}_{\lambda}\Big(\xi^{T}x-\sum_{i=1}^d \frac{x_i}{\lambda_i} > \epsilon^{-1/2}\sqrt{Var(\xi^Tx)}\Big) \leq \epsilon
\end{equation*}
which, combined with $\mathcal U_{data}$ as in \eqref{parametric uncertainty1}, reduces the ambiguous chance constraint into a robust conic quadratic constraint
\begin{equation*}
    \epsilon^{-1/2}\sqrt{\sum_{i=1}^d(\frac{x_i}{\lambda_i})^2}+a^T x+ \epsilon^{-1/2}\sum_{i=1}^d\frac{x_i}{\lambda_i}-b\leq 0, \text{} \quad \forall \lambda:\sum_{i=1}^d(1-\frac{\lambda_i}{\hat\lambda_i})^2 \leq \frac{\chi^2_{1-\alpha,d}}{n},
\end{equation*}
The above can be tractably reformulated as in Section 5 of \cite{convex_app} on problems in the form of 5(b), with $\Omega=(\min_i(\hat\lambda_i) (1-\frac{\chi^2_{1-\alpha,d}}{n}))^{-1}$ where $\hat\lambda_i$ represents the MLE estimate of $\lambda_i$. Finally, the underlying true parameters are taken as $\lambda_i=1, \forall i$, and results are summarized in Table \ref{single compare_exp}.

\captionsetup{skip=0.66\baselineskip}

% The following macro gets used three times in the body of the document
\newcommand\tabokbody{%
	\toprule
% 	\cmidrule{1-3}\\
	& RO & Extended SO & Extended FAST  \\
	\midrule
	$n$ & 80 & 80& 80 &\\
	$N_{exact}$ & NA & 447 & 447 \\
	$N$ & NA & 2887 & 1079 \\
	$\hat{\epsilon}$ & 0.0045 & 0.0047  & 0.0016 \\
	$Q_{95}$ & 0.0094 & 0.0100  & 0.0050 \\
	$f_{val}$ & -0.6978
	& -0.6981 & -0.6701 \\

	\bottomrule}

\begin{table}[h]
	\centering
	\caption{Comparisons for single linear CCP under Exponential: $\epsilon=0.05$, $d=10$ and $\beta=0.001$.}
	\begin{tabular*}{0.85\textwidth}{@{}l @{\extracolsep{\fill}} *{5}{c} @{}}
		\tabokbody\label{single compare_exp}
	\end{tabular*}
	
\end{table}

From the results of the experiments, we can see the vast majority of solutions produced by three methods satisfy statistical feasibility. In fact, all methods are conservative with respect to the violation probability $\epsilon$, although some are more conservative than the other. In particular, when $\xi$ is  Gaussian, RO takes advantage of an exact formulation to produce less conservative solution with lower objective value (closer to the optimal value). This can be seen in Table 4, where $\hat \epsilon=0.018$ $f_{val}=-0.80$ for RO and $\hat \epsilon=0.0011$ $f_{val}=-0.72$ only for Extended SO. When $\xi$ is no longer Gaussian, RO appears to produce similar-quality solutions as Extended SO in terms of feasibility or optimality. For example in Table 5, we have $\hat \epsilon=0.0045$ $f_{val}=-0.6978$ for RO and $\hat \epsilon=0.0047$ $f_{val}=-0.6981$ for Extended SO. Note that while the validity of RO depends crucially on the applicability and accuracy of convex approximation, the validity of Extended SO or Extended FAST is not restricted by the distributions of $\xi$, and they also do not require intensive, case-specific analysis as RO. In general, we observe consistent performances of our methods in both experiments.

\subsection{Joint Linear Chance Constraint Problem}
Next, we consider a joint chance-constrained linear problem:
\begin{equation}\label{joint_ccp}
\begin{aligned}
& \underset{x \in \mathcal{X} \subseteq \mathbb{R}^d}{\text{min}}
& & c^Tx \\
& \quad\text{s.t.}
& & \mathbb{P}((A+\Xi) x \leq b) \geq 1-\epsilon,  x\geq 0
\end{aligned}
\end{equation}
where $x\in\mathbb{R}^d$ is the decision variable, $A \in \mathbb R^{m\times d},c\in \mathbb{R}^d$ and $b\in \mathbb{R}^{m}$ are fixed and $\Xi \in \mathbb{R}^{m\times d}$ is a random matrix following some parametric distribution. We set $c$, each row of $A$ and $b$ to be the same as in the single linear CCP. We treat $\Xi \in \mathbb R^{m\times d}$ as a matrix concatenated from a random vector $\xi\in\mathbb R^{md} \sim \mathcal N(\theta,\Sigma)$ with fixed but a priori randomly generated positive definite covariance matrix $\Sigma\in\mathbb R^{md\times md}$ and unknown $\theta\in\mathbb R^{md}$. To solve RO, we use Bonferroni's inequality as in \cite{nemirovski2007convex} to first divide the violation probability $\epsilon$ uniformly across $m$ individual chance constraints and then follow the procedure as in single linear CCP. The results are summarized in Table \ref{joint_compare_exp}.

\captionsetup{skip=0.66\baselineskip}

% The following macro gets used three times in the body of the document
\newcommand\taboksbbody{%
	\toprule
% 	\cmidrule{1-3}\\
	& RO & Extended SO & Extended FAST  \\
	\midrule
	$n$ & 80 & 80& 80 &\\
	$N_{exact}$ & NA & 291 & 291 \\
	$N$ & NA & 2388 & 1214 \\
	$\hat{\epsilon}$ & 0.0003 & 0.0012  & 0.0226 \\
	$Q_{95}$ & 0.0007 & 0.0033  & 0.0564 \\
	$f_{val}$ & -0.6448
	& -0.6626 & -0.6466 \\

	\bottomrule}

\begin{table}[h]
	\centering
	\caption{Comparisons for Joint linear CCP under Gaussian: $\epsilon=0.05$, $m=3$, $d=10$ and $\beta=0.001$.}
	\begin{tabular*}{0.85\textwidth}{@{}l @{\extracolsep{\fill}} *{5}{c} @{}}
		\taboksbbody\label{joint_compare_exp}
	\end{tabular*}
	
\end{table}

In this joint linear example, Extended FAST provides the least conservative solution in terms of the achieved tolerance level ($\hat \epsilon=0.0226$, which is closer to $0.05$, compared to $0.0003$ in RO and $0.0012$ in Extended SO), and Extended SO is the least conservative in terms of the objective value ($f_{val}=-0.6626$ compared to $-0.6448$ in RO and $-0.6466$ in Extended FAST). RO appears to be the most conservative in terms of both the achieved tolerance level and objective value. Note that this occurs even though the underlying randomness is Gaussian, which allows exact reformulation in the single chance constraint case. The conservative performance here is likely (and unsurprisingly) due to the crude Bonferroni's correction. Note that other alternatives to using Bonferroni, if one considers tractable reformulation, would be to use moment-based DRO where tractability can be achieved (e.g., \cite{xie2018deterministic}). However, it is unclear if using moment-based DRO would be more or less conservative than using Bonferroni correction along with exact reformulation for the individualized constraints, which could comprise an interesting comparison for a future study. Nonetheless, our Extended SO/FAST, being purely sampled-based, avoids the additional conservativeness coming from the Bonferroni correction. However, we note that a large number of Monte Carlo samples are required due to the large size of $\mathcal U_{data}$ in this high-dimensional problem.

\subsection{ Non-Linear Chance Constrained Problems}  In this section, we conduct numerical experiments for non-linear CCP. We consider two examples. First is a quadratic objective with joint linear chance constraints, and second is a linear objective with a quadratic chance constraint, similar as the QM problem considered in \cite{lejeune2016solving}.

\subsubsection{Quadratic Objective with Joint Linear Chance Constraint} We adopt the same setup (thus the robust feasibility condition remains the same) as in \eqref{joint_ccp} except we modify the objective with a quadratic term
\begin{equation}\label{qp_ccp}
\begin{aligned}
& \underset{x \in \mathcal{X} \subseteq \mathbb{R}^d}{\text{min}}
& & \frac{1}{2}x^T H x+c^Tx \\
& \quad\text{s.t.}
& & \mathbb{P}((A+\Xi) x \leq b) \geq 1-\epsilon,  x\geq 0
\end{aligned}
\end{equation}
for a fixed but a priori randomly generated positive definite matrix $H$.  We use $\epsilon=0.05$. Results are summarized in Table \ref{ccp_qp}. As we can see, feasibility in terms of violation probability is satisfied by all methods, though RO suffers from higher conservativeness compared to Extended SO/FAST in terms of the objective value ($f_{val}=-0.48$ compared to $-0.5547$ and $-0.5476$ for Extended SO and FAST respectively). Like the previous example, this could be attributed to the Bonferroni correction used in the RO. Extended FAST gives the least conservative solution in terms of the tolerance level ($\hat \epsilon=0.0096$), using only one third of the samples compared to Extended SO (3888 vs 1384). On the other hand, Extended SO gives the least conservative solution in terms of the objective value ($f_{val}=-0.5547$).

\captionsetup{skip=0.66\baselineskip}

% The following macro gets used three times in the body of the document
\newcommand\taboksbsdbody{%
	\toprule
% 	\cmidrule{1-3}\\
	& RO & Extended SO & Extended FAST  \\
	\midrule
	$n$ & 200 & 200& 200 &\\
	$N_{exact}$ & NA & 447 & 447 \\
	$N$ & NA & 3888 & 1384 \\
	$\hat{\epsilon}$ & 0 & 0.0006  & 0.0096 \\
	$Q_{95}$ & 0 & 0.0017  & 0.0253 \\
	$f_{val}$ & -0.4800
	& -0.5547 & -0.5476 \\

	\bottomrule}

\begin{table}[h]
	\centering
	\caption{Comparisons for quadratic objective with joint linear chance constraint under Gaussian: $\epsilon=0.05$, $m=5$, $d=10$ and $\beta=0.001$.}
	\begin{tabular*}{0.85\textwidth}{@{}l @{\extracolsep{\fill}} *{5}{c} @{}}
		\taboksbsdbody\label{ccp_qp}
	\end{tabular*}
	
\end{table}

\subsubsection{Linear Objective with Quadratic Chance Constraint} We consider the following setup:
\begin{equation}\label{qm_ccp}
\begin{aligned}
& \underset{x \in \mathcal{X} \subseteq \mathbb{R}^d}{\text{min}}
& & c^Tx \\
& \quad\text{s.t.}
& & \mathbb{P}(x^T\Xi x +a^Tx \leq b) \geq 1-\epsilon,  x\geq 0
\end{aligned}
\end{equation}
We set  $\Xi=\frac{1}{m}\sum_{i=1}^{m}\xi_i \xi_i^T$ and $\xi_i\in\mathbb R^{d}\sim\mathcal N(\theta,\Sigma)$ are i.i.d. with unknown $\theta$. We set $\theta_{true}=0\in\mathbb R^d$ and consequently $m\Xi$ follows a Wishart distribution $\mathcal W(\Sigma, m)$ with $m$ degrees of freedom and covariance matrix $\Sigma$ under $\mathbb P$. We use $\epsilon=0.05$. The RO formulation for this problem is not readily available while our sampling-based methods are still directly applicable. We thus focus on evaluating the performance of Extended FAST under different hyper-parameters. The results are summarized in Table \ref{LPQCC}. As we can see, the high dimensions of the problem do not affect the sample size requirement of Extended FAST dramatically, as it increases moderately form $N=154$ when $d=5$ to $N=334$ when $d=10$ and to $N=422$ when $d=15$. Moreover, the average optimal value $f_{val}$ is around $-0.85$ and feasibility is satisfied ($\hat\epsilon$ all within $0.05$), showing the consistent effectiveness of our method.

\captionsetup{skip=0.33\baselineskip}

% The following macro gets used three times in the body of the document
\newcommand\tabfinalbody{%

	& \multicolumn{3}{c}{} \\
	\cmidrule{1-4}
	& $\epsilon=0.1, d=5,m=5$ & $\epsilon=0.05, d=10, m=10$ &$\epsilon=0.05,d=15,m=15$ \\
	\midrule
	$n$ & 80 & 200 & 300  \\
	$N_{exact}$ & 113 & 371  & 504   \\
	$N$ & 154 & 334 &  422 \\
	$\hat\epsilon$ & 0.0092 & 0.0050 & 0.0048 \\
	$Q_{95}$& 0.0263 & 0.0133 & 0.0128 \\
	$f_{val}$& -0.8393 & -0.8576 & -0.8672 \\

	\bottomrule}

\begin{table}[h]
	\centering
	\caption{Linear objective with quadratic chance constraint for different $\epsilon$, $m$ and $d$.}
	\begin{tabular*}{0.85\textwidth}{@{}l @{\extracolsep{\fill}} *{3}{c} @{}}
		\tabfinalbody\label{LPQCC}
	\end{tabular*}
	
\end{table}

\section{Conclusion}
\label{sec:conclusion}

We consider data-driven chance constrained problems with limited data. In such situation, standard approaches in SO may not be able to generate statistically feasible solutions. We investigate an approach that uses divergence-based DRO to efficiently incorporate parametric information through a data-driven uncertainty set, and subsequently uses Monte Carlo sampling to generate enough samples to handle the distributionally robust chance constraint. In this way our framework translates the data size requirement in SO into a Monte Carlo requirement, the latter could be much more abundant thanks to cheap modern computational power.

To exploit the full capability of our framework, we have investigated the optimality of the generating distribution in drawing the Monte Carlo samples in the sense of minimizing its required sample size. We have shown that, while the optimal choice is the baseline distribution in the unambiguous and nonparametric DRO cases, this natural choice can be dominated by other distributions in the parametric DRO case. We proved this by connecting the Neyman-Pearson lemma in statistical hypothesis testing to DRO and SO, which comprises the first such results of its kind as far as we know. We then studied several ways to find better generating distributions by searching for mixtures that enhance distributional variability. Lastly, we showed some numerical results to demonstrate how our approach can give rise to feasible solutions in a wide range of settings where other methods such as RO cannot be utilized directly or give more conservative solutions.
% , the effectiveness and generality of our approach.
% Moseveral aspects that Our approach makes use of a distributionally robust optimization (DRO) formulation that translates the data size requirement into a Monte Carlo sample size requirement drawn from what we call a generating distribution. We show that, while the optimal choice of this generating distribution is the one eliciting the data or the baseline distribution in a nonparametric divergence-based DRO, it is not necessarily so in the parametric case. Correspondingly, we develop procedures to obtain generating distributions that improve upon these basic choices. We support our findings with several numerical examples.

% to demonstrate the results.
% We summarize the results in Table 7.

\begin{acks}
We gratefully acknowledge support from the National Science Foundation under grants CMMI-1542020, CAREER CMMI-1653339/1834710 and IIS-1849280.
% A preliminary conference version of this work has appeared in \cite{lam2018sampling}.
\end{acks}

% The acknowledgments section is defined using the "acks" environment (and NOT an unnumbered section). This ensures
% the proper identification of the section in the article metadata, and the consistent spelling of the heading.

%
% The next two lines define the bibliography style to be used, and the bibliography file.
\bibliographystyle{ACM-Reference-Format}
\bibliography{sample-base}

%
% If your work has an appendix, this is the place to put it.
\appendix
\section{Regularity Conditions to Verify Assumption \ref{con_and_nor} }\label{sec:conditions}

We consider the following conditions:

(C1) $p(x,\theta_1)=p(x,\theta_2)$ for all $x$ implies $\theta_1=\theta_2$.

(C2) $\theta_{true}$ is an inner point of $\Theta\subseteq \mathbb R^D$.

(C3) The support of distribution  $\{x:p(x,\theta)>0\}$ does not depend on $\theta$.

(C4) There exists a measurable function $L_1(x)$ such that $\mathbb E_{\theta_{true}}L_1^2<\infty$ and
\begin{equation}
    |\log p(x,\theta_1)-\log p(x,\theta_2)| \leq L_1(x) \|\theta_1-\theta_2\|_2
\end{equation}
for all $\theta_1,\theta_2$ in a neighborhood of $\theta_{true}$.

(C5) $I(\theta_{true})$ is non-singular.

(C6) The density family $\{{\mathbb P}_{\theta}\}_{\theta\in\Theta}$ is differentiable in quadratic mean at $\theta_{true}$, i.e., there exists a measurable function $L_2(x) : \mathcal X \rightarrow \mathbb R^D$ such that for any $h\in\mathbb R^D$ that converges to 0,

\begin{equation}
    \int \big(\sqrt{p(x,\theta_{true}+h)}-\sqrt{p(x,\theta_{true})}-\frac{1}{2} h^T L_2(x)\sqrt{p(x,\theta_{true})}\big)^2 dx=o(\|h\|_2^2).
\end{equation}

The consistency and asymptotic normality of MLE in Assumption \ref{con_and_nor} is guaranteed under conditions (C1)-(C6). See \cite{van2000asymptotic,lehmann2004elements}.

\section{Alternate Bounds Using $\chi^2$ Distance}\label{relax}
%  To provide a first intuition for using $\chi^2$ distance, we introduce the following derivation.
Consider the $\chi^2$-based uncertainty set $\mathcal U_{data}=\{\mathbb Q \in \mathcal P_{para}:\chi^2(\mathbb P_{\hat\theta},\mathbb Q)\leq \lambda\}$. Here we provide an alternate upper bound for the function $M(\mathbb P_0,\mathcal U_{data},\delta)$, which we call $\tilde M(\mathbb P_0,\mathcal U_{data},\delta)$. That is, we find $\tilde M(\mathbb P_0,\mathcal U_{data},\delta)$ that satisfies
\begin{equation*}
\sup\limits_{\mathbb Q\in\mathcal U_{data}}\mathbb Q(\xi\in A)\leq \tilde M(\mathbb P_0,\mathcal U_{data},\delta), \quad \text{for all $A$ such that $\mathbb P_0(A) \leq \delta$}.
\end{equation*}
For any $\mathbb Q$ absolutely continuous with respect to $\mathbb P_0$, we have
\begin{align}\label{Mbound}
\sup\limits_{\mathbb Q\in\mathcal{U}_{data}} \mathbb Q(\xi \in {A})&=\mathbb{P}_{0}(\xi\in {A})+\Big(\sup\limits_{\mathbb Q\in\mathcal{U}_{data}} \mathbb{Q}(\xi \in {A})-\mathbb{P}_{0}(\xi\in {A})\Big)\nonumber\\
&=\mathbb{P}_{0}(\xi\in {A})+\sup\limits_{\mathbb Q\in\mathcal{U}_{data}}\int\mathbf 1\{\xi\in A\} \big(\frac{d\mathbb{Q}}{d\mathbb P_0}-1\big) d\mathbb{P}_0(\xi)\nonumber\\
&\leq\mathbb{P}_{0}(\xi\in {A})+\sup\limits_{\mathbb Q\in\mathcal{U}_{data}}\Big(\int\mathbf{1}\{\xi \in {A}\} d\mathbb P_0(\xi) \Big)^{1/2} \cdot \Big(\int\big(\frac{d\mathbb Q}{d\mathbb P_0}-1\big)^2 d\mathbb P_0(\xi)\Big)^{1/2}\nonumber\\
& \leq \delta+\delta^{1/2}\cdot (\sup\limits_{\mathbb Q\in\mathcal{U}_{data}} \chi^2(\mathbb{P}_0,\mathbb{Q}))^{1/2},
\end{align}
where the fourth line follows from the Cauchy-Schwarz inequality. Thus, we can set
\begin{equation*}
\tilde M(\mathbb P_0,\mathcal U_{data},\delta)=\delta+\delta^{1/2}\cdot (\sup\limits_{\mathbb Q\in\mathcal{U}_{data}} \chi^2(\mathbb{P}_0,\mathbb{Q}))^{1/2}=\delta+\delta^{1/2}\cdot (\mathcal D_{data}(\mathbb P_0))^{1/2},\label{M choice}
\end{equation*}
which is non-decreasing in $\delta$. By \eqref{interim beta}, we can choose $\delta_{\epsilon}$ such that $\delta_{\epsilon}+\delta_{\epsilon}^{1/2}\left(\mathcal D_{data}(\mathbb P_0)\right)^{1/2}\leq\epsilon$, or equivalently,
\begin{equation}
\delta_{\epsilon}\leq\epsilon+\frac{\mathcal D_{data}(\mathbb P_0)}{2}-\sqrt{\epsilon\cdot\mathcal D_{data}(\mathbb P_0)+\frac{1}{4}(\mathcal D_{data}(\mathbb P_0))^2},\label{chi-square formula}
\end{equation}
by solving the quadratic equation. In the case where we relax the parametric constraint completely, we have $\mathcal D_{data}(\mathbb P_0)=\lambda$.
% In fact, using strong duality results from \cite{jiang2016data}, we are able to derive the optimal bound using just $\chi^2$ distance as the constriant on uncertainty sets, implying:
% \begin{equation}\label{opt}
% \delta_\epsilon \leq \epsilon-\frac{\sqrt{(\mathcal D_{data}(\mathbb P_0))^2+4\mathcal D_{data}(\mathbb P_0)(\epsilon-\epsilon^2)}-(1-2\epsilon)\mathcal D_{data}(\mathbb P_0)}{2\mathcal D_{data}(\mathbb P_0)+2}.
% \end{equation}
% \begin{remark}  We note that the even though \eqref{opt} is the optimal bound derived  based only on $\chi^2$ distances, it can be shown that the intuitive bound derived in \eqref{chi-square formula} using change of measure actually as \eqref{opt}
Compared to the bound obtained from Theorem \ref{inp2} and Corollary \ref{para nonpara}, \eqref{chi-square formula} is less tight, but the gap can be shown to asymptotically vanish when $\epsilon, \frac{\chi^2_{1-\alpha,D}}{n} \rightarrow 0$.
% \end{remark}
% Now, we may use \eqref{opt} to solve algorithm \ref{esgnew} without relying on results from \cite{jiang2016data}.

%For example, if $d_\phi$ is the $\chi^2$ distance,
%for any $ A \subset  \mathcal Y$:
%\begin{eqnarray}\label{jiangresult}
%\underset{\chi^2(\mathbb P_{\hat\theta},
%	\mathbb Q) \leq \lambda}{\text{max}} \mathbb Q( A) \leq \epsilon \iff %\mathbb P_{\hat\theta} ( A) %\leq\epsilon-\frac{\sqrt{\lambda^2+4\lambda(\epsilon-\epsilon^2)}-(1-2\e%psilon)\lambda}{2\lambda+2}.
%\end{eqnarray}

\section{Proofs and Other Technical Results}\label{sec:proofs}

% \subsection{Proof of Lemma \ref{convexity}}
\begin{proof}[Proof of Lemma \ref{convexity}]
	
% 	We want to show that the right hand side of \eqref{minimax} is a convex optimization problem. First,
First, by definition $\mathcal{P}(\Theta)$ is a convex set and, for any $\mu_1,\mu_2 \in \mathcal{P}(\Theta)$ and $0 < t < 1$, we have
	\begin{equation*}
	(1-t)\mu_1+t\mu_2 \in \mathcal P(\Theta).
	\end{equation*}
	Next, fixing $\theta\in\mathcal{U}_{data}$, the function $L(\cdot,\theta)$ is convex since:
	\begin{align*}
	L((1-t)\mu_1+t\mu_2,\theta)=&\int_{\mathcal{Y}}\frac{(p(y; \theta))^2}{\int_{\Theta} p(y;\theta') ((1-t)\mu_1+t\mu_2)(d\theta')} dy \nonumber\\
	=&\int_{\mathcal{Y}}\frac{(p(y; \theta))^2}{(1-t)\int_{\Theta} p(y;\theta') \mu_1(d\theta')+t\int_{\Theta} p(y;\theta') \mu_2(d\theta')} dy \nonumber\\
	\leq & (1-t)\int_{\mathcal{Y}}\frac{(p(y; \theta))^2}{\int_{\Theta} p(y;\theta') \mu_1(d\theta')}dy+t\int_{\mathcal{Y}}\frac{(p(y; \theta))^2}{\int_{\Theta} p(y;\theta') \mu_2(d\theta')}dy \nonumber\\
	=& (1-t)L(\mu_1,\theta)+tL(\mu_2,\theta)
	\end{align*}
	for any $0<t<1$ where the inequality follows from the convexity of the function $1/x$ for $x>0$. Thus, as the supremum of convex functions, $
	l(\mu)\triangleq\sup\limits_{\theta\in\mathcal{U}_{data}}L(\mu,\theta)$ is also convex.
\end{proof}

% \subsection{Proof of Theorem \ref{dktheorem}}\label{appdk}
We provide a version of Danskin' Theorem needed to prove Theorem \ref{ana}. Alternately, one can also resort to a generalized version in \cite{clarke1975generalized} by verifying the conditions there. Here we opt for the former and provide a self-contained proof, which mostly relies on the techniques from Proposition 4.5.1 of \cite{bertsekas2003convex} but with some slight modification to handle issues regarding the domain of the involved function. We have:
% as the domain $t$ of $L((1-t_k)\mu_1+t_k\mu_2+t(\mu_2+\mu_1),\theta)$depicted in the proof below is different for each $0\leq t_k\leq 1$. First we have:
% For the sake of completeness, we include a simplified, self contained proof here.

\begin{lemma}\label{dir}
	Fix probability measures $\mu_1,\mu_2 \in\mathcal P(\Theta)$. Suppose $t_k \downarrow 0$ is a positive sequence such that $(1-t_k)\mu_1+t_k\mu_2\in \mathcal P(\Theta)$ for all $k$ and  $\theta_k\in\Theta^*((1-t_k)\mu_1+t_k\mu_2)$ is a sequence such that $\theta_k\rightarrow\theta_0$ for some $\theta_0 \in \mathcal U_{data}$, then we have
	\begin{align*}
	\limsup_{k \rightarrow \infty}\frac{{L}((1-t_k)\mu_1+t_k\mu_2, \theta_k )-{L}(\mu_1,\theta_k)}{t_k} \leq   \lim_{t\downarrow 0} \frac{L((1-t)\mu_1+t\mu_2,\theta_0)-L(\mu_1,\theta_0)}{t},
	\end{align*}
	if we assume $L((1-t)\mu_1+t\mu_2,\theta)$ is jointly continuous in $0\leq t \leq 1$ and $\theta \in \Theta$.
\end{lemma}
\begin{proof}
	It is known that if $f:\mathbb{I}\rightarrow\mathbb{R}$ is a convex function with $\mathbb{I}$ being an open interval containing some point $x$, we then have the following results \cite{bertsekas2003convex}:
	\begin{equation}\label{plus0}
	f^{+}(x)=\lim_{t\downarrow 0} \frac{f(x+t)-f(x)}{t}=\inf_{t>0}\frac{f(x+t)-f(x)}{t},
	\end{equation}
	\begin{equation}\label{plus1}
	f^{-}(x)=\lim_{t\downarrow 0} \frac{f(x)-f(x-t)}{t}=\sup_{t>0}\frac{f(x)-f(x-t)}{t},
	\end{equation}
	and
	\begin{equation}\label{plus2}
	f^{+}(x) \geq f^{-}(x).
	\end{equation}
	In other words, these limits exist and satisfy the above relations for convex functions. Thus, if we define $f_k(t)={L}((1-t_k)\mu_1+t_k\mu_2+t(\mu_2-\mu_1),\theta_k)$, it follows from the convexity of $\mathcal P(\Theta)$ and ${L}(\cdot,\theta_k)$ that $f_k(t)$ is convex and well-defined for $-t_k \leq t \leq 1-t_k$. Using the above results in \eqref{plus0}, \eqref{plus1} and \eqref{plus2}, we then have
	\begin{align}\label{part1}
	\frac{{L}((1-t_k)\mu_1+t_k\mu_2, \theta_k )-{L}(\mu_1,\theta_k)}{t_k}
	=&\frac{f_k(0)-f_k(-t_k)}{t_k}\nonumber\\
	\leq &\sup_{t>0} \frac{f_k(0)-f_k(-t)}{t}
	= f^{-}_k(0)
	\leq  f^{+}_k(0)
	=\inf_{t>0}\frac{f_k(t)-f_k(0)}{t}.
	\end{align}
	On the other hand, if we define $f_0(t)=L((1-t)\mu_1+t\mu_2,\theta_0)$, it also follows that $f_0(t)$ is convex and well-defined for $0\leq t \leq 1$. It follows from the convexity of $f_0(\cdot)$ as well as \eqref{plus0} that
	\begin{align}\label{part1.5}
	\lim_{t\downarrow 0} \frac{L((1-t)\mu_1+t\mu_2,\theta_0)-L(\mu_1,\theta_0)}{t}=&\lim_{t\downarrow 0} \frac{f_0(t)-f_0(0)}{t}\nonumber\\
	=&\inf_{t>0}\frac{f_0(t)-f_0(0)}{t}=f^+_0(0).
	\end{align}
	Then, it again follows from the convexity of $f_0(\cdot)$ that, given any $\tau>0$, we can find some $\eta>0$ such that
	\begin{equation}\label{part2}
	\frac{f_0(s)-f_0(0)}{s} \leq f^+_0(0)+\tau,
	\end{equation}
	for all $0<s<\eta$. It then follows from definitions and \eqref{part2} that
	\begin{align}
	\frac{L((1-s)\mu_1+s\mu_2,\theta_0)-L(\mu_1,\theta_0)}{s}
	=&\frac{L((\mu_1+s(\mu_2-\mu_1),\theta_0)-L(\mu_1,\theta_0)}{s}\nonumber\\
	=&\frac{f_0(s)-f_0(0)}{s}\leq f^+_0(0)+\tau,
	\end{align}
	for all $0<s<\eta$. Fixing one such $s$, since the function $L((1-t)\mu_1+t\mu_2,\theta)$ is jointly continuous in $0\leq t \leq 1$ and $\theta \in \Theta$, and the sequence satisfies $\theta_k\rightarrow\theta_0$, we have
	\begin{align*}
	\lim_{k\to\infty}\frac{f_k(s)-f_k(0)}{s}=&\lim_{k\to\infty}\frac{{L}((1-t_k)\mu_1+t_k\mu_2+s(\mu_2-\mu_1),\theta_k)-L((1-t_k)\mu_1+t_k\mu_2,\theta_k)}{s} \nonumber\\
	=& \frac{L((\mu_1+s(\mu_2-\mu_1),\theta_0)-L(\mu_1,\theta_0)}{s}=\frac{f_0(s)-f_0(0)}{s}\leq f^+_0(0)+2\tau,
	\end{align*}
	as long as we make $\eta>s>0$ small enough so that $\eta\leq 1-t_k$ for all $k$. Then, for $k$ large enough, we have
	\begin{equation}\label{part3}
	\inf_{t>0}\frac{f_k(t)-f_k(0)}{t} \leq \frac{f_k(s)-f_k(0)}{s} \leq f^+_0(0)+2\tau.
	\end{equation}
	Combining \eqref{part1}, \eqref{part1.5} and \eqref{part3}, we have that, for $k$ large enough,
	\begin{equation*}
	\frac{{L}((1-t_k)\mu_1+t_k\mu_2, \theta_k )-{L}(\mu_1,\theta_k)}{t_k} \leq \lim_{t\downarrow 0} \frac{L((1-t)\mu_1+t\mu_2,\theta_0)-L(\mu_1,\theta_0)}{t}+2\tau.
	\end{equation*}
	Finally, since $\tau$ is arbitrary, we conclude that
	\begin{align*}
	\limsup_{k \rightarrow \infty}\frac{{L}((1-t_k)\mu_1+t_k\mu_2, \theta_k )-{L}(\mu_1,\theta_k)}{t_k} \leq   \lim_{t\downarrow 0} \frac{L((1-t)\mu_1+t\mu_2,\theta_0)-L(\mu_1,\theta_0)}{t}.
	\end{align*}
\end{proof}

We now prove the following version of Danskins' Theorem:
% [Proof of Theorem \ref{dktheorem}]
\begin{theorem}\label{dktheorem}
	 Fix $\mu_1,\mu_2 \in \mathcal{P}(\mathcal{U}_{data})$. Suppose $t_k \downarrow 0$ is a positive sequence such that $(1-t_k)\mu_1+t_k\mu_2\in \mathcal P(\Theta)$ for all $k$ and  $L((1-t)\mu_1+t\mu_2,\theta)$ is jointly continuous in $0\leq t \leq 1$ and $\theta \in \Theta$. Then, if we let ${\psi}(t)={l}((1-t)\mu_1+t\mu_2)$ for $0\leq t \leq 1$ with ${l}(\cdot)=\sup
	\limits_{\theta\in\mathcal{U}_{data}}L(\cdot,\theta)$ defined as \eqref{def1}, we have
	\begin{align}\label{thedn}
	{\psi}^{+}(0)=&\sup_{\theta\in \Theta^*(\mu_1)}\lim_{t\downarrow 0} \frac{{L}((1-t)\mu_1+t\mu_2,\theta)-{L}(\mu_1,\theta)}{t}
	\end{align}
\end{theorem}

\begin{proof}
	For any $\theta_0\in\Theta^*(\mu_1)$ and $\theta_t\in\Theta^*((1-t)\mu_1+t\mu_2)$, we have
	\begin{align*}
	\frac{{\psi}(t)-{\psi}(0)}{t} = \frac{{l}((1-t)\mu_1+t\mu_2)-{l}(\mu_1)}{t}
	=&\frac{{L}((1-t)\mu_1+t\mu_2,\theta_t)-\tilde{L}(\mu_1,\theta_0)}{t}\nonumber\\
	\geq& \frac{{L}((1-t)\mu_1+t\mu_2,\theta_0)-{L}(\mu_1,\theta_0)}{t}.
	\end{align*}
	Thus, by taking $t\downarrow 0$ and taking the supremum over all $\theta_0\in\Theta^*(\mu_1)$, we have
	\begin{equation}\label{tpart1}
	{\psi}^{+}(0)\geq\sup_{\theta\in \Theta^*(\mu_1)}\lim_{t\downarrow 0} \frac{{L}((1-t)\mu_1+t\mu_2,\theta)-{L}(\mu_1,\theta)}{t}.
	\end{equation}
	Notice that the existence of the several limits above follows from the convexity of related functions.
	To prove the reverse inequality, we consider a sequence $\{t_k\}$ with $0<t_k<1$  and $t_k\downarrow 0$. Then, we pick another sequence $\{\theta_k\} \subseteq \mathcal{U}_{data}$ with $\theta_k\in \Theta^*((1-t_k)\mu_1+t_k\mu_2)$ for all $k$. Since $\mathcal{U}_{data}$ is compact, there exist a subsequence of $\{\theta_k\}$ converge to some $\theta_0\in \mathcal{U}_{data}$. Without loss of generality, we drop the subsequence and simply assume $\theta_k \rightarrow \theta_0$. We first show  $\theta_0\in \Theta^*(\mu_1)$.
	To do this, pick any $\tilde\theta_0 \in \Theta^*(\mu_1)$. Since
	${L}((1-t)\mu_1+t\mu_2,\theta)$ is jointly continous in $t$ and $\theta$, we have
	\begin{equation*}
	{L}(\mu_1,\theta_0)=\lim\limits_{k\rightarrow \infty} {L}((1-t_k)\mu_1+t_k\mu_2,\theta_k) \geq \lim\limits_{k\rightarrow \infty} {L}((1-t_k)\mu_1+t_k\mu_2,\tilde\theta_0) ={L}(\mu_1,\tilde\theta_0),
	\end{equation*}
	where the inequality follows from the definition of $\theta_k$. Now, since $\tilde\theta_0 \in \Theta^*(\mu_1)$ and ${L}(\mu_1,\theta_0)\geq 	{L}(\mu_1,\tilde\theta_0)$, we must have
	\begin{equation*}
	{L}(\mu_1,\theta_0)= 	{L}(\mu_1,\tilde\theta_0) \text{ and }
	\theta_0 \in \Theta^*(\mu_1).
	\end{equation*}
	Now, using the definition of $\Theta^*(\mu_1)$, we can write
	\begin{align}\label{danski}
	{\psi}^{+}(0)=\inf_{0<t} \frac{{\psi}(t)-{\psi}(0)}{t}
	\leq \frac{{\psi}(t_k)-{\psi}(0)}{t_k}
	=&\frac{{l}((1-t_k)\mu_1+t_k\mu_2)-{l}(\mu_1)}{t_k}\nonumber\\
	=& \frac{{L}((1-t_k)\mu_1+t_k\mu_2,\theta_k)-{L}(\mu_1,\theta_0)}{t_k} \nonumber\\
	\leq&\frac{{L}((1-t_k)\mu_1+t_k\mu_2,\theta_k)-{L}(\mu_1,\theta_k)}{t_k}.
	\end{align}
	Now we use \lemref{dir} to conclude that
	\begin{align}\label{tpart2}
	{\psi}^{+}(0) \leq&\limsup_{k \rightarrow \infty}\frac{{L}((1-t_k)\mu_1+t_k\mu_2, \theta_k )-{L}(\mu_1,\theta_k)}{t_k} \nonumber\\
	\leq&   \lim_{t\downarrow 0} \frac{L((1-t)\mu_1+t\mu_2,\theta_0)-L(\mu_1,\theta_0)}{t} \nonumber\\
	\leq& \sup_{\theta\in \Theta^*(\mu_1)}  \lim_{t\downarrow 0} \frac{L((1-t)\mu_1+t\mu_2,\theta)-L(\mu_1,\theta)}{t}
	\end{align}
	Finally, we combine \eqref{tpart1} and \eqref{tpart2} to conclude the proof
	\begin{equation*}
	{\psi}^{+}(0)=\sup_{\theta\in \Theta^*(\mu_1)}\lim_{t\downarrow 0} \frac{{L}((1-t)\mu_1+t\mu_2,\theta)-{L}(\mu_1,\theta)}{t}
	\end{equation*}
\end{proof}

\begin{proof}[Proof of Theorem \ref{ana}]
	The result can be obtained from Leibniz's integral rule (i.e. differentiation under the integral sign). See, for example, Theorem 2.27 in \cite{folland2013real}.
\end{proof}

% \subsection{Proof of Technical Lemmas}\label{sec:technical}
Next we prove Proposition \ref{collec}. For convenience, we note that \eqref{gaupdf} can be written in a compact form for exponential family \cite{nielsen2014chi}:
\begin{equation}\label{exfamily}
	p(y ; \theta)= e^{\langle t(y),\theta\rangle-F(\theta)+k(y)},
	\end{equation}
where $\langle a,b\rangle=a^{\intercal} b$ represents the usual inner product in the Euclidean space, and $t(\cdot), F(\cdot)$ and $k(\cdot)$ are known functions. In particular, we have \begin{equation}\label{example1}
F(\theta)=\frac{\theta^\intercal \Sigma^{-1}\theta}{2}
\end{equation}

To facilitate the calculation, we first introduce two lemmas involving the exponential parametric family based on \cite{nielsen2014chi}.
\begin{lemma}\label{pq}
	Pick $\theta_1,\theta_2 \in \Theta$. If $2\theta_2 -\theta_1\in \Theta$, then we have
	\begin{equation*}
	\int_{\mathcal{Y}}\frac{(p(y;\theta_2))^2}{p(y;\theta_1)}dy=e^{F(2\theta_2 -\theta_1)-(2F(\theta_2)-F(\theta_1))}.
	\end{equation*}
	In particular, if $F(\theta)=\frac{\theta^\intercal \Sigma^{-1}\theta}{2}$, then $	\int_{\mathcal{Y}}\frac{(p(y;\theta_2))^2}{p(y;\theta_1)}dy=e^{(\theta_2-\theta_1)^\intercal\Sigma^{-1}(\theta_2-\theta_1)}$.
\end{lemma}

\begin{proof}
	It follows from \eqref{exfamily} that
	\begin{align*}
	\int_{\mathcal{Y}}\frac{(p(y;\theta_2))^2}{p(y;\theta_1)}dy=&e^{<t(y),2\theta_2-\theta_1>-(2F(\theta_2)-F(\theta_1))+k(y)}dy \nonumber\\
	=&e^{F(2\theta_2-\theta_1)-(2F(\theta_2)-F(\theta_1))} \cdot \int_{\mathcal{Y}} p(y;2\theta_2-\theta_1) dy \nonumber\\
	=& e^{F(2\theta_2-\theta_1)-(2F(\theta_2)-F(\theta_1))} .
	\end{align*}
\end{proof}
\begin{lemma}\label{3pq}
	Pick $\theta_1$,$\theta_2$ and $\theta_3 \in \Theta$. If $2\theta_2-2{\theta}_{1}+\theta_3\in\Theta$, then we have
	\begin{align*}
	\int_{\mathcal{Y}}\frac{(p(y;\theta_2))^2p(y;\theta_3)}{(p(y;\theta_1))^2}dy =e^{F(2\theta_2-2{\theta}_{1}+\theta_3)-2F(\theta_2)+2F({\theta}_{1})-F(\theta_3)}.
	\end{align*}
	In particular, if $F(\theta)=\frac{\theta^\intercal \Sigma^{-1}\theta}{2}$, then $		\int_{\mathcal{Y}}\frac{(p(y;\theta_2))^2p(y;\theta_3)}{(p(y;\theta_1))^2}dy=e^{(\theta_2-\theta_1)^\intercal\Sigma^{-1}(\theta_2-\theta_1)+2(\theta_2-\theta_1)\Sigma^{-1}(\theta_3-\theta_1)}$.
\end{lemma}
\begin{proof}
	The proof follows from the same techniques as in \lemref{pq}.
\end{proof}
Then \eqref{final1} follows from \eqref{parametric uncertainty}, \eqref{example1} and \lemref{pq} so that
\begin{align*}
{\mathcal{U}}_{data}\triangleq\left\{\theta \in \Theta: \chi^2(\mathbb P_{\hat\theta}, \mathbb P_\theta)\leq \frac{\chi^2_{1-\alpha,D}}{n}\right\}
=&\left\{\theta \in \Theta: e^{F(2\theta -\hat\theta)-(2F(\theta)-F(\hat\theta))}-1\leq \frac{\chi^2_{1-\alpha,D}}{n}\right\} \nonumber\\
=&\left\{\theta \in \Theta: e^{(\theta-\hat\theta)^\intercal\Sigma^{-1}(\theta-\hat\theta)}-1\leq \frac{\chi^2_{1-\alpha,D}}{n}\right\}\nonumber\\
=&\left\{\theta: \hat\theta+\Sigma^{\frac{1}{2}}v, \quad \text {for all }  \|v\|_2^2 \leq \log(1+\frac{\chi^2_{1-\alpha,D}}{n})\right\},
\end{align*}
and \eqref{final2} follows. We now prove Proposition \ref{collec}:

\begin{proof}[Proof of Proposition \ref{collec}]
Following Theorem \ref{ana}, \lemref{pq} and \lemref{3pq}, we have
\begin{align}\label{cproof}
&\sup_{\theta\in \Theta^*(\delta_{\hat\theta})}\int_{\mathcal{Y}}\frac{(p(y; \theta))^2\cdot \int_{\Theta}p(y;\theta')(\delta_{\hat\theta}-\mu_{prop})(d\theta')}{(\int_{\Theta} p(y;\theta') \delta_{\hat\theta}(d\theta'))^2} dy \nonumber\\
=&\sup_{\theta\in \Theta^*(\delta_{\hat\theta})}\bigg(\int_{\mathcal{Y}}\frac{(p(y; \theta))^2}{ p(y;\hat\theta)} dy-\int_{\mathcal{Y}}\frac{(p(y; \theta))^2\cdot \int_{\Theta}p(y;\theta')(\mu_{prop})(d\theta')}{( p(y;\hat\theta))^2} dy\bigg) \nonumber\\
=& \sup_{\theta\in \Theta^*(\delta_{\hat\theta})}\bigg(\int_{\mathcal{Y}}\frac{(p(y; \theta))^2}{ p(y;\hat\theta)} dy-\int_{\Theta}\int_{\mathcal{Y}}\frac{(p(y; \theta))^2\cdot p(y;\theta')}{( p(y;\hat\theta))^2} dy\cdot \mu_{prop}(d\theta')\bigg) \nonumber\\
=& \sup_{\theta\in \Theta^*(\delta_{\hat\theta})}\bigg(e^{(\theta-\hat\theta)^\intercal\Sigma^{-1}(\theta-\hat\theta)}-\int_{\Theta}e^{(\theta-\hat\theta)^\intercal\Sigma^{-1}(\theta-\hat\theta)+2(\theta-\hat\theta)^\intercal\Sigma^{-1}(\theta'-\hat\theta)}\mu_{prop}(d\theta')\bigg) \nonumber\\
=&(1+\frac{\chi^2_{1-\alpha,D}}{n})\cdot \sup_{\theta\in \Theta^*(\delta_{\hat\theta})}\Big(1-\mathbb E_{\theta'\sim\mu_{prop}}[e^{2(\theta-\hat\theta)^\intercal\Sigma^{-1}(\theta'-\hat\theta)}]\Big) \nonumber\\
=&(1+\frac{\chi^2_{1-\alpha,D}}{n})\cdot\Big(1-\inf_{\theta\in \Theta^*(\delta_{\hat\theta})}\mathbb E_{\theta'\sim\mu_{prop}}[e^{2(\theta-\hat\theta)^\intercal\Sigma^{-1}(\theta'-\hat\theta)}]\Big).
\end{align}
Notice the second equality follows from Fubini's theorem. The third equality follows from \lemref{pq} and \lemref{3pq}. The fourth equality follows from \eqref{final2}. Now, following the last line \eqref{cproof}, for the search of descent direction, it is sufficient to prove
	\begin{equation*}
	\inf_{\theta\in \Theta^*(\delta_{\hat\theta})}\mathbb E_{\theta'\sim\mu_{prop}}[e^{2(\theta-\hat\theta)^\intercal\Sigma^{-1}(\theta'-\hat\theta)}]>1.
	\end{equation*}
	However, since $\mu_{prop}(d\theta')$ is a symmetrical distribution around $\hat\theta$, we know that
	\begin{equation*}
	\mathbb E_{\theta'\sim\mu_{prop}}[2(\theta-\hat\theta)^\intercal\Sigma^{-1}(\theta'-\hat\theta)]=0.
	\end{equation*}
	for any $\theta\in \Theta^*(\delta_{\hat\theta})$. Then, it follows from Jensen's inequality that
	\begin{equation*}
	\inf_{\theta\in \Theta^*(\delta_{\hat\theta})}\mathbb E_{\theta'\sim\mu_{prop}}[e^{2(\theta-\hat\theta)^\intercal\Sigma^{-1}(\theta'-\hat\theta)}] \geq 1.
	\end{equation*}
	Now suppose for the sake of contradiction that
	\begin{equation*}
	\inf_{\theta\in \Theta^*(\delta_{\hat\theta})}\mathbb E_{\theta'\sim\mu_{prop}}[e^{2(\theta-\hat\theta)^\intercal\Sigma^{-1}(\theta'-\hat\theta)}] = 1.
	\end{equation*}
	Then, let $\{\theta_k\}_k \subseteq \Theta^* (\delta_{\hat\theta})$ be a subsequence such that $\mathbb E_{\theta'\sim\mu_{prop}}[e^{2(\theta_k-\hat\theta)^\intercal\Sigma^{-1}(\theta'-\hat\theta)}] \rightarrow 1$. Due to the compactness of $\Theta^* (\delta_{\hat\theta})$, we can find a subsequence of $\{ \theta_k \}_k$ converging to some $\theta_{0} \in \Theta^* (\delta_{\hat\theta})$. For convenience we drop the subsequence and suppose $\theta_k \rightarrow \theta_{0}.$ Then the existence of $Y$ allows us to use dominated convergence theorem:
	\begin{equation*}
	\mathbb E[e^{2Y_{\theta_0}}]=\mathbb E_{\theta'\sim\mu_{prop}} [e^{2(\theta_{0}-\hat\theta)^\intercal\Sigma^{-1}(\theta'-\hat\theta)}]=\lim_{k\to\infty}\mathbb E_{\theta'\sim\mu_{prop}}[e^{2(\theta_k-\hat\theta)^\intercal\Sigma^{-1}(\theta'-\hat\theta)}]=1.
	\end{equation*}
	However, Jensen's inequality would indicate that $\mathbb E[e^{2Y_{\theta_{0}}}]=1$ if and only $\mathbb P(Y_{\theta_0}=0)=1$, which contradicts our assumption. Thus, we know that
	\begin{equation*}
	\inf_{\theta\in \Theta^*(\delta_{\hat\theta})}\mathbb E_{\theta'\sim\mu_{prop}}[e^{2(\theta-\hat\theta)^\intercal\Sigma^{-1}(\theta'-\hat\theta)}] > 1,
	\end{equation*}
	as claimed.
\end{proof}

\begin{proof}[Proof of Proposition \ref{cod}]
First we prove \eqref{codp1}. Letting $c=\frac{1}{(2\pi)^D|\Sigma|}$, we know that
\begin{align}\label{llll}
    \chi^2(\mathbb P_0,\mathbb P_\theta)=&\int \frac{p^2(y;\theta)}{p_0(y)}dy -1 \nonumber\\
    =& c\int \frac{e^{-(y-\theta)^T\Sigma^{-1}(y-\theta)}}{p_0(y)}dy -1\nonumber\\
    =& ce^{-\|\Sigma^{-1/2}(\theta-\hat{\theta})\|_2^2}\int \frac{e^{-(y-\hat{\theta})^T\Sigma^{-1}(y-\hat{\theta})}\cdot e^{-2(y-\hat\theta)^T\Sigma^{-1}(\hat\theta-\theta)}}{p_0(y)}dy -1 \nonumber\\
    =& c|\Sigma^{1/2}|e^{-\|\Sigma^{-1/2}(\theta-\hat{\theta})\|_2^2}\int \frac{e^{-z^Tz}\cdot e^{-2z^T\Sigma^{-1/2}(\hat\theta-\theta)}}{p_0(\Sigma^{1/2}z+\hat\theta)}dz -1 \nonumber\\
    =& c|\Sigma| e^{-\|\Sigma^{-1/2}(\theta-\hat{\theta})\|_2^2}\int \frac{e^{-z^Tz}\cdot e^{-2z^T\Sigma^{-1/2}(\hat\theta-\theta)}}{p_Z(z)}dz -1
\end{align}
where we denote $p_Z(\cdot)$ to be the density function of random variable $Z=\Sigma^{-1/2}(Y-\hat\theta)$ with $Y\sim\mathbb P_0$ and the last two lines follow from a change of variable $z=\Sigma^{-1/2}(y-\hat\theta)$. Now, since $\|\Sigma^{-1/2}(\theta_1-\hat\theta)\|_2^2=\|\Sigma^{-1/2}(\theta_2-\hat\theta)\|_2^2=r$ for some $r$ by assumption, it follows from \eqref{llll} that $\chi^2(\mathbb P_0,\mathbb P_{\theta_1})=\chi^2(\mathbb P_0,\mathbb P_{\theta_2})$ if we can show
\begin{equation*}
    \int \frac{e^{-z^Tz}\cdot e^{-2z^T\Sigma^{-1/2}(\hat\theta-\theta_1)}}{p_Z(z)}dz=\int \frac{e^{-z^Tz}\cdot e^{-2z^T\Sigma^{-1/2}(\hat\theta-\theta_2)}}{p_Z(z)}dz.
\end{equation*}
However, since $p_Z(z)$ and $e^{-z^Tz}$ are both rotationally invariant functions (i.e. $f(z)=f(Q^\intercal z)$ for all $z$ and rotational matrix $Q$, with $|Q|=1$), it can be shown that $\int \frac{e^{-z^Tz}\cdot e^{-2z^T\nu}}{p_Z(z)}dz$ holds the same value for any $\nu$ such that $\|\nu\|_2^2=r$. Notice the rotational invariance of $p_Z(z)$ follows from the rotational invariance of $Z$. This proves \eqref{codp1}.
To prove \eqref{codp2}, notice that for any $\theta\in\mathcal U_{data}$, we can find some $0\leq t \leq 1$ such that $$(((1-t)\hat\theta+t\theta^\star)-\hat\theta)^\intercal\Sigma^{-1}(((1-t)\hat\theta+t\theta^\star)-\hat\theta)=(\theta-\hat\theta)^\intercal\Sigma^{-1}(\theta-\hat\theta)$$
and hence $\chi^2(\mathbb P_0,\mathbb P_{((1-t)\hat\theta+t\theta^*)})=\chi^2(\mathbb P_0,\mathbb P_{\theta})$ by \eqref{codp1}.
\end{proof}

\begin{proof}[Proof of Proposition \ref{cod1}]
To check that $Y\sim\mathbb P_t$ with density
\begin{equation*}
p_{t}(y)=\int_{\Theta} p(y;\theta')  ((1-t)\delta_{\hat\theta}+t\mu_{prop})(d\theta')=(1-t)\mathbb P_{\hat\theta}+\int_{\Theta}p(y;\theta')\mu_{prop}(d\theta'),
\end{equation*}
leads to rotationally invariant $Z=\Sigma^{-1/2}(Y-\hat\theta)$, simply notice that
\begin{equation*}
    Y {\buildrel \mathcal D \over =} (1-U_t)(\hat\theta+X_1)+U_t(\hat\theta+\sqrt{\frac{\chi^2_{1-\alpha,D}}{n}}\cdot \Sigma^{1/2}\eta+X_2),
\end{equation*}
where $U_t$ is an independent Bernoulli variable with success rate $t$, $\eta$ is a random vector uniformly distributed on the surface of the $D$-dimensional unit ball and $X_1,X_2$ are independent $\mathcal N(0,\Sigma)$. Then, it follows that
\begin{equation*}
    \Sigma^{-1/2}(Y-\hat\theta) {\buildrel \mathcal D \over =} (1-U_t) Z_1+ U_t(\sqrt{\frac{\chi^2_{1-\alpha,D}}{n}} \eta +Z_2)
\end{equation*}
where $Z_1,Z_2$ are now independent $\mathcal N(0,I_D)$. Consequently, the rotational invariance of $Z$ now follows from the rotational invariance of $Z_1,Z_2,\eta$ and their independence.
\end{proof}

Following the comments after Proposition \ref{collec}, we show that $\theta\sim\mu_{prop}$ with $\theta\overset{\mathcal{D}}{=}\hat{\theta}+\sqrt{\frac{\chi^2_{1-\alpha,D}}{n}}\cdot\Sigma^{1/2}\cdot\eta$ provides a descent direction, with an alternate proof using the following lemma and the last line of \eqref{cproof}.
\begin{lemma}\label{move}
Fixing $\theta_1 \in \Theta^*(\delta_{\hat{\theta}})$, we have
	\begin{equation*} \mathbb{E}_{\theta_2\sim\mu_{prop}}[e^{2(\theta_1-\hat{\theta})^T\Sigma^{-1}(\theta_2-\hat{\theta})}] >1,
	\end{equation*}
	for $\theta_2\sim\mu_{prop}(d\theta)$ where $\theta_2\overset{\mathcal{D}}{=}\hat{\theta}+\sqrt{\frac{\chi^2_{1-\alpha,D}}{n}}\cdot\Sigma^{1/2}\cdot\eta$ with $\eta$ following the uniform distribution on the surface of the $D$-dimensional unit ball.
\end{lemma}

\begin{proof}[Proof of Lemma \ref{move}]
	Let $u_1 \in \mathbb{R}^D$ denote an arbitary point on the surface of $D$ dimensional unit ball ( $\|u_1\|_2^2=1$) and let $\eta=[\eta_1,\eta_2,...,\eta_D]$ be the random vector in $\mathbb{R}^D$ uniformly distributed on the surface of $D$ dimensional unit ball. Then we claim that $\frac{\mu_1^T\eta+1}{2}\sim Beta(\frac{D-1}{2},\frac{D-1}{2})$.
	
	To show this, assume without loss of generality that $u_1=[1,0,...,0]\in\mathbb{R}^D$. Then for any $t\in[-1,1]$, it follows that
	$\mathbb{P}(u_1^T\eta\in dt)$ is proportional to the infinitesimal surface area on the ball corresponding to $\eta_1\in dt$, which is in turn proportional to the product of the sub-dimension $D-2$ surface area on the belt $x_2^2+x_3^2+...+x_D^2=1-t^2$ with the infinitesimal width of this belt. Specifically, the sub-dimension $D-2$ surface area around the belt is proportional to $(\sqrt{1-t^2})^{D-2}$. This follows from the fact that points of the form $[0,\sqrt{1-t^2},0,...,0], [0,0,\sqrt{1-t^2},0,...,0],...,[0,0,...,0,\sqrt{1-t^2}]$ are on this belt. Also, the width of this belt, according to the Pythagorean theorem,  is $dt \cdot \sqrt{(\frac{d\sqrt{1-t^2}}{dt})^2+1}=\frac{dt}{\sqrt{1-t^2}}$. Thus,
	\begin{equation*}
	\mathbb{P}(u_1^T\eta\in dt)\propto \frac{(\sqrt{1-t^2})^{D-2}}{\sqrt{1-t^2}}dt=(1-t^2)^{\frac{D-3}{2}}dt.
	\end{equation*}
	Now, we can substitute $\frac{t+1}{2}=s$ with $s\in [0,1] $ to get
	\begin{equation*}
	\mathbb{P}(\frac{u_1^T\eta+1}{2}\in ds)\propto (s)^{\frac{D-1}{2}-1}(1-s)^{\frac{D-1}{2}-1}ds,
	\end{equation*}
	which can only be the density function for $Beta(\frac{D-1}{2},\frac{D-1}{2})$.  It now follows from \cite{winkelbauer2012moments} that $\frac{u_1^T\eta+1}{2}$ has moment generating function
	\begin{align}\label{hyper}
	M(t)\triangleq&\mathbb{E}[e^{t\cdot\frac{u_1^T\eta+1}{2}}] \nonumber\\
	=&_{1}F_{1}(\frac{D-1}{2},D-1,t)
	= e^{(t/2)}{} _{0}F_{1}(;\frac{D}{2},\frac{t^2}{16})
	\geq e^{t/2} (1+ct^2)
	>e^{(t/2)}.
	\end{align}
	for some $c>0$ where $_{1}F_{1}(\cdot,\cdot,\cdot)$ and $_{0}F_{1}(;,\cdot,\cdot)$ are the confluent hypergeometric function with identity $_1F_1(a,2a,x)=e^{x/2}_0F_1(;a+1/2,x^2/16)$ (see \cite{daalhuis2010confluent}),
	\begin{equation*}
	_{0}F_{1}(;\alpha,t)\triangleq\sum_{k=0}^{\infty}\frac{t^k}{(\alpha)_kk!} \quad\text{and}\quad _{1}F_{1}(\alpha,\beta,t)\triangleq\sum_{k=0}^{\infty}\frac{(\alpha)_kt^k}{(\beta)_kk!},
	\end{equation*}
	with $(\gamma)_k=\frac{\Gamma(\gamma+k)}{\Gamma(\gamma)}$ being the Pochhammer symbol \cite{winkelbauer2012moments}. To conclude the proof, denote $\rho_n=\sqrt{\log(1+\frac{\chi^2_{1-\alpha,D}}{n})}\cdot\sqrt{\frac{\chi^2_{1-\alpha,D}}{n}}$ and use \eqref{final1}, \eqref{final2} and \eqref{hyper} to write
	\begin{align*}
	&\mathbb{E}_{\theta_2\sim\mu_{prop}}[e^{2(\theta_1-\hat{\theta})^T\Sigma^{-1}(\theta_2-\hat{\theta})}] \nonumber\\
	=&\mathbb{E}_{\nu\sim\eta}[e^{2\rho_n\cdot\mu_1^T\nu}]
	= \mathbb{E}_{X\sim Beta(\frac{D-1}{2},\frac{D-1}{2})}[e^{2\rho_n\cdot(2X-1)}]
	=M(4\rho_n)/ e^{2\rho_n}
	\geq  (1+16c\rho_n^2)>1.
	\end{align*}
\end{proof}
\begin{remark}
Following Lemma \ref{move}, we discuss the numerical calculations of $\mathcal D(\mathbb P_0)$ following Proposition
\ref{cod1}. We use $\mathcal U_{data}=\{\mathbb P_\theta: \|\theta-\hat\theta\|_2^2\leq\frac{\chi^2_{1-\alpha,D}}{n}\}$ where $p(y;\theta)=(2\pi)^{-\frac{D}{2}}  e^{-\frac{1}{2}\|(y-\theta)\|_2^2}$. Then, for $\mu_1$, the nominal $p_0(y)$ is simply $p(y;\hat\theta)$ and $\mathcal D_{data}(\mathbb P_{0})=\mathcal D_{data}(\mathbb P_{\hat\theta})=\max_{\theta\in\mathcal U} e^{\|\theta-\hat\theta\|_2^2}-1=e^{\frac{\chi^2_{1-\alpha,D}}{n}}-1$ according to \eqref{final1} and Lemma \ref{pq}. For $\mu_2$, it can be shown that the nominal $\mathbb P_0$ follows $\mathcal N(\hat\theta,(1+\frac{1}{n}) \cdot I_D)$, and a direct computation would show that $\mathcal D_{data}(\mathbb P_0)=\max_{\theta\in\mathcal U
}(\frac{(n+1)^2}{n(n+2)})^{\frac{d}{2}} e^{\frac{n}{n+2} \|\theta-\hat\theta\|_2^2}-1=(\frac{(n+1)^2}{n(n+2)})^{\frac{d}{2}} e^{\frac{n}{n+2}\frac{\chi^2_{1-\alpha,D}}{n}}-1$. Finally, for $\mu_3$, assume w.l.o.g that $\hat\theta=0$. Then we use the derivation in Lemma \ref{move} that $\frac{\mu_1^T\eta+1}{2}\sim Beta(\frac{D-1}{2},\frac{D-1}{2})$ for any $u_1$ on the $D$-dimensional unit ball surface to show that, for any $v\in\mathbb{R}^{D}$,
\begin{equation}
    \mathbb E_{\eta}[e^{\eta^Tv}]=e^{-\|v\|_2} {_1F_1(\frac{D-1}{2},D-1,2\|v\|_2)},
\end{equation}
and consequently
\begin{equation*}
   p_0(y)=(2\pi)^{-\frac{D}{2}}_1F_1(\frac{d-1}{2},d-1,2(\frac{\chi^2_{1-\alpha,D}}{n})^{1/2}\|y\|_2) e^{-\frac{1}{2}(\|y\|_2+\frac{\chi^2_{1-\alpha,D}}{n})^2}.
\end{equation*}
Then, to calculate $\mathcal D_{data}(\mathbb P_0)$, we note that
\begin{align}\label{compprop}
    \mathcal D_{data}(\mathbb P_0)+1=&\max_{\|\theta\|_2^2\leq \frac{\chi^2_{1-\alpha,D}}{n}} \int \frac{p^2(y;\theta)}{p_0(y)}dy  \nonumber\\
    =& \max_{\|\theta\|_2^2\leq \frac{\chi^2_{1-\alpha,D}}{n}} (2\pi)^{-\frac{D}{2}} e^{-\frac{1}{2} \|y\|_2^2} \frac{e^{-\|\theta\|_2^2+2\theta^T y+\frac{\chi^2_{1-\alpha,D}}{2n}+(\frac{\chi^2_{1-\alpha,D}}{n})^{\frac{1}{2}}\|y\|_2}}{_1F_1(\frac{d-1}{2},d-1,2(\frac{\chi^2_{1-\alpha,D}}{n})^{\frac{1}{2}}\|y\|_2)}\nonumber\\
    =& \max_{\|\theta\|_2^2\leq \frac{\chi^2_{1-\alpha,D}}{n}}\mathbb E_{Y\sim\mathcal N(0,I_D)}\bigg[\frac{e^{-\|\theta\|_2^2+2\theta^T Y+\frac{\chi^2_{1-\alpha,D}}{2n}+(\frac{\chi^2_{1-\alpha,D}}{n})^{\frac{1}{2}}\|Y\|_2}}{_1F_1(\frac{d-1}{2},d-1,2(\frac{\chi^2_{1-\alpha,D}}{n})^{\frac{1}{2}}\|Y\|_2)}\bigg].
    \end{align}
Furthermore, through either direct verification or analysis similar to those in Lemma \ref{move},  we note that $Y\sim\mathcal N(0,I_D)$ shares the same distribution of $L\eta$ where $L\in\mathbb{R}^{+}$ and $\eta\in\mathbb{R}^{D}$ are two independent random variables with $L$ being the norm of $\mathcal N(0,I_D)$ bearing density $f_L(l)=1_{\{l\geq 0\}}\frac{2^{1-\frac{D}{2}}}{\Gamma(\frac{D}{2})}l^{d-1}e^{-\frac{l^2}{2}}$ and $\eta$ being the random vector on the $D$-dimensional unit ball surface. Thus, it follows from \eqref{compprop} that  \eqref{compprop} equals
\begin{align*}
   &\max_{\|\theta\|_2^2\leq \frac{\chi^2_{1-\alpha,D}}{n}}\mathbb E_{L}\bigg[\mathbb E_{\eta}\bigg[\frac{e^{-\|\theta\|_2^2+2L\theta^T \eta+\frac{\chi^2_{1-\alpha,D}}{2n}+(\frac{\chi^2_{1-\alpha,D}}{n})^{\frac{1}{2}}L}}{_1F_1(\frac{d-1}{2},d-1,2(\frac{\chi^2_{1-\alpha,D}}{n})^{\frac{1}{2}}L)} \bigg|L \bigg]\bigg]\nonumber\\
   =&\max_{\|\theta\|_2^2\leq \frac{\chi^2_{1-\alpha,D}}{n}}\mathbb E_{L}\bigg[e^{-\|\theta\|_2^2+\frac{\chi^2_{1-\alpha,D}}{2n}+(\frac{\chi^2_{1-\alpha,D}}{n})^{\frac{1}{2}}L-2L\|\theta\|_2}\frac{_1F_1(\frac{D-1}{2},D-1,4L\|\theta\|_2)}{_1F_1(\frac{D-1}{2},D-1,2(\frac{\chi^2_{1-\alpha,D}}{n})^{\frac{1}{2}}L)}\bigg]\nonumber\\
   =&\max_{\|\theta\|_2^2\leq \frac{\chi^2_{1-\alpha,D}}{n}}\mathbb E_{L}\bigg[e^{-\|\theta\|_2^2+\frac{\chi^2_{1-\alpha,D}}{2n}}\frac{_0F_1(;\frac{D}{2},L^2\|\theta\|_2^2)}{_0F_1(;\frac{D}{2}, L^2(\frac{\chi^2_{1-\alpha,D}}{4n})}\bigg]\nonumber\\
   =&\max_{t\leq \frac{\chi^2_{1-\alpha,D}}{n}}e^{-t+\frac{\chi^2_{1-\alpha,D}}{2n}} \int_{l\geq0}\frac{_0F_1(;\frac{D}{2},l^2t)}{_0F_1(;\frac{D}{2}, l^2(\frac{\chi^2_{1-\alpha,D}}{4n})} \frac{2^{1-\frac{D}{2}}}{\Gamma(\frac{D}{2})}l^{d-1}e^{-\frac{l^2}{2}} dl
\end{align*}
which is numerically tractable.
\end{remark}

\begin{proof}[Proof of Theorem \ref{more}]
It follows from routine calculation that we can find a compact neighborhood of $r$ around 0 such that $\nabla_r L(r,\theta)$ exists and is continuous. Thus we can use the main theorem in \cite{clarke1975generalized} to show that
    \begin{align*}
\lim_{r\downarrow 0} \frac{l(r)-l(0)}{r}=&\sup_{\theta\in \Theta^*(\delta_{\hat\theta})}\int_{\mathcal{Y}}-\frac{(p(y; \theta))^2\cdot \lim_{r\downarrow 0}\frac{p_r(y)-p(y;\hat\theta)}{r}}{( p(y;\hat\theta') )^2} dy \nonumber\\
=&\sup_{\theta\in \Theta^*(\delta_{\hat\theta})}\lim_{r\downarrow 0}\frac{1}{r}\int_{\mathcal{Y}}\frac{(p(y; \theta))^2{(p(y;\hat\theta)-p_r(y))}}{( p(y;\hat\theta') )^2} dy \nonumber\\
=&\sup_{\theta\in \Theta^*(\delta_{\hat\theta})}\lim_{r\downarrow 0}\frac{1}{r}\int_{\mathcal{Y}}\frac{(p(y; \theta))^2}{p(y;\hat\theta')}-\frac{(p(y; \theta))^2\int _{\theta'\in\Theta}p(y;\theta')\mu_r(d\theta')}{( p(y;\hat\theta') )^2} dy\nonumber\\
=& (1+\frac{\chi^2_{1-\alpha,D}}{n})\cdot\lim_{r\downarrow 0}\frac{1}{r}\Big(1-\inf_{\theta\in \Theta^*(\delta_{\hat\theta})}\mathbb E_{\theta'\sim\mu_{r}}[e^{2(\theta-\hat\theta)^\intercal\Sigma^{-1}(\theta'-\hat\theta)}]\Big).
\end{align*}
\end{proof}

To prove Corollary \ref{variability}, we present two technical Lemmas \ref{n1} and \ref{n2}.
\begin{lemma}\label{n1}
	For any $\theta \in \Theta^*(\delta_{\hat\theta})$, $\lim_{r\downarrow 0}\frac{1}{r}\Big(1-\mathbb E_{\theta'\sim\mu^1_{r}}[e^{2(\theta-\hat\theta)^\intercal\Sigma^{-1}(\theta'-\hat\theta)}]\Big)$ is a fixed negative value.
\end{lemma}

\begin{proof}[Proof of Lemma \ref{n1}]
	For any $\theta \in \Theta^*(\delta_{\hat\theta})$, we have $\|\Sigma^{-1/2}(\theta-\hat\theta)\|_2=\sqrt {\log (1+\frac{\chi^2_{1-\alpha, D}}{n})}$. Denote $\rho_n=\sqrt {\log (1+\frac{\chi^2_{1-\alpha, D}}{n})}$. Furthermore, under $\theta' \sim \mu^1_r(d\theta')$, we have $\Sigma^{-1/2}(\theta'-\hat\theta)\sim\eta_{\sqrt{r}}$,  the uniform distribution inside the $D$-dimensional unit ball with radius $\sqrt{r}$, which can be viewed as the product of two independent random variables
	\begin{equation*}
	\eta_{\sqrt{r}}\sim U \cdot R,
	\end{equation*}
	where $U$ is the uniform distribution on the surface of the $D$-dimensional unit ball and $R$ is the norm of the random vector ranged from $0$ to $\sqrt{r}$. For any $0 \leq s \leq \sqrt{r}$, since $\eta_{\sqrt{r}}$ follows a uniform distribution inside a $D$-dimensional unit ball, and the volume of a $D$-dimensional ball with radius $s$ is proportional to $s^D$, then $f_R(s)$, the density of $R$, must satisfy
	\begin{equation*}
	f_R(s) \sim \frac{d s^D}{ds} \sim s^{D-1},
	\end{equation*}
	which is equivalent to saying
	\begin{equation*}
	f_R(s) =\frac{D}{(\sqrt{r})^D} s^{D-1}, \quad \text {for $0 \leq s \leq \sqrt{r}$}.
	\end{equation*}
	Thus, we have that $\mathbb E[R^2]=c_1 r$ for some $c_1>0$. Now we let $u_1=[1,0,...,0]\in\mathbb{R}^D$. We utilize the proof in \lemref{move} as well as the independence of $R, U$ to show that
	\begin{align*}
	\mathbb E_{\theta'\sim\mu^1_{r}}[e^{2(\theta-\hat\theta)^\intercal\Sigma^{-1}(\theta'-\hat\theta)}]=&\mathbb E_{U,R}[e^{2\rho_n \cdot R \cdot  u_1^\intercal U}] \nonumber\\
	=&\mathbb E_{R}[\mathbb E[e^{2\rho_n \cdot R \cdot  u_1^\intercal U} | R]] \nonumber\\
	=&\mathbb E_{R}[M(4\rho_nR)/e^{2\rho_n R}] \nonumber\\
	\geq &\mathbb E[1+16c \rho_n^2 R^2] \geq 1+16c\rho_n^2c_1r.
	\end{align*}
	Now it follows that
	\begin{equation*}
	\lim_{r\downarrow 0}\frac{1}{r}\Big(1-\mathbb E_{\theta'\sim\mu^1_{r}}[e^{2(\theta-\hat\theta)^\intercal\Sigma^{-1}(\theta'-\hat\theta)}]\Big) \leq -16c\rho_n^2c_1.
	\end{equation*}
\end{proof}

\begin{lemma}\label{n2}
	For any $\theta \in \Theta^*(\delta_{\hat\theta})$, $\lim_{r\downarrow 0}\frac{1}{r}\Big(1-\mathbb E_{\theta'\sim\mu^2_{r}}[e^{2(\theta-\hat\theta)^\intercal\Sigma^{-1}(\theta'-\hat\theta)}]\Big)$ is a fixed negative value.
\end{lemma}

% ,  the $D$-dimensional Gaussian with covariance matrix $rI$
\begin{proof}[Proof of Lemma \ref{n2}]
	For any $\theta \in \Theta^*(\delta_{\hat\theta})$, we have $\|\Sigma^{-1/2}(\theta-\hat\theta)\|_2=\sqrt {\log (1+\frac{\chi^2_{1-\alpha, D}}{n})}$. Denote $\rho_n=\sqrt {\log (1+\frac{\chi^2_{1-\alpha, D}}{n})}$. Furthermore, under $\theta' \sim \mu^2_r(d\theta')$, we have $\Sigma^{-1/2}(\theta'-\hat\theta)\sim \mathcal{N}(0,rI_D)$. Using the moment generating function for Gaussian random variables, we have
	\begin{align*}
	\mathbb E_{\theta'\sim\mu^2_{r}}[e^{2(\theta-\hat\theta)^\intercal\Sigma^{-1}(\theta'-\hat\theta)}]=e^{(2r \cdot (\theta-\hat\theta)^\intercal\Sigma^{-1}(\theta'-\hat\theta) )}=e^{2r\rho_n^2} \geq 1+2r\rho_n^2.
	\end{align*}
	Now it follows that
	\begin{equation*}
	\lim_{r\downarrow 0}\frac{1}{r}\Big(1-\mathbb E_{\theta'\sim\mu^1_{r}}[e^{2(\theta-\hat\theta)^\intercal\Sigma^{-1}(\theta'-\hat\theta)}]\Big) \leq -2\rho_n^2.
	\end{equation*}
\end{proof}

\begin{proof}[Proof of Corollary \ref{variability}]
Lemmas \ref{n1} and \ref{n2} combined with \eqref{dproof} indicate that increasing $r$ to positive value would produce a descent direction for $l(r)$ at $r=0$.
\end{proof}
\begin{proof}[Proof of Corollary \ref{cod2}]
We proceed the proof as in Proposition \ref{cod1}. The proof for the case of $\mu_r^1(d\theta)$ is entirely similar. For the proof of the case $\mu_r^2(d\theta)$, we simply notice that if $Y\sim \mathbb P_t$, then
\begin{equation*}
    Y {\buildrel \mathcal D \over =} (1-U_t)(\hat\theta+X_1)+U_t(\hat\theta+\sqrt{r} X_2+X_3),
\end{equation*}
where $U_t$ is an independent Bernoulli variable with success rate $t$ and $X_1,X_2,X_3$ are independent $\mathcal N(0,\Sigma)$. Then, it follows that
\begin{equation*}
    \Sigma^{-1/2}(Y-\hat\theta) {\buildrel \mathcal D \over =} (1-U_t) Z_1+ U_t(\sqrt{r} Z_2 +Z_3)
\end{equation*}
where $Z_1,Z_2,Z_3$ are now independent $\mathcal N(0,I_D)$. Consequently, the rotational invariance of $Z$ now follows from the rotational invariance of $Z_1,Z_2,Z_3$ and their independence.

\end{proof}
\end{document}